\documentclass[reqno,11pt]{amsart}
 
\usepackage{amsmath}
\usepackage{amsthm,amssymb,color,comment}
\usepackage{bbm}
\usepackage{mathrsfs}
\usepackage{hyperref}
\usepackage[TS1,T1]{fontenc}
\usepackage[utf8]{inputenc}
\usepackage{dsfont}
\usepackage{tikz}
\usepackage{enumitem}
\usepackage{subfigure}
\usepackage{mathtools}
\usepackage{bbold}
\usepackage{esint}
\usepackage[sort]{cite}

\numberwithin{equation}{section}

\newtheorem{thm}{Theorem}[section]

\newtheorem{proposition}[thm]{Proposition}
\newtheorem{lem}[thm]{Lemma}

\newtheorem{Def}[thm]{Definition}

\theoremstyle{definition}
\newtheorem{Ass}[thm]{Assumption}
\newtheorem{rem}[thm]{Remark}

\DeclareMathOperator*{\esssup}{ess\,sup}
\DeclareMathOperator*{\essinf}{ess\,inf}
\DeclareMathOperator{\DIV}{div}
\DeclareMathOperator{\BL}{BL}
\DeclareMathOperator{\Lip}{Lip}

\newcommand{\bvmu}{\mathbf{v^{\mu}}}
\newcommand{\bvnu}{\mathbf{v^{\nu}}}
\newcommand{\bn}{\mathbf{n}}

\newcommand{\R}{\mathbb{R}}
\newcommand{\Td}{\mathbb{T}^{d}}
\newcommand{\Rd}{\mathbb{R}^{d}}

\newcommand{\N}{\mathbb{N}}
\newcommand{\p}{\partial}

\newcommand{\eps}{\varepsilon}

\newcommand{\diff}{\mathop{}\!\mathrm{d}}

\makeatletter
\newcommand{\doublewidetilde}[1]{{%
  \mathpalette\double@widetilde{#1}%
}}
\newcommand{\double@widetilde}[2]{%
  \sbox\z@{$\m@th#1\widetilde{#2}$}%
  \ht\z@=.9\ht\z@
  \widetilde{\box\z@}%
}
\makeatother
\author{José A. Carrillo}
\address{{\it José A. Carrillo:} Mathematical Institute, University of Oxford, Woodstock Road, Oxford, OX2 6GG, United Kingdom}
\email{carrillo@maths.ox.ac.uk}

\author{Piotr Gwiazda}
\address{{\it Piotr Gwiazda:} Institute of Mathematics, Polish Academy of Sciences, ul. Śniadeckich 8, 00-656 Warsaw, Poland; Interdisciplinary Centre for Mathematical and Computational Modelling, University of Warsaw, ul. Tyniecka 15/17, 02-630 Warsaw, Poland}
\email{pgwiazda@mimuw.edu.pl}
\thanks{}

\author{Jakub Skrzeczkowski}
\address{{\it Jakub Skrzeczkowski: } {St John's College, University of Oxford, St Giles, Oxford, OX1 3JP, United Kingdom} \& Mathematical Institute, University of Oxford, Woodstock Road, Oxford, OX2 6GG, United Kingdom}
\email{jakub.skrzeczkowski@maths.ox.ac.uk}

\begin{document}

\title[]{A new formula for the Wasserstein distance between solutions to (nonlinear) continuity equations}

\begin{abstract}
Given two continuity equations with density-dependent velocities, we provide a new formula for the Wasserstein distance between the solutions in terms of the difference of velocities evaluated at the same density. The formula is particularly attractive to deduce quantitative estimates and rates of convergence for singular limits. We illustrate it using several examples. For the porous medium equation with exponent $m$, we prove that solutions are Lipschitz continuous with respect to $m$, providing a quantitative version of the result of B\'{e}nilan and Crandall. This result can be extended to a general aggregation-diffusion equation. We also study the limit $m \to \infty$ (the so-called mesa problem or the incompressible limit) and we recover the rate of convergence $1/{\sqrt{m}}$. Last but not least, we improve the rate of nonlocal-to-local convergence for the quadratic porous medium equation from recently obtained $\sqrt{\eps}$ to numerically conjectured $\eps$. 
\end{abstract}




\maketitle

\setcounter{tocdepth}{1}
\section{Introduction}

\subsection{Motivation.} We consider two continuity equations 
\begin{equation}\label{eq:PDE_continuity_equation_1}
\partial_t \mu + \DIV(\mu \, \bvmu[\mu]) = 0,
\end{equation}
\begin{equation}\label{eq:PDE_continuity_equation_2}
\partial_t \nu + \DIV(\nu \, \bvnu [\nu]) = 0.
\end{equation}
Here, $\mu, \nu: [0,T] \times \Omega \to \R$ are solutions to \eqref{eq:PDE_continuity_equation_1}--\eqref{eq:PDE_continuity_equation_2} on $[0,T]\times\Omega$ while $\bvmu, \bvnu$ are $\R^d$-valued vector fields. We use the usual notation $\mu_t$, $\nu_t$ to express the dependence of $\mu$ and $\nu$ on time $t$. The notation ${\bf v}^{\mu}[\mu]$ means that $\bvmu$ is allowed to depend on variables $t$, $x$ as well as on the solution $\mu$ (via its spatial derivatives of arbitrary order as well as nonlocal spatial operators such as convolutions). This allows to cover the cases such as
$$
{\bf v}[\mu] = -\frac{m}{m-1} \, \nabla (\mu)^{m-1}, \qquad \qquad  {\bf v}[\mu] = -\nabla \mu \ast \omega,
$$
corresponding to the porous medium equation \cite{MR2286292} and aggregation equation \cite{carrillo2019aggregation}, respectively. We will make more assumptions on $\bvmu, \bvnu$ in Assumption \ref{eq:assumption_simplified_formula} (together with \eqref{eq:particular_vector_fields_Wasserstein_gf}) and in Assumptions~\ref{ass:BIG_regularity_of_sln_and_vf}--\ref{ass:semigroup_second_map}. Nevertheless, the cases of porous medium and aggregation equations are covered by our theory.\\

The set $\Omega$ can be $\Rd$, $R\,\Td$ or a smooth bounded domain of $\Rd$ and in the latter case, \eqref{eq:PDE_continuity_equation_1}--\eqref{eq:PDE_continuity_equation_2} has to be equipped with appropriate Neumann boundary conditions
\begin{equation}\label{eq:boundary_cond_1_intro}
\mu \, \bvmu[\mu] \cdot \bn = 0 \mbox{ on } \partial \Omega,
\end{equation}
\begin{equation}\label{eq:boundary_cond_2_intro}
\nu \, \bvnu[\nu] \cdot \bn = 0 \mbox{ on } \partial \Omega,
\end{equation}
where $\bn$ is the outer normal vector to the boundary $\partial \Omega$. \\

The target of this paper is to provide a new formula for the Wasserstein distance $\mathcal{W}_p(\mu_t, \nu_t)$ in terms of the difference of the vector fields $\bvmu - \bvnu$. We recall that for two probability measures $\mu, \nu \in \mathcal{P}_p(\Omega)$ and $p \in [1,\infty)$, the Wasserstein distance $\mathcal{W}_p(\mu, \nu)$ is defined as 
$$
\mathcal{W}^p_p(\mu, \nu) = \inf\left\{ \int_{\Omega\times\Omega} |x-y|^p \diff \pi(x,y) \right\},
$$
where the infimum is taken over all probability measures $\pi \in \mathcal{P}(\Omega\times\Omega)$ such that $\pi(A\times\Omega) = \mu(A)$, $\pi(\Omega\times B) = \nu(B)$ and $\mathcal{P}_p(\Omega)$ is the space of probability measures $\mu$ on $\Omega$ such that $\int_{\Omega} |x|^p \diff \mu(x) < \infty$. Formulas for $\mathcal{W}_p(\mu_t, \nu_t)$ lead to stability estimates for the PDEs \eqref{eq:PDE_continuity_equation_1}--\eqref{eq:PDE_continuity_equation_2}, yielding rates of convergence for numerical methods and singular limits. Indeed, PDEs of the form \eqref{eq:PDE_continuity_equation_1}--\eqref{eq:PDE_continuity_equation_2} and their limits arise in several applications like mathematical description of tumor growth \cite{MR3695889, MR3162474, MR4188329, MR4324293} and numerical methods to simulate the diffusion process (so-called blob method) \cite{Ol,MR1821479,MR4858611, carrillo2024nonlocal,carrillo2023nonlocal, MR3913840,Hecht2023porous, burger2022porous}. The most famous Benamou-Brenier formula \cite{MR1738163} combines the optimal transport theory and transportation nature of \eqref{eq:PDE_continuity_equation_1}--\eqref{eq:PDE_continuity_equation_2}. Although its huge impact on the field, particularly on the numerical optimal transport, it does not seem to be directly useful when comparing solutions to two different PDEs. Another, more sophisticated approach is based on the formula on $\frac{\diff}{\diff t} \mathcal{W}_p^p(\mu_t, \nu_t)$ (see \cite[Theorem 5.24]{MR3409718}) which after appropriate manipulations has allowed to obtain rates of convergence for the incompressible limit of porous medium equations \cite{david2024improved} and nonlocal approximation of the quadratic porous medium equation \cite{amassad2025deterministic}. A slightly different idea to estimate $\mathcal{W}_2(\mu_t, \nu_t)$ is to apply the Evolutionary Variational Inequality which also proved to be useful for the aforementioned nonlocal approximation \cite{PME_1d_Brinkman_Darcy_rate}. We note that, beyond their applications in singular limits, stability estimates for the difference between two solutions also provide key convergence and regularity guarantees in the study of inverse problems for PDEs, whether through optimization-based methods \cite{MR4896516} or Bayesian approaches \cite{newwork_Zuza_Piotr_AMC}.\\

Here, we provide a new formula to estimate $\mathcal{W}_p(\mu_t, \nu_t)$ which can be applied directly to the PDEs and allows to obtain stability/continuity estimates only by means of elementary algebraic manipulations.   \\

\subsection{The main result for 2-Wasserstein gradient flows} 

We first present the result for \eqref{eq:PDE_continuity_equation_1}--\eqref{eq:PDE_continuity_equation_2} with $p=2$ and a particular form of velocity field $\bvnu[\nu]$
\begin{equation}\label{eq:particular_vector_fields_Wasserstein_gf}
\bvnu[\nu] = -\nabla\frac{\delta \mathcal{G}}{\delta \nu}[\nu],
\end{equation}
where $\mathcal{G}: \mathcal{P}_2(\Omega) \to (-\infty, \infty]$, $\frac{\delta \mathcal{G}}{\delta \nu}$ is its first variation and we write $\mathcal{D}(\mathcal{G})$ for its domain. This particular choice for $\bvnu[\nu]$ allows us to make no assumptions on the form of velocity field $\bvmu[\mu]$. We also recall that $\Omega$ is $\Rd$, its bounded smooth domain or a periodic domain $R\,\Td$.

\begin{Ass}\label{eq:assumption_simplified_formula}
Let $	T>0$. We assume that:
\begin{enumerate}[label=($\text{A}_{\arabic*}^{\text{Wass}}$)]
\item\label{ass:regularity_solutions:add_sect_alternative_proof} $\mu_t, \nu_t \in \mathcal{P}_2(\Omega)\cap L^1(\Omega)$ solve \eqref{eq:PDE_continuity_equation_1}--\eqref{eq:PDE_continuity_equation_2} with some velocity field $\bvmu[\mu]$ and a velocity field $\bvnu[\nu]$ of the form \eqref{eq:particular_vector_fields_Wasserstein_gf} and that the functions $\bvmu[\mu_t]\,\sqrt{\mu_t}$, $\nabla\frac{\delta \mathcal{G}}{\delta \nu}[\nu_t] \,\sqrt{\nu_t}$, $\nabla\frac{\delta \mathcal{G}}{\delta \nu}[\mu_t] \,\sqrt{\mu_t}$ belong to $L^2((0,T)\times\Omega)$ while $\mu_t, \nu_t \in \mathcal{D}(\mathcal{G})$ for a.e. $t\in [0,T]$,
\item\label{eq:geodesic_conv_G_sect:alternative_proof} $\mathcal{G}: \mathcal{P}_2(\Omega) \to (-\infty, \infty]$ is $\lambda$-geodesically convex in $(\mathcal{P}_2(\Omega), \mathcal{W}_2)$ for some $\lambda \in \R$, i.e. for any $\mu, \nu \in \mathcal{P}_2(\Omega)$ and the 2-Wasserstein geodesic $\{\gamma_s\}_{s\in[0,1]}$ connecting $\mu$ and $\nu$ we have 
$$
\mathcal{G}[\gamma_s] \leq s\, \mathcal{G}[\nu]+(1-s)\, \mathcal{G}[\mu] - \frac{s(1-s)\lambda}{2} \mathcal{W}_2^2(\mu, \nu),
$$
\item\label{eq:derivative_via_1st_var_sect:alternative_proof} for a.e. $t \in (0,T)$, the 2-Wasserstein geodesic $\{\gamma_s\}_{s\in[0,1]}$ connecting $\mu_t$ and $\nu_t$ satisfies for $s = 0,1$ 
\begin{equation}\label{eq:slope_G_geodesic_general_assumption_1}
\frac{\diff}{\diff s} \mathcal{G}[\gamma_s] \Big|_{s=0} \geq -\int_{\Omega} \nabla \frac{\delta \mathcal{G}}{\delta \nu}[\mu_t]\, \nabla \varphi(x) \diff \mu_t(x),
\end{equation}
\begin{equation}\label{eq:slope_G_geodesic_general_assumption_2}
\frac{\diff}{\diff s} \mathcal{G}[\gamma_s] \Big|_{s=1} \leq \int_{\Omega} \nabla \frac{\delta \mathcal{G}}{\delta \nu}[\nu_t]\, \nabla\phi(x) \diff \nu_t(x),
\end{equation}
where $\varphi$, $\phi$ are Kantorovich potentials for the optimal transport problem with a~quadratic cost $c(x,y)=\frac{1}{2}|x-y|^2$ between $\mu$, $\nu$ and $\nu$, $\mu$, respectively (see below). Here, derivatives in \eqref{eq:slope_G_geodesic_general_assumption_1}--\eqref{eq:slope_G_geodesic_general_assumption_2} are defined via
$$
\frac{\diff}{\diff s} \mathcal{G}[\gamma_s] \Big|_{s=0} = \lim_{h\to 0} \frac{\mathcal{G}[\gamma_h]-\mathcal{G}[\gamma_0]}{h}, \quad \frac{\diff}{\diff s} \mathcal{G}[\gamma_s] \Big|_{s=1} = \lim_{h\to 0} \frac{\mathcal{G}[\gamma_1]-\mathcal{G}[\gamma_{1-h}]}{h}.
$$
These limits exist by $\lambda\, \mathcal{W}_2^2(\mu,\nu)$-convexity of the function $s\mapsto \mathcal{G}[\gamma_s]$ although they may be equal $-\infty$ for $\frac{\diff}{\diff s} \mathcal{G}[\gamma_s] \Big|_{s=0}$ and $+\infty$ for $\frac{\diff}{\diff s} \mathcal{G}[\gamma_s] \Big|_{s=1}$ even if $\gamma_s \in \mathcal{D}(\mathcal{G})$ for all $s\in[0,1]$ (think about $[0,1]\ni x \mapsto -\sqrt{x\,(1-x)}$).
\end{enumerate}
\end{Ass} 

On $\Rd$ (or its bounded domain), the Kantorovich potentials for the cost $c(x,y) = \frac{1}{2}|x-y|^2$ satisfy $\nabla \varphi(x) = x-T(x)$, $\nabla \phi(x) = x-S(x)$ where $T$ is the optimal transport moving $\mu$ onto $\nu$ and $S$ is its inverse, see \cite[Prop.~1.15]{MR3409718}. On the periodic domain, the same is true but $T$ and $S$ are the maps such that their projection on $R\,\Td$ is the optimal transport, see \eqref{eq:Kantorovich_potential_periodic_domain} in Appendix \ref{rem:optimal_transport_periodic_domain}.  \\
  
We remark that condition \ref{eq:derivative_via_1st_var_sect:alternative_proof} can be expected to be satisfied in most cases of interest. Indeed, the geodesic $\gamma_s$ solves $\partial_s \gamma_s + \DIV(\gamma_s {\bf V}_s) = 0$ where ${\bf V}_0 = -\nabla \varphi$, ${\bf V}_1=\nabla \phi$. Therefore, for regular functionals $\mathcal{G}$ (see \cite[Definition 7.12]{MR3409718}) we can compute informally
$$
\frac{\diff}{\diff s} \mathcal{G}[\gamma_s] = \int_{\Omega} \frac{\delta \mathcal{G}}{\delta \nu}[\gamma_s]\, \partial_s \gamma_s \diff x = \int_{\Omega} \nabla \frac{\delta \mathcal{G}}{\delta \nu}[\gamma_s]\, {\bf V}_s \gamma_s \diff x
$$
and we arrive at \ref{eq:derivative_via_1st_var_sect:alternative_proof} by plugging $s=0$ and $s=1$. Furthermore, the condition $\nu_t \in \mathcal{D}(\mathcal{G})$ is usually satisfied for most reasonable functionals. For instance, it holds if $\mathcal{G}$ is lower semicontinuous with respect to $\mathcal{W}_2$ and bounded from below because the minimizing movement scheme is then known to converge \cite[Theorem 4.25]{MR3050280} and the lower semicontinuity gives $\mathcal{G}[\nu_t]<\infty$ for all $t$.\\

We also remark that by inspection of the proof of Theorem \ref{thm:simplified_for_2-Wasserstein}, one can relax condition \ref{eq:geodesic_conv_G_sect:alternative_proof} to require the geodesic convexity of $\mathcal{G}$ only on some (geodesically convex) subset $\mathcal{X} \subset \mathcal{P}_2(\Omega)$ if solutions $\mu_t, \nu_t \in \mathcal{X}$ for all $t \in [0,T]$. Last but not least, we note that instead of assumptions \ref{eq:geodesic_conv_G_sect:alternative_proof} and \ref{eq:derivative_via_1st_var_sect:alternative_proof}, one could require the inequality 
\begin{equation}\label{eq:condition_replacing_A2+A_3_by_Fabian}
 -\int_{\Omega} \nabla \frac{\delta \mathcal{G}}{\delta \nu}[\mu_t]\, \nabla \varphi(x) \diff \mu_t(x) + \lambda \, \mathcal{W}_2^2(\mu_t,\nu_t) \leq \int_{\Omega} \nabla \frac{\delta \mathcal{G}}{\delta \nu}[\nu_t]\, \nabla\phi(x) \diff \nu_t(x),
\end{equation}
with the notation as in \ref{eq:derivative_via_1st_var_sect:alternative_proof}, since this is the only inequality used in the proof of Theorem~\ref{thm:simplified_for_2-Wasserstein} (see \eqref{eq:additional_claim_geodesic_convexity}). However, we are not aware of a direct way to establish \eqref{eq:condition_replacing_A2+A_3_by_Fabian} without appealing to geodesic convexity and estimating the slopes at $s=0$ and $s=1$. For this reason, we prefer to retain both assumptions \ref{eq:geodesic_conv_G_sect:alternative_proof} and \ref{eq:derivative_via_1st_var_sect:alternative_proof}.

 \begin{thm}\label{thm:simplified_for_2-Wasserstein}
Let $\Omega$ be $\Rd$, its bounded smooth domain or $R\,\Td$. Let $\mu_t, \nu_t \subset \mathcal{P}_2(\Omega)\cap L^1(\Omega)$ be solutions to \eqref{eq:PDE_continuity_equation_1}--\eqref{eq:PDE_continuity_equation_2} with velocity field $\bvnu[\nu]$ given by \eqref{eq:particular_vector_fields_Wasserstein_gf} and with initial conditions $\mu_0$, $\nu_0$. We assume that Assumption \ref{eq:assumption_simplified_formula} holds true. Then, for a.e. $t\in (0,T)$,
 \begin{equation}\label{eq:claim_formula_for_Wass_gradient_flows}
\frac{\diff}{\diff t} \mathcal{W}_2(\mu_t, \nu_t) + \lambda \, \mathcal{W}_2(\mu_t,\nu_t) \leq \left( \int_{\Omega} \left| \nabla\frac{\delta \mathcal{G}}{\delta \nu}[\mu_t] + \bvmu[\mu_t] \right|^2 \diff \mu_t(x) \right)^{\frac{1}{2}}. 
\end{equation}
 \end{thm}

The main advantage of Theorem \ref{thm:simplified_for_2-Wasserstein} is that both velocity fields are evaluated at the same solution $\mu_t$. This is reminiscent of the relative entropy method \cite{MR3615546}, where the distance between two solutions is estimated by a functional that depends only on one of them, typically the one with higher regularity. Another advantage, comparing to previous formulas, is that we estimate $\mathcal{W}_2(\mu_t,\nu_t)$ rather than $\mathcal{W}_2^2(\mu_t,\nu_t)$. This has direct consequences for example in Theorem \ref{thm:rate_of_conv_nonlocal_to_local} where we improve the known rate of $\sqrt{\eps}$ to $\eps$. \\

We remark that this formula for aggregation-diffusion equations (with fixed nonlinearity) has been obtained in \cite[Appendix A.2]{MR4896516}. Moreover, a similar formula appeared in \cite[Lemma 5.2]{MR3864212} under the assumption that one of the velocities, say $\nabla\frac{\delta \mathcal{G}}{\delta \nu}[\mu_t]$, belongs to $L^{\infty}(0,T; W^{1,\infty}(\Td))$. However, this norm also appears in the final estimate, making the result inapplicable, for instance, to the porous medium equation. Finally, in \cite{MR3522009}, the Authors obtain stability estimates for the particular case of the nonlinear Fokker-Planck equations in the total variation norm and in the Wasserstein distance with respect to both the diffusion matrix and the velocity (see also \cite[Prop.~3.1]{MR2672640} and \cite[Theorem~2.18]{MR3846438} for results in a similar direction). \\

The proof of Theorem \ref{thm:simplified_for_2-Wasserstein} is based on the aforementioned formula for $\frac{\diff}{\diff t} \mathcal{W}_2^2(\mu_t, \nu_t)$ (see \cite[Theorem 5.24]{MR3409718}). However, thanks to \ref{eq:derivative_via_1st_var_sect:alternative_proof} and the Cauchy-Schwarz inequality, we are able to estimate $\mathcal{W}_2(\mu_t, \nu_t)$ rather than $\mathcal{W}_2^2(\mu_t, \nu_t)$.

\subsection{The main result for general nonlinear continuity equations.} Surprisingly, we can also prove a formula like \eqref{eq:claim_formula_for_Wass_gradient_flows} for general nonlinear continuity equations \eqref{eq:PDE_continuity_equation_1}--\eqref{eq:PDE_continuity_equation_2}, without imposing the Wasserstein gradient flow structure. Let us first present the assumptions. The first one is of technical nature and roughly speaking, it says that \eqref{eq:PDE_continuity_equation_1}--\eqref{eq:PDE_continuity_equation_2} have smooth solutions for initial conditions forming a large, possibly dense set of $\mathcal{P}_p(\Omega)$. 

\begin{Ass}\label{ass:BIG_regularity_of_sln_and_vf}
We assume that there exist two sets $\mathcal{A}_1, \mathcal{A}_2 \subset \mathcal{P}_p(\Omega)\cap L^1(\Omega)$ such that $\mathcal{A}_1 \subset \mathcal{A}_2$ and if $\mu_0\in\mathcal{A}_1$, $\nu_0 \in \mathcal{A}_2$ then \eqref{eq:PDE_continuity_equation_1} and \eqref{eq:PDE_continuity_equation_2} admit unique (weak) solutions $\mu_t \in \mathcal{A}_1$, $\nu_t \in \mathcal{A}_2$ with initial conditions~$\mu_0$ and $\nu_0$, respectively. Moreover, if $\mu_0 = \nu_0$, the solutions $\mu_t$, $\nu_t$ satisfy 
\begin{enumerate}[label=($\text{A}_{\arabic*}^{\text{gen}}$)]
\item\label{ass:weak_continuity} the maps $t \mapsto \mu_t, \nu_t$ are narrowly continuous,
\item\label{ass:regularity_vector_fields} there exists $h>0$ and $\delta>0$ such that the maps $(t,x)\mapsto \bvmu[\mu_t], \bvnu[\nu_t]$ belong to $L^{2+\delta}(0,h; \BL(\Omega))$, where $\BL(\Omega)$ is the space of bounded Lipschitz functions on $\Omega$ equipped with the norm
\begin{equation}\label{eq:norm_BL}
\|f \|_{\BL(\Omega)} = \|f\|_{L^{\infty}(\Omega)} + \Lip(f), \qquad \Lip(f):= \sup_{x,y\in \Omega,\, x\neq y} \frac{|f(x)-f(y)|}{|x-y|},
\end{equation}
\item\label{ass:C_admissibility} if $\rho \in \mathcal{A}_1$ then $\left|\bvmu[\rho]- \bvnu[\rho]  \right|^p \in C(\Omega) \cap L^{\infty}(\Omega)$,
\item\label{ass:continuity_difference} there exists $h>0$ such that $(t,x)\mapsto \bvnu[\nu_t] - \bvmu[\mu_t]$ is in $L^{\infty}((0,h)\times \Omega)$ and for each compact subset $K \subset \overline{\Omega}$ we have
$$
\esssup_{0\leq \tau \leq h} \Big\| \big|\bvnu[\nu_\tau ] - \bvmu[\mu_\tau]\big|^p - \big|\bvnu[\mu_0] - \bvmu[\mu_0]\big|^p \Big\|_{L^{\infty}(K)} \to 0  \mbox{ as } h \to 0.
$$
\item\label{ass:pointwise_boundary_cond} if $\Omega$ is a bounded domain, the boundary conditions $\bvmu[\mu_t] \cdot \bn = 0$, $\bvnu[\nu_t] \cdot \bn \leq 0$ on $\partial \Omega$ hold pointwisely for a.e. $t \in [0,h]$ for some small $h>0$.
\end{enumerate}
\end{Ass}

\begin{Ass}\label{ass:abs_cont_first_map}
We assume that for any initial conditions $\mu_0 \in \mathcal{A}_1$, $\nu_0 \in \mathcal{A}_2$, the maps $t~\mapsto~\mu_t$, $t\mapsto \nu_t$ solving \eqref{eq:PDE_continuity_equation_1}--\eqref{eq:PDE_continuity_equation_2} are absolutely continuous in $\mathcal{W}_p$, i.e. for all $\mu_0~\in~\mathcal{A}_1, \nu_0~\in~\mathcal{A}_2$, there exists a function $h: [0,\infty) \to \R$, $h \in L^1_{\text{loc}}([0,\infty))$, $h \geq 0$ such that for all $s, t \in [0,T]$, $s<t$ we have
$$
\mathcal{W}_p(\mu_t, \mu_s), \mathcal{W}_p(\nu_t, \nu_s) \leq \int_s^t h(u) \diff u.
$$
\end{Ass}

Our main additional assumption is to require that the PDE \eqref{eq:PDE_continuity_equation_2} has a semigroup structure with sufficient continuity properties. To this end, we define the map $\mathcal{S}_{t}: \mathcal{P}(\Omega)\to\mathcal{P}(\Omega)$ such that $\mathcal{S}_{t} \rho_0$ is the solution of \eqref{eq:PDE_continuity_equation_2} at time $t$ with condition $\rho_0$ at time $t=0$. The semigroup property reads $\mathcal{S}_{t+s} = \mathcal{S}_{t}  \circ  \mathcal{S}_{s}$.

\begin{Ass}\label{ass:semigroup_second_map}
We assume that \eqref{eq:PDE_continuity_equation_1} and \eqref{eq:PDE_continuity_equation_2} generates a semigroup (in the sense of Definition \ref{def:semigroup_autonomous}) on $(\mathcal{A}_1, \mathcal{W}_p)$ and $(\mathcal{A}_2, \mathcal{W}_p)$ respectively. Moreover, for \eqref{eq:PDE_continuity_equation_2}, the semigroup is a Lipschitz semigroup (in the sense of Definition \ref{def:semigroup_Lipschitz-AC-autonomous}) on the space $(\mathcal{A}_2, \mathcal{W}_p)$. More precisely, for each $t$, there exists a~constant $K(t)>0$ such that the semigroup $\mathcal{S}_t$ generated by \eqref{eq:PDE_continuity_equation_2} satisfies
\begin{equation}\label{eq:Wasserstein_Lipschitz-AC}
\mathcal{W}_p(\mathcal{S}_{t} \nu^1_0, \mathcal{S}_{t} \nu^2_0) = \mathcal{W}_p( \nu^1_t, \nu^2_t) \leq K(t) \, \mathcal{W}_p(\nu^1_0,\nu^2_0).
\end{equation}
Moreover, we assume that $t\mapsto K(t)$ is continuous.
\end{Ass}

The second main result of this work reads:

\begin{thm}\label{thm:main}
Let $\Omega$ be $\Rd$, its bounded smooth domain or $R\,\Td$. Suppose that Assumptions~\ref{ass:BIG_regularity_of_sln_and_vf}, \ref{ass:abs_cont_first_map} and \ref{ass:semigroup_second_map} are satisfied for $p\geq 1$ and let $\mathcal{A}_1$ be the set as in Assumption \ref{ass:BIG_regularity_of_sln_and_vf}. Let $\mu_t, \nu_t$ be solutions to \eqref{eq:PDE_continuity_equation_1} and \eqref{eq:PDE_continuity_equation_2} with the same initial condition $\mu_0=\nu_0 \in \mathcal{A}_1$ and boundary conditions \eqref{eq:boundary_cond_1_intro}--\eqref{eq:boundary_cond_2_intro} if $\Omega$ is a bounded domain. Then,
$$
\mathcal{W}_p(\mu_t, \nu_t) \leq  \int_0^t K(t-s)\, \left( \int_{\Omega} \left|\bvnu[\mu_s]- \bvmu[\mu_s]  \right|^p \diff \mu_s  \right)^{1/p} \diff s,
$$
where $K$ is the constant from \eqref{eq:Wasserstein_Lipschitz-AC}.
\end{thm}

We should stress that Assumption~\ref{ass:BIG_regularity_of_sln_and_vf} is of technical nature. If solutions to \eqref{eq:PDE_continuity_equation_1}--\eqref{eq:PDE_continuity_equation_2} can be approximated by smooth ones, then the final conclusion of Theorem \ref{thm:main} holds for larger class of solutions than the ones belonging to $\mathcal{A}_1$, $\mathcal{A}_2$ since the constant $K$ does not depend on the regularity in Assumption~\ref{ass:BIG_regularity_of_sln_and_vf}.\\

Although Theorem \ref{thm:main} looks more complicated, it has one clear advantage over Theorem~\ref{thm:simplified_for_2-Wasserstein}. The latter involves estimating $\frac{\diff}{\diff s} \mathcal{G}[\gamma_s] \Big|_{s=0}$, which usually requires some regularity and can be sometimes quite hard to check. In fact, already in our first application - the proof of Theorem \ref{thm:PME} in Section~\ref{sect:applications_to_PME} - we will see that checking the assumptions of Theorem \ref{thm:main} can sometimes be easier than those of Theorem \ref{thm:simplified_for_2-Wasserstein}. On the other hand, Theorem \ref{thm:simplified_for_2-Wasserstein} applies to a priori mildly regular solutions, so it’s useful when there is no convenient approximation scheme, as in the case of the aggregation-diffusion system in Section \ref{subsect:general_A_D_intro}.\\

Comparing the settings of Theorem \ref{thm:simplified_for_2-Wasserstein} and Theorem \ref{thm:main}, condition \ref{ass:regularity_solutions:add_sect_alternative_proof} on regularity of solutions corresponds to Assumption \ref{ass:BIG_regularity_of_sln_and_vf} (although the latter requires much more regularity) while Assumption \ref{ass:semigroup_second_map} corresponds to the geodesic convexity condition \ref{eq:geodesic_conv_G_sect:alternative_proof}.\\

Regarding the method of the proof, the Lipschitz continuity of the semigroup $\mathcal{S}_t$ that we require allows to use a formula from the semigroup theory (see Lemma~\ref{lem:bressan_estimate_between_map_semigroup_aut}) which combined with the Benamou-Brenier formula yields our main result (Theorem~\ref{thm:main}).

\begin{rem}\label{rem:velocity_field_depend_on_time}
Theorem \ref{thm:main} does not work for velocity fields that depend on time explicitly because of Assumption \ref{ass:abs_cont_first_map}, where the semigroup is required to be autonomous in time. Nevertheless, one can probably extend the result to this case since the main tool from the theory of semigroups that we use in the proof, Lemma \ref{lem:bressan_estimate_between_map_semigroup_aut}, can be extended to nonautonomous semigroups, see \cite[Proposition I.9]{MR4309603}. For the sake of simplicity, we decided to concentrate on the autonomous case. 
\end{rem}

\subsection{Applications to the porous medium equation.}\label{subsect:PME_introduction} Theorems \ref{thm:simplified_for_2-Wasserstein} and \ref{thm:main} can be applied to easily deduce stability estimates for numerous PDEs. Here, to illustrate, we consider two porous medium equations of the form
\begin{equation}\label{eq:PME_1}
\partial_t \mu = \Delta \mu^m,
\end{equation}
\begin{equation}\label{eq:PME_2}
\partial_t \nu = \Delta \nu^{n}.
\end{equation}
If considered on a bounded domain, they are equipped with boundary conditions:
\begin{equation}\label{eq:Neuman_boundary_conditions_PME}
\mu\, \nabla \mu^{m-1} \cdot \bn = \nu \, \nabla \nu^{n-1} \cdot \bn = 0 \mbox{ on } \partial \Omega.
\end{equation}
We will consider also more general variant of \eqref{eq:PME_1}--\eqref{eq:PME_2} with advection and nonlocal interactions in Subsection \ref{subsect:general_A_D_intro} but for the sake of simplicity, we want to discuss the simpler case first. \\

\noindent The continuity of \eqref{eq:PME_1}, \eqref{eq:PME_2} with respect to the exponents $n$, $m$ was first established in \cite{MR604277}. In \cite[Theorem 3.2]{MR772052} the first quantitative result on H{\"o}lder continuity with respect to the exponent was given which was further extended to Lipschitz continuity in \cite{MR1669570, MR2187640,dkebiec2025lipschitz}. Here, we provide a quantitative result in the 2-Wasserstein distance. 

\begin{thm}\label{thm:PME}
Let $\mu_0 \in \mathcal{P}_2(\Omega) \cap L^{\infty}(\Omega)$ where $\Omega$ is $\Rd$, a bounded smooth domain or $R\,\Td$. Let $\mu_t$, $\nu_t$ be solutions to \eqref{eq:PME_1}--\eqref{eq:PME_2} with Neumann boundary conditions \eqref{eq:Neuman_boundary_conditions_PME} (if $\Omega$ is a bounded domain) and the same initial condition $\mu_0 = \nu_0$. Suppose that 
$$
1<m_- \leq m \leq m_+ < \infty,  \qquad 1<n_- \leq n \leq n_+ < \infty.
$$
Then, there exists a constant $C = C(m_-, m_+, n_-, n_+, \|\mu_0\|_{L^{\infty}(\Omega)})$ such that
\begin{equation}\label{eq:continuity_wrt_exponent_W2_PME_main_thm}
\mathcal{W}_2(\mu_t, \nu_t) \leq C\, \sqrt{t}\, |m-n|.
\end{equation}
\end{thm}

The constant $C$ can be expressed explicitly. Assuming $m<n$, we introduce parameters $\alpha \in [0,m-1)$ and $\beta \in [0, 2n-m-1)$. Then, $C$ is given by
\begin{equation}\label{eq:constant_continuity_wrt_exponent_W2_PME_main_thm}
\begin{split}
C = \, & \frac{1}{\sqrt{m\,(2n-m-1)\,(2n-m)}} \left( \int_{\Omega} \mu_0^{2n-m} \diff x \right)^{\frac{1}{2}}\\
 &+\frac{m}{\sqrt{m\,(m-1-\alpha)\,(m-\alpha)}} \, C_{\frac{\alpha}{2}}\, \left(\int_{\Omega} \mu_0^{m-\alpha} \diff x \right)^{\frac{1}{2}}\\
 &+\frac{m}{\sqrt{m\,(2n-m-1-\beta)\,(2n-m-\beta)}} \,  C_{\frac{\beta}{2}}\,  \left(\int_{\Omega} \mu_0^{2n-m-\beta} \diff x \right)^{\frac{1}{2}},
\end{split}
\end{equation}
where 
\begin{equation}\label{eq:constant_C_kappa_continuity_PME}
C_{\kappa} = \min\left(\frac{1}{e\,\kappa} + \|\mu_0\|_{L^{\infty}(\Omega)}^{\kappa+1}, \,  \|\mu_0\|_{L^{\infty}(\Omega)}^{\kappa}\,\|\log \mu_0\|_{L^{\infty}(\Omega)} \right), \qquad \kappa = \frac{\alpha}{2},\,  \frac{\beta}{2}.
\end{equation}
Observe that the constant $C_{\frac{\alpha}{2}}$, $C_{\frac{\beta}{2}}$ in \eqref{eq:constant_C_kappa_continuity_PME} converge to $\|\log\mu_0\|_{L^{\infty}(\Omega)}$ when $\alpha, \beta \to 0$, i.e. when $m \to 1$. Therefore, when $m\to 1$, the constant blows up unless we work on a bounded domain and $\mu_0$ is bounded from below. Hence, the presented method works only for $1<m,n$; we refer to Remark~\ref{rem:extending_below_1_the_method_for_PME} for a discussion of the main difficulties in extending the range of exponents below 1. \\

\noindent We provide two proofs of Theorem \ref{thm:PME} in Section \ref{sect:applications_to_PME}, based on either Theorem \ref{thm:simplified_for_2-Wasserstein} or \ref{thm:main}. \\

\noindent One may ask whether the precise formula \eqref{eq:constant_continuity_wrt_exponent_W2_PME_main_thm} can be used to deduce a rate of convergence of the porous medium equation to the heat equation. If $\Omega$ is a bounded domain and $\mu_0 \geq \varepsilon > 0$, consider the solutions $\mu_{m_k,t}$, $\nu_{n_k,t}$ to \eqref{eq:PME_1}--\eqref{eq:PME_2} with exponents $m_k = 1+\frac{n-1}{2^{k+1}}$, $n_k = 1 +\frac{n-1}{2^{k}}$ where $k \in \mathbb{N}$ and $n_0 = n$ is fixed. Then, Theorem~\ref{thm:PME} yields the estimate $\mathcal{W}_2(\nu_{n,t}, \mu_{1,t}) \leq \sum_{k=0}^{\infty} \mathcal{W}_2(\nu_{n_k,t}, \mu_{n_k,t}) \leq C\, \sqrt{n-1}$ for a constant $C$ depending on $t$, $\|\mu_0\|_{L^{\infty}(\Omega)}$, and linearly on $\|\log\mu_0\|_{L^{\infty}(\Omega)}$. Next, by considering a sequence of initial conditions $\mu_0^{\eps} = (1-\eps) \mu_0 + \frac{\eps}{|\Omega|}$ and choosing $\eps \approx |n-1|$, one can obtain $\mathcal{W}_2(\nu_{n,t}, \mu_{1,t}) \leq C\, \sqrt{n-1}\, |\log(n-1)|$ for a general initial condition $\mu_0$, not necessarily strictly positive. This argument, however, is somewhat technical. In Appendix~\ref{app:rate_PME_nto1} we present an alternative approach based on the Evolutionary Variational Inequalities, which yields the rate $\sqrt{n-1}$ in a straightforward manner. On the other hand, it is easy to see that if $\mu, \nu$ solve \eqref{eq:PME_1}--\eqref{eq:PME_2}, the method in Appendix \ref{app:rate_PME_nto1} yields $\mathcal{W}_2(\mu_t, \nu_t)\leq C\sqrt{|m-n|}$ which is weaker than the bound provided by Theorem \ref{thm:PME}, thereby illustrating the applicability of Theorems~\ref{thm:simplified_for_2-Wasserstein} and \ref{thm:main}. \\

\noindent Theorems \ref{thm:simplified_for_2-Wasserstein} and \ref{thm:main} can be also used to prove rates of convergence for singular limits. To illustrate, we discuss here the limit $m \to \infty$ in \eqref{eq:PME_1}, commonly called {\it the mesa problem} \cite{MR2028110, MR1155452, MR916741, MR983523} or {\it the incompressible limit} \cite{MR4324293,MR3162474,MR3729490}. Several authors studied the rate of convergence as $m \to \infty$. The first rate in the Wasserstein distance of $\frac{1}{m^{1/24}}$ was obtained in \cite{MR3190322} which was improved recently to $\frac{1}{\sqrt{m}}$ in \cite{david2024improved}. The latter work exploits the formula for the time derivative of $\mathcal{W}_2$ in terms of Kantorovich potentials which are further estimated by using the Monge-Ampère equation. Here, we obtain the same result by directly applying Theorems \ref{thm:simplified_for_2-Wasserstein} or \ref{thm:main}. We also point out that the rate of $\frac{1}{\sqrt{m}}$ was obtained also in \cite{MR4604155} in the $H^{-1}$ distance. In this work, the Authors could include also the source terms since the exploited metric does not require conservation of mass. Finally, we point out that the results of \cite{david2024improved} are valid also for a general aggregation-diffusion equation. 

\begin{thm}\label{thm:PME_incompress_limit}
Let $\Omega$ be $\Rd$, a bounded smooth domain or $R\,\Td$. Let $\mu_{m,t}$ be the solution to \eqref{eq:PME_1} with Neumann boundary conditions (if $\Omega$ is a bounded domain) and the initial condition $\mu_{m,0} \in \mathcal{P}_2(\Omega) \cap L^{\infty}(\Omega)$ satisfying
\begin{equation}\label{ass:bound_integral_incompressible_limit}
\mathcal{C} := \limsup_{m \to \infty} \int_{\Omega} \mu_{m,0}^{2m} \diff x  < \infty, \qquad \lim_{m\to \infty} \mathcal{W}_2(\mu_{m,0}, \mu_{\infty,0}) = 0,
\end{equation}
for some $\mu_{\infty,0} \in \mathcal{P}_2(\Omega) \cap L^{\infty}(\Omega)$. Then, $\{\mu_{m,t}\}_m$ converges in $\mathcal{W}_2$ when $m\to \infty$ uniformly on compact intervals of time to the limit $\mu_{\infty,t}$. Moreover, for $m \geq 2$ we have 
$$
\mathcal{W}_2(\mu_{m,t}, \mu_{\infty,t}) \leq \mathcal{W}_2(\mu_{m,0}, \mu_{\infty,0}) + \frac{\sqrt{t}}{\sqrt{m-1}} + \frac{\sqrt{2 \, t}}{\sqrt{m}} \, \mathcal{C}^{1/2}.
$$
\end{thm}

The first condition in \eqref{ass:bound_integral_incompressible_limit} is standard in the context of incompressible limit. A~typical condition requires that $\mu_{m,0}$ is compactly supported and $\mu_{m,0} \leq C^{1/m}$ for some constant $C>0$ \cite{MR4604155}. For such initial data, the first condition in \eqref{ass:bound_integral_incompressible_limit} is satisfied.\\

The classical results about the mesa problem assert that if $\mu_0$ is fixed and independent of $m$, two situations may occur: if $\mu_0(x) \leq 1$ for all $x$ then $\mu_{m,t}(x) \to \mu_0(x)$ for all $x$ while if $\mu_0(x) >1$ for some $x$, $\mu_{m,t}$ collapses and forms a plateau at height 1 \cite{MR3391972}. Our assumption \eqref{ass:bound_integral_incompressible_limit} leads to the first scenario and we can prove that $\mu_{\infty,t} = \mu_{\infty,0}$ for all $t\geq0$. Indeed, letting $p_{m,t} = \mu_{m,t}^{m-1}$, writing the PDE as $\partial_t \mu_{m,t} = \frac{m}{m-1} \DIV(\mu_{m,t}\, \nabla p_{m,t})$ and multiplying it by $\mu_{m,t}^{m-2}$ we obtain
$$
\int_{\Omega} p_{m,t}(x) \diff x + \frac{m(m-2)}{m-1} \int_0^t \int_{\Omega} |\nabla p_{m,s}|^2 \diff x \diff s \leq \int_{\Omega} p_{m,0}(x) \diff x.
$$
By assumption \eqref{ass:bound_integral_incompressible_limit} and by interpolating the $L^{m-1}(\Omega)$ norm between $L^1(\Omega)$ and $L^{2m}(\Omega)$, $\int_{\Omega} p_{m,0}(x) \diff x$ is uniformly bounded in $m$ so that $\nabla p_{m,t} \to 0$ strongly in $L^2((0,\infty)\times \Omega)$ as $m\to\infty$. Then, by the Benamou-Brenier formula in Lemma \ref{lem:B-B} and the maximum principle for solutions to \eqref{eq:PME_1} we have for $s<t$
$$
\mathcal{W}^2_2(\mu_{m,t}, \mu_{m,s}) \leq |t-s|\, \int_s^t \int_{\Omega} |\nabla p_{m,u}|^2 \, \mu_{m,u} \diff x \diff u \leq |t-s|\, \|\nabla p_{m,u}\|^2_{L^2((0,\infty)\times\Omega)}\, \|\mu_{m,0}\|_{L^{\infty}(\Omega)}. 
$$
By \eqref{ass:bound_integral_incompressible_limit}, $\|\mu_{m,0}\|_{L^{\infty}(\Omega)}\to 1$ so that $\mathcal{W}^2_2(\mu_{\infty,t}, \mu_{\infty,s}) = 0 $ and $\mu_{\infty,t} = \mu_{\infty,0}$.

\subsection{Nonlocal-to-local limit}\label{subsect:intro_nonlocal2local_limit}

Here, we present how to apply Theorem \ref{thm:simplified_for_2-Wasserstein} to obtain a rate of convergence for the one dimensional problem
\begin{equation}\label{eq:PDE_nonlocal_1D_eps}
\partial_t \mu^{\eps} = \partial_x(\mu^{\eps}\, \partial_x \mu^{\eps} \ast\omega_{\eps}).
\end{equation}
Here, $\mu^{\eps}:[0,T]\times\R \to \R^+$, $\mu^{\eps} \ast\omega_{\eps} = \int_{\R} \mu^{\eps}(y) \, \omega_{\eps}(x-y) \diff y$ is the convolution operator with the kernel $\omega_{\eps}(x) = \frac{1}{\eps} \omega\left(\frac{x}{\eps}\right)$ where $\omega$ satisfies Assumption \ref{ass:kernel} below. One popular choice is $\omega$ solving the PDE
\begin{equation}\label{eq:elliptic_PDE_omega}
-\p^2_{x} \omega + \omega = \delta_{0},
\end{equation}
i.e. $\omega(x)=\frac{1}{2} e^{-|x|}$. This particular kernel was studied in several works \cite{MR4880211,MR4712820,MR4146915,PME_1d_Brinkman_Darcy_rate} due to its connection to tissue growth models as well as due to the fact that it resembles the difficulties of the Newtonian potential without integrability issues. Formally, passing to the limit $\eps \to 0$ in \eqref{eq:PDE_nonlocal_1D_eps} one obtains
\begin{equation}\label{eq:PDE_local_1D_eps}
\partial_t \mu = \partial_x(\mu \, \partial_x \mu).
\end{equation}
The rigorous passage to the limit of \eqref{eq:PDE_nonlocal_1D_eps} towards \eqref{eq:PDE_local_1D_eps} has been studied in \cite{MR1821479,Ol,carrillo2023nonlocal,MR3913840,MR4858611,carrillo2024nonlocal,MR4880211,MR4712820,david2025Brinkman2Darcy,di2024approximation} also in more complicated contexts of systems and nonlinear equations with more general pressure laws. The problem of quantifying the rate of convergence was studied for the first time in \cite{amassad2025deterministic} where the rate of $\sqrt{\eps}$ in the Wasserstein distance was obtained for any kernel satisfying \ref{ass:kernel_convexity} in Assumption \ref{ass:kernel}. The proof is based on exploiting the formula for $\frac{\diff}{\diff t} \mathcal{W}_2^2(\mu_{\eps},\mu)$ as in \cite[Theorem 5.24]{MR3409718}. Then, in \cite{PME_1d_Brinkman_Darcy_rate} a~simpler proof was given for the particular choice of kernel satisfying \eqref{eq:PDE_nonlocal_1D_eps}. Here, the argument exploits the Evolutionary Variational Inequality. Nevertheless, the numerical simulations presented in \cite{PME_1d_Brinkman_Darcy_rate} suggested that the rate could be improved to $\eps$ and the result presented in this section confirms these expectations.\\
 
We first present assumptions for the kernel $\omega$. To motivate, it is easy to see that they are satisfied for $\omega(x) = \frac{1}{2}e^{-|x|}$ satisfying \eqref{eq:elliptic_PDE_omega} and $\omega(x) = (1-|x|)\,\mathds{1}_{|x|\leq 1}$. 
\begin{Ass}\label{ass:kernel}
We assume that $\omega$ satisfies:
\begin{enumerate}[label=($\text{A}_{\arabic*}^{\text{ker}}$)]
\item\label{ass:kernel_mass} $\omega \geq 0$, $\omega(x) = \omega(-x)$, $\int_{\R} \omega(y)\diff y = 1$ and $\int_{\R} \omega(y)\,|y|^2 \diff y < \infty$,
\item\label{ass:kernel_regularity} $\omega, \partial_x \omega \in L^{\infty}(\R)$ and $\partial^2_x \omega \in \mathcal{M}(\R)$,
\item\label{ass:kernel_convexity} $\omega$ is convex on the set $\{x\in \R:x\geq0\}$.
\end{enumerate}
\end{Ass}

We recall that $\mathcal{M}(\R)$ denotes the space of bounded Radon measures with the total variation norm, coinciding with the $L^1(\R)$ norm if the measure has a density with respect to the Lebesgue measure. Concerning conditions above, \ref{ass:kernel_convexity} is necessary for the geodesic convexity of the corresponding interaction energy $\mathcal{G}[\mu] = \frac{1}{2} \int_{\R} \mu\, \mu \ast \omega_{\eps} \diff x$. We also point out that assumption \ref{ass:kernel_convexity} implies some irregularity of $\omega$ at $x=0$ so our arguments do not work for smooth kernels $\omega$. 

\begin{thm}\label{thm:rate_of_conv_nonlocal_to_local}
Let $\mu_0 \in \mathcal{P}_2(\R) \cap L^{\infty}(\R)$, $\partial_x \mu_0 \in L^{\infty}(\R)$, $|\partial_x^2 \mu_0|^- = - \partial^2_x \mu_0 \, \mathds{1}_{\partial^2_x \mu_0 <0} \in \mathcal{M}(\R)$.  Let $\mu^{\eps}$ be the solution of \eqref{eq:PDE_nonlocal_1D_eps} with $\omega_{\eps} = \frac{1}{\eps}\omega\left(\frac{x}{\eps}\right)$ where $\omega$ satisfies Assumption \ref{ass:kernel} and with initial condition $\mu_0$. Let $\mu$ be the solution of \eqref{eq:PDE_local_1D_eps} with initial condition $\mu_0$. Then,
\begin{equation}\label{eq:main_thm_rate_nonlocal_to_local}
\mathcal{W}^2_2(\mu^{\eps}_{t}, \mu_t) \leq 2\,t \, \eps^2\,  \Big(5 t\,\| \partial_x \mu_0\|^2_{L^{\infty}(\R)}\,  + 
\frac{1}{2}\, \|\mu_0\|_{L^{\infty}(\R)} \Big) \, \| |\partial_{x}^2 \mu_0|^{-} \|_{\mathcal{M}(\R)}\, \| \omega\,|y|^2 \|_{L^1(\R)}.
\end{equation}
\end{thm}
Comparing to the previous arguments in \cite{amassad2025deterministic,PME_1d_Brinkman_Darcy_rate}, thanks to the Theorem \ref{thm:simplified_for_2-Wasserstein}, we can prove
\begin{equation}\label{eq:explanation_step_proof_theorem_rate_nonlocal_local}
\mathcal{W}_2(\mu^{\eps}_{t}, \mu_t) \leq \sqrt{t}\, \left(\int_0^t \int_{\R} |\partial_x \mu_s \ast \omega_{\eps} - \partial_x \mu_s |^2 \mu_s \diff x \diff s \right)^{\frac{1}{2}},
\end{equation}
where $\mu_t$ solves \eqref{eq:PDE_local_1D_eps}. We can use regularity of solutions to \eqref{eq:PDE_local_1D_eps}, namely the Aronson-B\'{e}nilan estimates \cite{MR524760} reviewed in Lemma \ref{lem:estimates_PME_1D}, to control the (RHS) of \eqref{eq:explanation_step_proof_theorem_rate_nonlocal_local}. This is the new ingredient implying better rate of convergence than known so far. From this point of view, our method is very similar to the relative entropy method \cite{MR3615546}, where higher regularity of the limit is used to quantify the rate of convergence. \\

We prove Theorem \ref{thm:rate_of_conv_nonlocal_to_local} using Theorem \ref{thm:simplified_for_2-Wasserstein}, as applying Theorem \ref{thm:main} would require overcoming several technical difficulties. To satisfy the assumptions of Theorem \ref{thm:main}, the continuity equations need to have smooth velocity fields so some regularization is required. For \eqref{eq:PDE_local_1D_eps}, the natural regularization is adding a small constant to the initial condition which is only allowed on a bounded domain (one cannot add a small gaussian on the whole space since the estimates would blow up for large arguments). However, \eqref{eq:PDE_nonlocal_1D_eps} posed on a bounded domain will not satisfy \ref{ass:pointwise_boundary_cond} in general. This is easy to see, say on $\Omega = (-1,1)$ and $\omega(x) = \frac{1}{2}e^{-|x|}$, since the velocity field has a wrong sign at the boundary. Indeed, at $x=-1$ we have
$$
-(\partial_x \omega\ast \mu)(-1) = \int_{-1}^1 e^{-|-1-y|}\, \text{sgn}(-1-y)\, \mu(y) \diff y  < 0 
$$
so the velocity field points outwards of $\Omega$.\\

One can also prove Theorem \ref{thm:rate_of_conv_nonlocal_to_local} on a periodic domain but this requires several modifications, including periodization of the kernel (since it may be not compactly supported). We briefly discuss the necessary modifications in Section \ref{subsect:nonlocal2local_periodic_domain}.

\subsection{General aggregation-diffusion equation}\label{subsect:general_A_D_intro} We now extend Theorem \ref{thm:PME} to the case of a general aggregation-diffusion equation. To this end, we consider
\begin{equation}\label{eq:aggregation_diffusion_1}
\partial_t \mu = \Delta \mu^m + \DIV(\mu \nabla(V_{\mu} + W_{\mu}\ast\mu)),
\end{equation}
\begin{equation}\label{eq:aggregation_diffusion_2}
\partial_t \nu = \Delta \nu^n +  \DIV(\nu \nabla(V_{\nu} + W_{\nu}\ast\nu)),
\end{equation}
posed on either $\Rd$ or a periodic domain. Under Assumption \ref{ass:potential} below, both problems admit unique distributional solutions for an initial condition in $\mathcal{P}_2(\Omega) \cap L^{\infty}(\Omega)$. Indeed, both problems can be interpreted as a 2-Wasserstein gradient flows with the corresponding energies being $\lambda$-geodesically convex for some $\lambda$ (by \ref{ass:hessian_of_potential}). It follows that there exist unique EVI solutions to \eqref{eq:aggregation_diffusion_1} and \eqref{eq:aggregation_diffusion_2}, which are equivalent to subdifferential solutions \cite[Theorem 4.35]{MR3050280} and are also distributional solutions \cite[Propositions 4.36--4.38]{MR3050280}.

\begin{Ass}\label{ass:potential}
We assume that each $\mathcal{V} := V_{\mu}, V_{\nu}, W_{\mu}, W_{\nu}$ satisfies
\begin{enumerate}[label=($\text{A}_{\arabic*}^{\text{pot}}$)]
\item $\mathcal{V}: \Omega\to \R$,
\item\label{ass_item:growth_conditions_V_W_main_thm} there are nonnegative constants $C^0_{\mathcal{V}}$, $C^1_{\mathcal{V}}$, $C^2_{\mathcal{V}}$, $C^3_{\mathcal{V}}$, $C^4_{\mathcal{V}}$ such that 
$$
-C^0_{\mathcal{V}} \leq \mathcal{V}(x) \leq C^1_{\mathcal{V}}+ C^2_{\mathcal{V}}|x|^2, \quad |\nabla \mathcal{V}(x)|\leq C^3_{\mathcal{V}}\, + C^4_{\mathcal{V}}\, |x|,
$$
\item\label{ass:hessian_of_potential} $\nabla^2 \mathcal{V}\in L^{\infty}(\Omega)$,
\item for $\mathcal{V} = W_{\mu}, W_{\nu}$ we have $\mathcal{V}(x) = \mathcal{V}(-x)$.
\end{enumerate}
\end{Ass}
The main result of this subsection reads:
\begin{thm}\label{thm:stability_estimate_AD_eq}
Let $\mu_0 \in \mathcal{P}_2(\Omega) \cap L^{\infty}(\Omega)$. Let $\mu_t$, $\nu_t$ be the distributional solutions to \eqref{eq:aggregation_diffusion_1}--\eqref{eq:aggregation_diffusion_2} on $\Omega=\Rd$ or $\Omega = R\,\Td$ and the same initial condition $\mu_0 = \nu_0$. Suppose that $V_{\mu}$, $V_{\nu}$, $W_{\mu}$, $W_{\nu}$ satisfy Assumption \ref{ass:potential} and that 
$$
1<m_- \leq m \leq n \leq n_+ < \infty.
$$
Finally, suppose that
\begin{equation}\label{eq:ass_convexity_stability_aggr_diffusion_exponent_gamma}
\nabla^2 V_{\nu} \geq c_{V_{\nu}} \, \mathcal{I}, \qquad \quad \nabla^2 W_{\nu} \geq c_{W_{\nu}} \, \mathcal{I}, \quad \qquad \Lambda := \begin{cases}
c_{W_{\nu}} &\mbox{ if } V = 0,\\
c_{V_{\nu}} &\mbox{ if } c_{W_{\nu}} > 0,\\
c_{V_{\nu}} +c_{W_{\nu}}  &\mbox{ if } c_{W_{\nu}} \leq 0.
\end{cases}
\end{equation}
Let $C$ be given by \eqref{eq:constant_continuity_wrt_exponent_W2_PME_main_thm} with the constant $C_{\kappa}$ defined now as
\begin{equation}\label{eq:constant_c_kappa_for_aggr_diff_PDE}
C_{\kappa} = \min\left(\frac{1}{e\,\kappa} + L_{\mu}(t)\,\|\mu_0\|_{L^{\infty}(\Omega)}^{\kappa+1}, \,  \|\mu_0\|_{L^{\infty}(\Omega)}^{\kappa}\,(\|\log \mu_0\|_{L^{\infty}(\Omega)} + \log L_{\mu}(t)) \right),
\end{equation}
and $L_\mu(t) = \exp(t\, (\|\Delta V_{\mu}\|_{L^{\infty}(\Omega)} +\|\Delta W_{\mu}\|_{L^{\infty}(\Omega)}))$. Then,
\begin{equation}\label{eq:continuity_wrt_exponent_W2_aggr-diff_main_thm}
\begin{split}
\mathcal{W}_2(\mu_t, \nu_t) \leq & \, C\, \left(\frac{1-e^{-2\Lambda t}}{2\Lambda}\right)^{1/2}  \,L_{\mu}(t)^{\frac{2n-m-1}{2}}\, |m-n| \\
& + \left(\frac{1-e^{-2\Lambda t}}{2\Lambda}\right)^{1/2}  \left( \int_0^t  \int_{\Omega} \left| \nabla V_{\mu}-\nabla V_{\nu} \right|^2 \, \mu_s \diff x  \diff s\right)^{\frac{1}{2}}\\
&+  \left(\frac{1-e^{-2\Lambda t}}{2\Lambda}\right)^{1/2}  \left(  \int_0^t \int_{\Omega} \left| \nabla (W_{\mu}- W_{\nu})\ast\mu_s \right|^2 \, \mu_s \diff x  \diff s\right)^{\frac{1}{2}}.
\end{split}
\end{equation}
\end{thm}

We remark that when $\Lambda =0$, the term $\frac{1-e^{-2\Lambda t}}{2\Lambda}$ should be interpreted as $t$ so that we recover the same constant as for the porous medium equation case. Furthermore, we remark that the estimate for the case $m=n$ has been obtained in \cite[Appendix A.2]{MR4896516}.\\

To prove Theorem \ref{thm:stability_estimate_AD_eq}, we use Theorem \ref{thm:simplified_for_2-Wasserstein}, and one may wonder why we do not instead apply Theorem~\ref{thm:main}. The reason is that, in order to apply Theorem \ref{thm:main}, we need to regularize equations \eqref{eq:aggregation_diffusion_1}--\eqref{eq:aggregation_diffusion_2}; see, for example, the proof of Theorem \ref{thm:PME} via Theorem \ref{thm:main} in Section \ref{sect:applications_to_PME}. To obtain the required regularity for the vector fields in Assumption \ref{ass:BIG_regularity_of_sln_and_vf}, we must work on a bounded domain. Even if we start with an initial condition $\mu_0^{\eps} \geq \eps$ on a bounded domain $\Omega$, it is not known whether this property is preserved for all times on a bounded domain, unlike in the periodic or whole-space setting $\Rd$. Therefore, the regularization has to be performed on a periodic domain. This, in turn, requires a suitable modification of the potentials so that they become periodic, as performed, for example, in Remark \ref{rem:justification_smoothness_solutions_to_aggr-diff}. Consequently, the lower bound on the hessians in \eqref{eq:ass_convexity_stability_aggr_diffusion_exponent_gamma} becomes weaker (see \eqref{eq:estimates_hessian_VP_app:finite_speed} for the estimate of the hessian of the modified potential), leading to a weaker constant in \eqref{eq:continuity_wrt_exponent_W2_aggr-diff_main_thm}.

\begin{rem}
Comparing with Subsection \ref{subsect:PME_introduction} on the case of pure porous medium equation, we cannot prove an analogue of Theorem \ref{thm:PME_incompress_limit} on the rate $\frac{1}{\sqrt{m}}$ for the limit $m\to\infty$. The reason is that the result in Theorem \ref{thm:PME_incompress_limit} is based on the estimate
$$
k\,(k+1) \, m\, \int_0^t \int_{\Omega} \mu_s^{k+m-2}\, |\nabla \mu_s|^2 \diff x \diff s \leq \int_{\Omega} \mu^{k+1}_0 \diff x
$$
for all $k\geq0$, see \eqref{eq:energy_estimate_general_k}. For equation \eqref{eq:aggregation_diffusion_1}, it takes the form
$$
m\, k \, (k+1) \int_0^t \int_{\Omega} \mu_s^{m+k-2}\, |\nabla \mu_s|^2 \diff x \diff s \leq e^{k\,t\, (\|\Delta V_{\mu}\|_{L^{\infty}(\Omega)} +\|\Delta W_{\mu}\|_{L^{\infty}(\Omega)})} \, \int_{\Omega} \mu_0^{k+1},
$$
see \eqref{eq:a_priori_estimate_power_density_A-D}. The exponential factor does not allow to control uniformly the (RHS) for large $k$. We remark that the rate for $m\to \infty$ for \eqref{eq:aggregation_diffusion_1} has been achieved recently in \cite{david2024improved}.
\end{rem}

\subsection{Structure of the paper}
The structure of the paper is as follows. In Section \ref{sect:proof_2_main_theorems} we prove Theorem \ref{thm:simplified_for_2-Wasserstein} (Subsection \ref{sect:alternative_proof}) and Theorem \ref{thm:main} (Subsection \ref{sect:proof_main_theorem}). Then, Sections~\ref{sect:applications_to_PME}, \ref{sect:applications_to_nonlocal2local}, and \ref{sect:aggregation-diffusion} present results regarding the porous medium equation, the nonlocal-to-local limit, and the aggregation-diffusion equation, introduced in Sections \ref{subsect:PME_introduction}, \ref{subsect:intro_nonlocal2local_limit}, and \ref{subsect:general_A_D_intro}. There are also five appendices. Appendix \ref{app:semigroup_theory} reviews semigroup theory while Appendix \ref{app:continuity_eq_opt_transport} summarizes exploited facts about the continuity equation and the optimal transport theory (both are relevant to the proof of Theorem \ref{thm:main}). Next, Appendix \ref{app:rate_PME_nto1} presents another rate of convergence relevant for the discussion in Section \ref{subsect:PME_introduction} while Appendix \ref{app:PME_m2_and_nonlocal2local_facts} summarizes estimates for the quadratic porous medium equation relevant for the proof of Theorem \ref{thm:rate_of_conv_nonlocal_to_local}. Finally, Appendix \ref{app:general_aggr_diff_PDE} presents a finite speed of propagation result for aggregation-diffusion equations, which is necessary for the proof of Theorem \ref{thm:stability_estimate_AD_eq}. 
 
\section{Proofs of Theorem \ref{thm:simplified_for_2-Wasserstein} and \ref{thm:main}}\label{sect:proof_2_main_theorems}

\subsection{Proof of Theorem \ref{thm:simplified_for_2-Wasserstein}}\label{sect:alternative_proof}
First, by \cite[Theorem 8.4.7]{MR2129498}, 
\begin{equation}\label{eq:derivative_W_2^2_alternative_proof_our_formula}
\frac{\diff}{\diff t} \frac{1}{2} \mathcal{W}_2^2(\mu_t, \nu_t) =  \int_{\Omega} \nabla\varphi(x) \cdot \bvmu[\mu_t]\diff \mu_t(x)  - \int_{\Omega} \nabla\phi(x) \cdot \nabla\frac{\delta \mathcal{G}}{\delta \nu}[\nu_t] \, \diff \nu_t(x),
\end{equation}
where $\varphi$, $\phi$ are Kantorovich potentials for the optimal transport problem with a quadratic cost $c(x,y)= \frac{1}{2}|x-y|^2$ between $\mu_t$, $\nu_t$ and $\nu_t$, $\mu_t$, respectively, where we omit the time dependence for simplicity. On $\Omega=R\,\Td$, one can obtain \eqref{eq:derivative_W_2^2_alternative_proof_our_formula} by exploiting the dual formulation of optimal transport as in \cite[Theorems 1.25, 5.24]{MR3409718}.\\

We will prove that 
\begin{equation}\label{eq:additional_claim_geodesic_convexity}
 \lambda \, \mathcal{W}_2^2(\mu_t,\nu_t) - \int_{\Omega} \nabla \phi(x) \cdot \nabla\frac{\delta \mathcal{G}}{\delta \nu}[\nu_t] \diff \nu_t(x) \leq  \int_{\Omega} \nabla\varphi(x) \cdot \nabla\frac{\delta \mathcal{G}}{\delta \nu}[\mu_t] \diff \mu_t(x).
\end{equation}
Assuming \eqref{eq:additional_claim_geodesic_convexity} is proved, we deduce from \eqref{eq:derivative_W_2^2_alternative_proof_our_formula} that
\begin{multline*}
\frac{\diff}{\diff t} \frac{1}{2} \mathcal{W}_2^2(\mu_t, \nu_t) + \lambda \, \mathcal{W}_2^2(\mu_t,\nu_t) \leq  \int_{\Omega} \nabla \varphi(x) \cdot \left( \nabla\frac{\delta \mathcal{G}}{\delta \nu}[\mu_t] + \bvmu[\mu_t] \right)\diff \mu_t(x) \\
\leq \left(\int_{\Omega} |\nabla\varphi(x)|^2 \diff \mu_t(x)\right)^{\frac{1}{2}} \left( \int_{\Omega} \left| \nabla\frac{\delta \mathcal{G}}{\delta \nu}[\mu_t] + \bvmu[\mu_t] \right|^2 \diff \mu_t(x) \right)^{\frac{1}{2}}.
\end{multline*}
Since $\nabla \varphi(x) = x- T(x)$, $\int_{\Omega} |\nabla \varphi(x)|^2 \diff \mu_t(x) = \mathcal{W}_2^2(\mu_t,\nu_t)$ (on $\Omega=R\,\Td$, see \eqref{eq:Wass_dis_periodic_integral_of_Kantorovich}), we arrive at \eqref{eq:claim_formula_for_Wass_gradient_flows}.\\

It remains to prove \eqref{eq:additional_claim_geodesic_convexity}. Let $\gamma_s$ be a geodesic connecting $\mu_t$ with $\nu_t$, where we omit the dependence on $t$ for simplicity. By $\lambda\, \mathcal{W}_2^2(\rho_0, \rho_1)$-convexity of the function $s\mapsto \mathcal{G}[\gamma_s]$ in \ref{eq:geodesic_conv_G_sect:alternative_proof} and $\mu_t, \nu_t \in \mathcal{D}(\mathcal{G})$ in \ref{ass:regularity_solutions:add_sect_alternative_proof}, we know that $\frac{\diff}{\diff s} \mathcal{G}[\gamma_s] \Big|_{s=0}$ and $\frac{\diff}{\diff s} \mathcal{G}[\gamma_s] \Big|_{s=1}$ exist (we use $\mu_t, \nu_t \in \mathcal{D}(\mathcal{G})$ in order to guarantee that we do not subtract two infinite numbers) so we deduce from \ref{eq:geodesic_conv_G_sect:alternative_proof}
$$
\mathcal{G}[\nu_t] \geq \mathcal{G}[\mu_t] + \frac{\diff}{\diff s}\mathcal{G}[\gamma_s] \Big|_{s=0} + \frac{\lambda}{2} \, \mathcal{W}_2^2(\mu_t,\nu_t), \,\,\, \mathcal{G}[\mu_t] \geq \mathcal{G}[\nu_t] - \frac{\diff}{\diff s}\mathcal{G}[\gamma_s] \Big|_{s=1}+ \frac{\lambda}{2} \, \mathcal{W}_2^2(\mu_t,\nu_t). 
$$
In particular, we have
\begin{equation}\label{eq:monotonicity_derivatives_geodesics_sect:additional_proof}
 \frac{\diff}{\diff s}\mathcal{G}[\gamma_s] \Big|_{s=1} \geq \frac{\diff}{\diff s}\mathcal{G}[\gamma_s] \Big|_{s=0} + \lambda \, \mathcal{W}_2^2(\mu_t,\nu_t).
\end{equation}
Applying \ref{eq:derivative_via_1st_var_sect:alternative_proof} we obtain
$$
\int_{\Omega} \nabla \frac{\delta \mathcal{G}}{\delta \nu}[\nu_t]\, \nabla \phi(x) \diff \nu_t(x)\geq -\int_{\Omega} \nabla \frac{\delta \mathcal{G}}{\delta \nu}[\mu_t]\, \nabla\varphi(x) \diff \mu_t(x) + \lambda \, \mathcal{W}_2^2(\mu_t,\nu_t).
$$
Rearranging this inequality, we arrive at \eqref{eq:additional_claim_geodesic_convexity}.
 
\subsection{A Trotter-Kato type estimate}\label{sect:trotter-kato_flow_estimate}
Here, we prove a technical estimate, in the spirit of Trotter-Kato-type results, for approximating the semigroup $e^{t(A+B)}$ by the compositions $\left(e^{\frac{t}{n}A} \, e^{\frac{t}{n}B}\right)^n$, where $A$, $B$ are some operators \cite{MR108732, MR538020}. This estimate is a key ingredient in the proof of Theorem~\ref{thm:main}. Given a nonautonomous vector field ${\bf v}(t,x):(0,T) \times \Omega \to \Rd$ we define its flow $X(s,t,x)$ by the formula
\begin{equation}\label{eq:flow_of_the_vf}
    \begin{split}
 \partial_t  X(s,t,x) &= {\bf v}(t, X(s,t,x)),\\
X(s,s,x) &= x.       
    \end{split}
\end{equation}
We use three arguments $s$, $t$ and $x$ in $X(s,t,x)$ to keep track of the initial time, current time and initial position, respectively. Equation \eqref{eq:flow_of_the_vf} is well-posed assuming ${\bf v} \in L^1(0,T; \BL(\Omega))$, where $\BL(\Omega)$ is the space of bounded Lipschitz functions on $\Omega$ with norm \eqref{eq:norm_BL}. If $\Omega$ is a bounded domain, we need to assume ${\bf v} \cdot \bn \leq 0$ on $\partial \Omega$ for a.e. $t$, where $\bn$ is an outward normal vector to $\partial \Omega$.   

\begin{proposition}\label{thm:trotter-kato-estimate}
    Let ${\bf v^1}, {\bf v^2} :(0,T)\times\Omega \to \R^d$ be two vector fields (with ${\bf v^1} \cdot \bn \leq 0$, ${\bf v^2} \cdot \bn \leq 0$ on $\partial \Omega$ for a.e. $t$ if $\Omega$ is a bounded smooth domain) and let $X^{1}(s,t,x)$, $X^{2}(s,t,x)$ be the flows of ${\bf v^1}$ and ${\bf v^2}$, respectively. Let $ X^{1+2}(s,t,x)$ be the flow generated by the sum ${\bf v^{1} + v^2}$. Then
\begin{align*}
    &\left| X^{1+2}(s,s+h,x) - X^1\big(s, s+h, X^2(s,s+h,x)\big)\right| \leq   \\
    & \qquad \qquad \qquad \qquad \qquad  \leq 2 \left( \int_{s}^{s+h} 
\left(\|{\bf v^1}(u,\cdot) \|_{\BL(\Omega)} + \|{\bf v^2}(u,\cdot) \|_{\BL(\Omega)} \right) \diff u \right)^2.
\end{align*}
\end{proposition}

\begin{proof}
Let us denote the difference to be estimated by $\Delta$. By the definition of $X^2$
$$
X^2(s, s+h, x) = x + \int_{s}^{s+h} {\bf v^2}(u,X^2(s, u, x)) \diff u,
$$
so that by the definition of $X^1$
\begin{multline*}
X^1\big(s, s+h, X^2(s,s+h,x)\big) =\\=
x + \int_{s}^{s+h} {\bf v^2}(u,X^2(s, u, x)) \diff u + \int_{s}^{s+h} {\bf v^1}\big(u, X^1\big(s, u, X^2(s,s+h,x)\big)\big) \diff u.
\end{multline*}
On the other hand, $X^{1+2}$ satisfies
$$
X^{1+2}(s,s+h,x) = x + \int_{s}^{s+h} {\bf v^2}(u,X^{1+2}(s, u, x)) \diff u + \int_{s}^{s+h} {\bf v^1}(u,X^{1+2}(s, u, x)) \diff u.
$$
Subtracting the two terms above we get
\begin{align*}
&\Delta \leq  \int_{s}^{s+h} \left|{\bf v^2}(u,X^{1+2}(s, u, x)) - {\bf v^2}(u,X^2(s, u, x))\right| \diff u \\
& \quad + \int_{s}^{s+h} \left|{\bf v^1}(u,X^{1+2}(s, u, x)) -  {\bf v^1}\big(u, X^1\big(s, u, X^2(s,s+h,x)\big)\big)\right| \diff u\\
& \quad =: A+ B. \phantom{\int_{s+h}^{s+2h}}
\end{align*}
\underline{Term $A$.} Using Lipschitz continuity of ${\bf v^2}$ we obtain
\begin{align*}
A &\leq \int_{s}^{s+h} \Lip({\bf v^2}(u,\cdot))\,  |X^{1+2}(s, u, x) -X^2(s, u, x)| \diff u \\
&\leq \int_{s}^{s+h} \Lip({\bf v^2}(u,\cdot))\,  \left(|X^{1+2}(s, u, x) - x| + |X^2(s, u, x) - x|\right) \diff u, 
\end{align*}
where in the second step we added and subtracted $x$. Now, since $u \in [s,s+h]$,
\begin{equation}\label{eq:flow_difference_integral_in_time}
|X^{1+2}(s, u, x) - x| \leq \int_{s}^{s+h} 
\left(\|{\bf v^1}(\tau,\cdot) \|_{L^{\infty}(\Omega)} + \|{\bf v^2}(\tau,\cdot) \|_{L^{\infty}(\Omega)} \right) \diff \tau 
\end{equation}
and similarly for the term $|X^{2}(s, u, x) - x|$. We conclude that
$$
A \leq 2\, \int_{s}^{s+h} \Lip({\bf v^2}(u,\cdot)) \diff u\,  \int_{s}^{s+h} \left(\|{\bf v^1}(u,\cdot) \|_{L^{\infty}(\Omega)} + \|{\bf v^2}(u,\cdot) \|_{L^{\infty}(\Omega)}\right) \diff u.
$$
\underline{Term $B$.} Using Lipschitz continuity of ${\bf v^1}$ and 
$$
B \leq \int_{s}^{s+h} \Lip({\bf v^1}(u,\cdot))\, \big|X^{1+2}(s, u, x) - X^1\big(s, u, X^2(s,s+h,x)\big)\big| \diff u.
$$
The term inside can be estimated by triangle inequality
\begin{multline*}
\big|X^{1+2}(s, u, x) - X^1\big(s, u, X^2(s,s+h,x)\big)\big|  \leq \big|X^{1+2}(s, u, x) -x \big| + \\ + \big|x- X^2(s,s+h,x)  \big| +\big|X^2(s,s+h,x) - X^1\big(s, u, X^2(s,s+h,x)\big)\big|.
\end{multline*}
Since all three terms can be written as an integral of a vector field (see \eqref{eq:flow_difference_integral_in_time} for example), we arrive at 
\begin{multline*}
\big|X^{1+2}(s, u, x) - X^1\big(s, u, X^2(s,s+h,x)\big)\big| \leq \\
\leq 2\, \int_{s}^{s+h} 
\left(\|{\bf v^1}(\tau,\cdot) \|_{L^{\infty}(\Omega)} + \|{\bf v^2}(\tau,\cdot) \|_{L^{\infty}(\Omega)} \right) \diff \tau. 
\end{multline*}
Collecting these estimates, we obtain
$$
B \leq \, 2\, \int_{s}^{s+h} \Lip({\bf v^1}(u,\cdot)) \diff u \int_{s}^{s+h} 
\left(\|{\bf v^1}(u,\cdot) \|_{L^{\infty}(\Omega)} + \|{\bf v^2}(u,\cdot) \|_{L^{\infty}(\Omega)} \right) \diff u.  
$$
By estimating $\|{\bf v^i}(u,\cdot) \|_{L^{\infty}(\Omega)}, \Lip({\bf v^i}(u,\cdot)) \leq \|{\bf v^i}(u,\cdot) \|_{\BL(\Omega)}$, the proof is concluded. 
\end{proof}

\subsection{Proof of Theorem \ref{thm:main}}\label{sect:proof_main_theorem}

We recall that we consider both problems \eqref{eq:PDE_continuity_equation_1} and \eqref{eq:PDE_continuity_equation_2} with the same initial condition $\mu_0 = \nu_0 \in \mathcal{A}_1$. By Assumptions \ref{ass:abs_cont_first_map}--\ref{ass:semigroup_second_map}, there exists a~Lipschitz, absolutely continuous in time semigroup $\mathcal{S}_{t}$ generated by \eqref{eq:PDE_continuity_equation_2} so that $\nu_t = \mathcal{S}_{t} \mu_0$. Moreover, the map $t\mapsto \mu_t$ is absolutely continuous (Assumption \ref{ass:abs_cont_first_map}). Therefore, Lemma \ref{lem:bressan_estimate_between_map_semigroup_aut} applied on the metric space $(\mathcal{A}_2, \mathcal{W}_p)$ yields
\begin{equation}\label{eq:first_step_of_the_proof_application_bressan}
\mathcal{W}_p(\mu_t, \nu_t) =  \mathcal{W}_p(\mu_t, \mathcal{S}_{t} \mu_0) \leq  \int_0^t K(t-s) \,\liminf_{h\to0^+} \frac{\mathcal{W}_p(\mu_{s+h}, \mathcal{S}_{h} \mu_{s})}{h} \diff s. 
\end{equation}
We now fix $h$ and we want to estimate $\mathcal{W}_p(\mu_{s+h}, \mathcal{S}_{h} \mu_{s})$. Since both problems generate a semigroup (Assumption \ref{ass:semigroup_second_map}), we have that $\rho_\tau^1:=\mu_{s+\tau}$ and $\rho_{\tau}^2:=  \mathcal{S}_{\tau} \mu_{s}$, where $\tau \in [0,h]$, solve the PDEs
\begin{equation}\label{eq:continuity_equation_PDE}
\begin{split}
\partial_\tau \rho_{\tau}^{1} + \DIV(\rho_{\tau}^{1} \, \bvmu[\rho^1_{\tau}]) = 0, \quad  \partial_\tau \rho_{\tau}^{2} + \DIV(\rho_{\tau}^{2} \, \bvnu[\rho^2_{\tau}]) = 0, \quad   
\rho^1_0 = \rho^2_0 = \mu_s. 
\end{split}
\end{equation}
To estimate $\mathcal{W}_p(\mu_{s+h}, \mathcal{S}_{h} \mu_{s}) = \mathcal{W}_p(\rho^1_{h}, \rho^2_{h})$, we introduce an auxiliary curve $\rho_{\tau}^3$ with $\tau \in [0,h]$ solving
\begin{align*}
\partial_\tau \rho_{\tau}^{3} + \DIV\left(\rho_{\tau}^{3}\, \left(\bvnu[\rho^2_{\tau}]- \bvmu[\rho^1_{\tau}] \right) \right) &= 0,\\
\rho_0^3 &= \mu_{s+h}.
\end{align*}
The PDE is well-posed by the regularity assumption \ref{ass:regularity_vector_fields} and the boundary conditions~\ref{ass:pointwise_boundary_cond} which implies $\left(\bvnu[\rho^2_{\tau}]- \bvmu[\rho^1_{\tau}] \right) \cdot \bn \leq 0$ on $\partial \Omega$ so that characteristics stay in $\Omega$ (see Lemma~\ref{lem:representation_solution_continuity_equation} for the well-posedness result). Moreover, by Lemma \ref{lem:p_moment_stays_bounded}, we have $\rho^3_\tau \in \mathcal{P}_p(\Omega)$ for all $\tau \in [0,h]$. By the triangle inequality 
\begin{equation}\label{eq:triangle_equation_mu123}
\frac{\mathcal{W}_p(\rho^1_{h}, \rho^2_{h})}{h} \leq \frac{\mathcal{W}_p(\rho^1_{h}, \rho^3_{h})}{h} + \frac{\mathcal{W}_p(\rho^3_{h}, \rho^2_{h})}{h}.
\end{equation}
The terms on the (RHS) of \eqref{eq:triangle_equation_mu123} will be analyzed independently.\\

\underline{The term ${\mathcal{W}_p(\rho^3_{h}, \rho^2_{h})}/{h}$.} We will prove
\begin{equation}\label{eq:main_part_curves_3_2}
\lim_{h\to 0^+} \frac{\mathcal{W}_p(\rho^3_{h}, \rho^2_{h})}{h} = 0.
\end{equation}
We define velocity fields 
$$
{\bf V^1}(\tau,x) := \bvmu[\rho^1_{\tau}], \qquad {\bf V^2}(\tau,x) := \bvnu[\rho^2_{\tau}], \qquad {\bf V^3}(\tau, x) := {\bf V^2}(\tau,x) - {\bf V^1}(\tau,x),
$$
and $X^{i}$ to be the corresponding flow of the vector field ${\bf V^i}$ as in \eqref{eq:flow_of_the_vf}. The flows are well-defined because \ref{ass:pointwise_boundary_cond} implies ${\bf V^1}(\tau,x) \cdot \bn, {\bf V^2}(\tau,x) \cdot \bn, {\bf V^3}(\tau,x) \cdot \bn \leq 0$ on $\partial \Omega$. It follows from Lemma~\ref{lem:representation_solution_continuity_equation} that 
\begin{equation}\label{eq:representation_measures_auxiliary}
\rho^3_h = X^{3}(0,h,\cdot)^{\#}\, X^1(0,h,\cdot)^{\#} \, \mu_s, \qquad \qquad \rho^2_h = X^2(0,h, \cdot)^{\#} \mu_s.
\end{equation}
Note that $X^{3}(0,h,\cdot)^{\#}\, X^1(0,h,\cdot)^{\#} = X^3(0,h, X^1(0, h, \cdot))^{\#}$ so that using Lemma \ref{lem:estimate_Wasserstein_by_pushforwards} 
$$
\mathcal{W}_p(\rho^3_{h}, \rho^2_{h}) \leq \left\| X^{3}(0,h,X^1(0,h,\cdot)) -  X^2(0,h, \cdot) \right\|_{L^{\infty}(\Omega)} \, \| \mu_s \|^{1/p}_{L^1(\Omega)}.
$$
Now, it remains to notice that ${\bf V^3} + {\bf V^1} = {\bf V^2}$ so we are exactly in the setting of Proposition~\ref{thm:trotter-kato-estimate}. From \ref{ass:regularity_vector_fields} we know that ${\bf V^1}$, ${\bf V^2}$ belong to $L^{2+\delta}(0,T; \BL(\Omega))$. Hence, using Proposition~\ref{thm:trotter-kato-estimate} and H{\"o}lder inequality
\begin{multline*}
\left\| X^{3}(0,h,X^1(0,h,\cdot)) -  X^2(0,h, \cdot) \right\|_{L^{\infty}(\Omega)}  \leq \\ 
\leq C\, \left( \int_{0}^{h} 
\left(\|{\bf V^1}(u,\cdot) \|_{\BL(\Omega)} + \|({\bf V^2 -V^1})(u,\cdot) \|_{\BL(\Omega)} \right) \diff u \right)^2 \leq C\,h^{2\,\frac{1+\delta}{2+\delta}},
\end{multline*}
where the second constant $C$ involves norms of ${\bf V^1}$, ${\bf V^2}$. Hence, we deduce $\mathcal{W}_p(\rho^3_{h}, \rho^2_{h}) \leq C\,h^{2\,\frac{1+\delta}{2+\delta}}$ where $C$ does not depend on $h$. Since $2\,\frac{1+\delta}{2+\delta}>1$, the conclusion follows.\\

\underline{The term ${\mathcal{W}_p(\rho^1_{h}, \rho^3_{h})}/{h}$.} By the Benamou-Brenier formula (Lemma \ref{lem:B-B}) 
\begin{equation}\label{eq:formula_for_wass_bb_in_the_main_proof}
\frac{\mathcal{W}_p(\rho^1_{h}, \rho^3_{h})}{h} \leq  \left(\frac{1}{h} \int_0^h \int_{\Omega} \left|\bvnu[\rho^2_{\tau}]- \bvmu[\rho^1_{\tau}]  \right|^p \diff \rho_{\tau}^3 \diff \tau \right)^{\frac{1}{p}}.
\end{equation}
Here, if $\Omega$ is a bounded domain, we could use Lemma \ref{lem:B-B} because we know that on $\partial \Omega$ we have $(v^2[\mu^2_{\tau}]- v^1[\mu^1_{\tau}]) \cdot \bn \leq 0$ on $\partial \Omega$ for a.e. $\tau\in[0,h]$ thanks to \ref{ass:pointwise_boundary_cond}.\\

Now, we claim
\begin{equation}\label{eq:claim_limit_as_hto0}
\lim_{h\to 0^+} \frac{1}{h} \int_0^h \int_{\Omega} \left|\bvnu[\rho^2_{\tau}]- \bvmu[\rho^1_{\tau}]  \right|^p \diff \rho_{\tau}^3 \diff \tau = \int_{\Omega} \left|\bvnu[\mu_s]- \bvmu[\mu_s]  \right|^p \diff \mu_s.
\end{equation}
Note carefully that $\rho^3_{\tau}$ depends on $h$ through the initial condition. To prove \eqref{eq:claim_limit_as_hto0}, we let $f(\tau,x) := \left|\bvnu[\rho^2_{\tau}]- \bvmu[\rho^1_{\tau}]  \right|^p$, $f_0 := f(0,x) = \left|\bvnu[\mu_s]- \bvmu[\mu_s]  \right|^p$. By \ref{ass:C_admissibility} and \ref{ass:continuity_difference} we have $f_0\in L^{\infty}(\Omega)$ and $f \in L^{\infty}((0,h)\times\Omega)$. We let $\Omega_n = \Omega\cap \{|x|\leq n\}$ and we estimate the difference
\begin{align*}
&\left|\frac{1}{h} \int_0^h \int_{\Omega} (f(\tau,x)-f_0(x)) \diff \rho^3_{\tau}(x) \diff \tau \right| \leq\\ 
&\leq \frac{1}{h} \int_0^h \int_{\Omega_n} |f(\tau,x)-f_0(x)| \diff \rho^3_{\tau}(x) \diff \tau  + \frac{1}{h} \int_0^h \int_{\Omega\setminus \Omega_n} |f(\tau,x)-f_0(x)| \diff \rho^3_{\tau}(x) \diff \tau\\
&\leq  \esssup_{\tau \in [0,h]} \|f(\tau,\cdot)-f_0(\cdot)\|_{L^{\infty}(\Omega_n)} + (\|f\|_{L^{\infty}(\Omega)} + \|f_0\|_{L^{\infty}(\Omega}) \sup_{\tau \in [0,h]} \rho^3_{\tau}(\Omega \setminus \Omega_n). 
\end{align*}
The first term converges as $h \to 0^+$ for each $n$ by \ref{ass:continuity_difference}. Regarding the second term, we first estimate
\begin{equation}\label{eq:estimate_x3_x1_identity}
\begin{split}
&|X^3(0,\tau, X^1(0,h,x)) - x|
\leq \\
&\leq \|X^3(0,\tau, X^1(0,h,x)) - X^1(0,h,x)\|_{L^{\infty}(\Omega)} + \|X^1(0,h,x) - x\|_{L^{\infty}(\Omega)}\\
& \leq \int_0^h \| {\bf V^1}(u,\cdot)  \|_{L^{\infty}(\Omega)} \diff u +\int_0^h \|{\bf V^3}(u,\cdot) \|_{L^{\infty}(\Omega)} \diff u.
\end{split}
\end{equation}
Hence, using the representation $\rho^3_{\tau} = X^{3}(0,\tau,\cdot)^{\#}\, X^1(0,h,\cdot)^{\#} \, \mu_s$ as in~\eqref{eq:representation_measures_auxiliary}
$$
\int_{\Omega} |x|^p \diff \rho^3_{\tau} \leq 2^p \int_{\Omega} |x|^p \diff \mu_s + 2^p\left( \int_0^h \| {\bf V^1}(u,\cdot)  \|_{L^{\infty}(\Omega)} \diff u +\int_0^h \|{\bf V^3}(u,\cdot) \|_{L^{\infty}(\Omega)} \diff u\right)^p,
$$
where the last two integrals are finite by \ref{ass:regularity_vector_fields}. It follows that $ \sup_{\tau \in [0,h]} \rho^3_{\tau}(\Omega \setminus \Omega_n) \leq \frac{C}{n^p}$ independently of $h$ so that we obtain
\begin{equation}\label{eq:part_1_conv_integral_second_term_Wp}
\frac{1}{h} \int_0^h \int_{\Omega} (f(\tau,x)-f_0(x)) \diff \rho^3_{\tau}(x) \diff \tau \to 0 \mbox{ as } h \to 0^+.
\end{equation}
Next, we use $\rho^3_{\tau} = X^{3}(0,\tau,\cdot)^{\#}\, X^1(0,h,\cdot)^{\#} \, \mu_s$ to estimate \begin{equation}\label{eq:splitting_the_integral_In_Jn}
\begin{split}
&\left|\frac{1}{h} \int_0^h \int_{\Omega} f_0(x) \diff (\rho^3_{\tau}-\mu_s)(x) \diff \tau \right| \leq \\
& \qquad \qquad \leq  \sup_{\tau \in [0,h]} \left| \int_{\Omega} (f_0(X^3(0,\tau, X^1(0,h,x))) - f_0(x)) \diff \mu_{s}(x)  \right| =: I_n + J_n,
\end{split}
\end{equation}
where $I_n$ and $J_n$ correspond to splitting of the inner integral for sets $\Omega_n$ and $\Omega\setminus \Omega_n$ as above. Since $\mu_s \in \mathcal{P}_p(\Omega) \cap L^1(\Omega)$ and $f_0 \in L^{\infty}(\Omega)\cap C(\Omega)$ (by \ref{ass:C_admissibility}), $J_n \leq \frac{C}{n^p}$. Concerning $I_n$, from \eqref{eq:estimate_x3_x1_identity} we observe that
$$
\|X^3(0,\tau, X^1(0,h,x)) - x\|_{L^{\infty}(\Omega)} \leq 1
$$
for sufficiently small $h$ so that if $|x|\leq n$ then $|X^3(0,\tau, X^1(0,h,x))| \leq n+1$. Since $f_0 \in C(\Omega)$ by \ref{ass:C_admissibility}, it is uniformly continuous on $\{|x|\leq n+1\}$ and has a modulus of continuity $\omega_n$ on this set. Hence, we estimate $I_n$ by
$$
I_n \leq \omega_n\left(\int_0^h \| {\bf V^1}(u,\cdot)  \|_{L^{\infty}(\Omega)} \diff u +\int_0^h \|{\bf V^3}(u,\cdot) \|_{L^{\infty}(\Omega)} \diff u \right) \to 0 \mbox{ as } h\to 0^+.
$$
Consequently from \eqref{eq:splitting_the_integral_In_Jn}
$
\limsup_{h \to 0^+} \left|\frac{1}{h} \int_0^h \int_{\Omega} f_0(x) \diff (\rho^3_{\tau}-\mu_s)(x) \diff \tau \right| \leq \frac{C}{n^p}
$
and since $n$ is arbitrary, we obtain
\begin{equation}\label{eq:part_2_conv_integral_second_term_Wp}
\left|\frac{1}{h} \int_0^h \int_{\Omega} f_0(x) \diff (\rho^3_{\tau}-\mu_s)(x) \diff \tau \right| \to 0 \mbox{ as } h \to 0^+.
\end{equation}
Combining convergences \eqref{eq:part_1_conv_integral_second_term_Wp} and \eqref{eq:part_2_conv_integral_second_term_Wp} we arrive at \eqref{eq:claim_limit_as_hto0}. Then, combining \eqref{eq:formula_for_wass_bb_in_the_main_proof} and \eqref{eq:claim_limit_as_hto0}, we obtain
\begin{equation}\label{eq:main_part_curves_1_2}
\limsup_{h \to 0^+} \frac{\mathcal{W}_p(\rho^1_{h}, \rho^3_{h})}{h} \leq \left( \int_{\Omega} \left|\bvnu[\mu_s]- \bvmu[\mu_s]  \right|^p \diff \mu_s  \right)^{\frac{1}{p}}.
\end{equation}
\underline{Conclusion of the proof.} In view of \eqref{eq:triangle_equation_mu123}, \eqref{eq:main_part_curves_3_2} and \eqref{eq:main_part_curves_1_2}
$$
\liminf_{h \to 0^+} \frac{\mathcal{W}_p(\rho^1_{h}, \rho^2_{h})}{h} \leq \limsup_{h \to 0^+} \frac{\mathcal{W}_p(\rho^1_{h}, \rho^2_{h})}{h} \leq \left( \int_{\Omega} \left|\bvnu[\mu_s]- \bvmu[\mu_s]  \right|^p \diff \mu_s  \right)^{\frac{1}{p}}.
$$
Therefore, the assertion of the theorem follows directly from \eqref{eq:first_step_of_the_proof_application_bressan}.

\section{Proofs of Theorems \ref{thm:PME} and \ref{thm:PME_incompress_limit}}\label{sect:applications_to_PME}

To illustrate applicability of Theorems \ref{thm:simplified_for_2-Wasserstein} and \ref{thm:main}, we provide two independent proofs of Theorem \ref{thm:PME}.

\begin{proof}[Proof of Theorem \ref{thm:PME} (by using Theorem \ref{thm:simplified_for_2-Wasserstein})] {\phantom{...}} \\

\underline{Step 1: Application of Theorem \ref{thm:simplified_for_2-Wasserstein}.} We assume additionally that $\mu_0$ is compactly supported (this is an empty assumption if $\Omega$ is a bounded domain). We will prove 
\begin{equation}\label{eq:stability_PME_result_step1}
\mathcal{W}_2(\mu_t, \nu_t) \leq \sqrt{t}\left( \int_0^t  \int_{\Omega} \left| m  - n \, \mu_s^{n-m} \right|^2 |\nabla \mu_s|^2 \, \mu_s^{2m-3} \diff x  \diff s\right)^{\frac{1}{2}}.
\end{equation}
To this end, we define 
\begin{equation}\label{eq:functionals_F_G_for_PME}
\mathcal{F}[\mu] = \frac{1}{m-1} \int_{\Omega} \mu^m \diff x, \qquad \mathcal{G}[\nu] = \frac{1}{n-1} \int_{\Omega} \nu^n \diff x, 
\end{equation}
$$
\frac{\delta \mathcal{F}}{\delta \mu}[\mu]=\frac{m}{m-1}\mu^{m-1}, \qquad \frac{\delta \mathcal{G}}{\delta \nu}[\nu]= \frac{n}{n-1}\nu^{n-1},
$$
and we will verify Assumption \ref{eq:assumption_simplified_formula} with $\bvmu=-\nabla\frac{\delta \mathcal{F}}{\delta \mu}[\mu]$ and $\mathcal{G}$ as above.\\

We recall that since $\mu_0$ is compactly supported in this step, $\mu_t$ and $\nu_t$ stay compactly supported for all times by the finite speed of propagation of the porous medium equation \cite{MR2286292}. Moreover, it will be useful to note down the following identity satisfied by $\mu_t$ and all $k \geq 0$
\begin{equation}\label{eq:energy_estimate_general_k}
\int_{\Omega} \mu_t^{k+1} \diff x + k\,(k+1) \, m\, \int_0^t \int_{\Omega} \mu_s^{k+m-2}\, |\nabla \mu_s|^2 \diff x \diff s = \int_{\Omega} \mu^{k+1}_0 \diff x.
\end{equation}
This identity is obtained by multiplying \eqref{eq:PME_1} by $\mu_t^k$ and integrating by parts. In particular, since a similar identity is satisfied for $\nu_t$ and $\mu_0 \in L^{1}(\Omega)\cap L^{\infty}(\Omega)$, we deduce that $\mu_t$ and $\nu_t$ are in $L^p(\Omega)$ for all $p\in[1,\infty]$ and all $t\geq 0$.\\

We now verify Assumption \ref{eq:assumption_simplified_formula}. Regarding condition \ref{ass:regularity_solutions:add_sect_alternative_proof}, both functions $\nabla\frac{\delta \mathcal{F}}{\delta \mu}[\mu_t] \,\sqrt{\mu_t}$, $\nabla \frac{\delta \mathcal{G}}{\delta \nu}[\nu_t] \,\sqrt{\nu_t}$ belong to $L^2((0,T)\times\Omega)$ because of standard energy identities
$$
\partial_t \mathcal{F}[\mu_t] + \int_{\Omega} \left| \nabla \frac{\delta \mathcal{F}}{\delta \mu}[\mu_t] \,\sqrt{\mu_t} \right|^2 \diff x\leq 0, \quad \partial_t \mathcal{G}[\nu_t] + \int_{\Omega} \left| \nabla \frac{\delta \mathcal{G}}{\delta \nu}[\nu_t] \,\sqrt{\nu_t} \right|^2 \diff x\leq 0.
$$
Moreover, function $\nabla \frac{\delta \mathcal{G}}{\delta \nu}[\mu_t] \,\sqrt{\mu_t} = n\,\mu_t^{n-\frac{3}{2}}\,\nabla\mu_t$ is also in $L^2((0,T)\times\R)$ thanks to \eqref{eq:energy_estimate_general_k} with $k=2n-m-1$. Finally, $\mu_t, \nu_t \in \mathcal{D}(\mathcal{G})$ by the $L^p(\Omega)$ regularity mentioned above. Next, regarding \ref{eq:geodesic_conv_G_sect:alternative_proof}, it is well-known that $\mathcal{G}$ is 0-geodesically convex, see \cite[Theorem 5.15]{MR1964483} and \cite[Theorem 1.3]{MR2118836}.\\

\noindent Finally, we verify \ref{eq:derivative_via_1st_var_sect:alternative_proof}. We assume $\Omega\subseteq \Rd$ (the necessary small adaptions for $\Omega=R\,\Td$ are discussed below, in Step 4). We consider the geodesic $\{\gamma_s\}_{s\in[0,1]}$ given by $\gamma_s= T_s^{\#} \mu_t$ where $T_s(x) = (1-s)\,x+s\,T(x)$, $\gamma_1 = \nu_t$ and by \cite[Theorem 5.30]{MR1964483} we know
\begin{equation}\label{eq:slope_functional_at_geodesic_pure_PME}
\frac{\diff}{\diff s} \mathcal{G}[\gamma_s]\big|_{s=0} =- \int_{\Omega} \mu_t^{n}\, (\Delta_A \psi - d)\diff x,
\end{equation}
where $\psi$ is the convex function such that $T = \nabla \psi$ and $\Delta_A$ is the Laplace operator in the Aleksandrov sense (for a convex function $f$, $\Delta f$ is always a nonnegative, locally finite measure since it is a nonnegative distribution; we write $\Delta f = \Delta_A f + \Delta_s f$ where $\Delta_A f$, $\Delta_s f$ are the absolutely continuous and singular parts of $\Delta f$, respectively). Now, for any $g \in C^{\infty}_c(\Omega)$, $g \geq 0$
\begin{multline*}
\int_{\Omega} g^n\, (\Delta_A \psi - d)\diff x \leq \int_{\Omega} g^{n}\, (\Delta \psi - d)\diff x  = \\ =
- \int_{\Omega} \nabla g^n\, (\nabla \psi - x) \diff x = 
-\frac{n}{n-1} \int_{\Omega} g\, \nabla g^{n-1}\, (T(x) - x) \diff x. 
\end{multline*} 
Then, we consider a sequence $\{g_k\}_{k}$ of smooth functions such that $g_k \geq 0$, $g_k \to \mu_t$ a.e., the supports of $\{g_k\}_{k}$ are uniformly bounded, $\nabla g_k^{n-\frac{1}{2}} \to \nabla \mu_t^{n-\frac{1}{2}}$ strongly in $L^2(\Omega)$ (this uses that $\nabla \mu_t^{n-\frac{1}{2}} \in L^2(\Omega)$ for a.e. $t$). Since $g_k$ has uniformly bounded support and $|T(x)|$ is bounded since it is controlled by the support of $\nu_t$, $\sqrt{g_k}\,(T(x)-x)\to \sqrt{\mu_t}\,(T(x)-x)$ in $L^2(\Omega)$ so that
\begin{equation}\label{eq:convergence_integrals_to_prove_A3_pure_PME}
-\frac{n}{n-1} \int_{\Omega} g_k\, \nabla g_k^{n-1}\, (T(x) - x) \diff x \to -\frac{n}{n-1} \int_{\Omega} \mu_t\, \nabla \mu^{n-1}_t\, (T(x) - x) \diff x \mbox{ as } k \to \infty.
\end{equation}
By Fatou lemma
$$
\int_{\Omega} \mu_t^{n}\, (\Delta_A \psi - d)\diff x \leq \liminf_{k \to \infty} \int_{\Omega} g^n_k\, (\Delta_A \psi - d)\diff x \leq -\frac{n}{n-1} \int_{\Omega} \mu_t\, \nabla \mu_t^{n-1}\, (T(x) - x) \diff x.
$$
It follows from \eqref{eq:slope_functional_at_geodesic_pure_PME} that
\begin{equation}\label{eq:bound_slope_PME_s0_real_case}
\frac{\diff}{\diff s} \mathcal{G}[\gamma_s]\big|_{s=0} \geq \int_{\Omega} \nabla \frac{\delta \mathcal{G}}{\delta \nu}[\mu_t] \,(T(x)-x)\diff \mu_t(x)
\end{equation}
and we can conclude by noting that $T(x)-x =-\nabla \varphi(x)$. Now, to obtain the second inequality \eqref{eq:slope_G_geodesic_general_assumption_2} for $\frac{\diff}{\diff s} \mathcal{G}[\gamma_s]\big|_{s=1}$ we consider the geodesic $\{\eta_s\}$ given by $\eta_s = S_s^{\#}\nu_t$, $S_s(x) =s\, S(x) + (1-s)\,x$, where $S$ is the optimal transport of $\nu_t$ onto $\mu_t$. Note that $S_s(T(x)) = T_{1-s}(x)$. Therefore, $\eta_s = \gamma_{1-s}$. By the reasoning above (this requires $\nabla \nu_t^{n-\frac{1}{2}} \in L^2(\Omega)$ for a.e. $t$ which is true by the bound on $\nabla \frac{\delta \mathcal{G}}{\delta \nu}[\nu_t] \,\sqrt{\nu_t}$ from \ref{ass:regularity_solutions:add_sect_alternative_proof})
\begin{equation}\label{eq:proof_the_slope_at_s1_by_inversion_pure_PME}
-\frac{\diff}{\diff s} \mathcal{G}[\gamma_s]\big|_{s=1}  = \frac{\diff}{\diff s} \mathcal{G}[\eta_s]\big|_{s=0} \geq \int_{\Omega} \nabla \frac{\delta \mathcal{G}}{\delta \nu}[\nu_t] \,(S(x)-x)\diff \nu_t(x).
\end{equation}
We conclude since $S(x)-x=-\nabla\phi(x)$. Hence, we may apply Theorem \ref{thm:simplified_for_2-Wasserstein} and H{\"o}lder's inequality in time to arrive at \eqref{eq:stability_PME_result_step1}.
 \\

\underline{Step 2. Proof of \eqref{eq:continuity_wrt_exponent_W2_PME_main_thm} for compactly supported initial conditions.} Starting from \eqref{eq:stability_PME_result_step1}, by triangle inequality we obtain
\begin{align*}
\mathcal{W}_2(\mu_t, \nu_t) \leq &\sqrt{t} \left(\int_0^t  \int_{\Omega} \left| m  - m  \, \mu_s^{n-m} \right|^2 |\nabla \mu_s|^2 \, \mu_s^{2m-3} \diff x  \diff s\right)^{\frac{1}{2}} \\
 &+  \sqrt{t} \left( \int_0^t  \int_{\Omega} \left| m-n \right|^2 |\nabla \mu_s|^2 \, \mu_s^{2n-3} \diff x  \diff s \right)^{\frac{1}{2}} =: X+Y.
\end{align*}
The term $Y$ can be estimated using \eqref{eq:energy_estimate_general_k} with $k = 2n - m -1$
$$
Y \leq  \frac{|m-n|\,\sqrt{t}}{\sqrt{m\,(2n-m-1)\,(2n-m)}} \left( \int_{\Omega} \mu_0^{2n-m} \diff x \right)^{\frac{1}{2}}.
$$
Concerning the term $X$, we have for all $x \geq 0$ and $\theta \in \R$ 
$$
|1- x^{\theta}| =|1-e^{\theta \, \log x}| \leq e^{\max(0, \theta \, \log x)}\, |\theta \, \log x| \leq |\theta| \, (1+x^{\theta})\, |\log x|.
$$
We want to use this estimate with $\theta = n-m$ and $x = \mu_s$. The sum $1+x^{\theta}$ leads to two terms
\begin{equation}\label{eq:definition_term_X1_in_the_estimate_PME}
X_1 = \sqrt{t}\,m\,|m-n|\, \left( \int_0^t  \int_{\Omega} |\nabla \mu_s|^2 \, \mu_s^{2m-3} \, |\log \mu_s|^2\diff x  \diff s\right)^{\frac{1}{2}},
\end{equation}
$$
X_2 = \sqrt{t}\,m\,|m-n|\, \left( \int_0^t  \int_{\Omega} |\nabla \mu_s|^2 \, \mu_s^{2n-3} \, |\log \mu_s|^2 \diff x  \diff s\right)^{\frac{1}{2}}.
$$
Let us focus on $X_1$. We introduce a parameter $\alpha \in [0, m-1)$ so that $\mu_s^{2m-3} \, |\log \mu_s|^2 \leq \mu_s^{2m-3-\alpha}\, \| \mu_s^{\alpha}\,|\log \mu_s|^2 \|_{L^{\infty}(\Omega)}$. Using \eqref{eq:energy_estimate_general_k} with $k = m-1-\alpha$, we deduce
$$
X_1 \leq \frac{m\,|m-n|\, \sqrt{t}}{\sqrt{m\,(m-1-\alpha)\,(m-\alpha)}} \, \| \mu_s^{\frac{\alpha}{2}}\,|\log \mu_s| \|_{L^{\infty}(\Omega)}\, \left(\int_{\Omega} \mu_0^{m-\alpha} \diff x \right)^{\frac{1}{2}}.
$$
Similarly, for $X_2$ we introduce a parameter $\beta \in [0, 2n-m-1)$ and we use \eqref{eq:energy_estimate_general_k} with $k=2n-m-1-\beta$ to obtain
$$
X_2 \leq \frac{m\,|m-n|\,\sqrt{t}}{\sqrt{m\,(2n-m-1-\beta)\,(2n-m-\beta)}} \,  \| \mu_s^{\frac{\beta}{2}}\,|\log \mu_s| \|_{L^{\infty}(\Omega)}\, \left(\int_{\Omega} \mu_0^{2n-m-\beta} \diff x \right)^{\frac{1}{2}}.
$$
It remains to estimate the terms $\| \mu_s^{\kappa}\,|\log \mu_s| \|_{L^{\infty}(\Omega)}$ for $\kappa = \frac{\alpha}{2}, \frac{\beta}{2}$ in terms of $\mu_0$. By the maximum principle, $\inf_{x\in \Omega} \mu_0 \leq \mu_s \leq \sup_{x\in\Omega} \mu_0$, and hence $\|\log \mu_s\|_{L^{\infty}(\Omega)} \leq \|\log \mu_0\|_{L^{\infty}(\Omega)}$. Moreover, the minimum of the function $x\mapsto x^{\kappa} \log x$ is $-\frac{1}{e\,\kappa}$, so that $|x^{\kappa} \log x|\leq \frac{1}{e\,\kappa} + x^{\kappa+1}$. Therefore, we can estimate $\| \mu_s^{\kappa}\,|\log \mu_s| \|_{\infty}$ in two different ways, depedning on whether $\log \mu_0 \in L^{\infty}(\Omega)$ or not:
\begin{equation}\label{eq:how_to_handle_log_PME}
\| \mu_s^{\kappa}\,|\log \mu_s| \|_{\infty} \leq \min\left(\frac{1}{e\,\kappa} + \|\mu_0\|_{L^{\infty}(\Omega)}^{\kappa+1},  \|\mu_0\|_{L^{\infty}(\Omega)}^{\kappa}\,\|\log \mu_0\|_{L^{\infty}(\Omega)} \right).
\end{equation}
Finally, note that $\log \mu_0 \notin L^{\infty}(\Omega)$ if $\Omega$ is unbounded since $\mu_0 \in L^1(\Omega)$ cannot be bounded away from zero. Hence, the second bound is finite only on bounded domains and when $\mu_0$ is bounded away from zero.\\

\underline{Step 3. Proof of \eqref{eq:continuity_wrt_exponent_W2_PME_main_thm} for general initial conditions.} 
This step is only necessary for $\Omega=\Rd$. Let $\mu_0 \in \mathcal{P}_2(\Rd) \cap L^{\infty}(\Rd)$. We consider the sequence $\{\mu_0^R\}_{R\geq 1}$ defined by $$
\mu_0^R =\frac{1}{\int_{B_R} \mu_0(x)\diff x}\, \mu_0 \, \mathds{1}_{B_R}. 
$$
We have $\int_{\Rd} |x|^2 \mu_0^R \diff x \to \int_{\Rd} |x|^2 \mu_0 \diff x$ (by dominated convergence), $\mu_0^R \to \mu_0$ in $\mathcal{W}_2$ and $L^p(\Rd)$ for all $p\in[1,\infty)$ and $\limsup_{R\to\infty}\|\mu_0^R \|_{L^{\infty}(\Rd)} \leq \|\mu_0 \|_{L^{\infty}(\Rd)}$. Denoting by $\mu_t^R$, $\nu_t^R$ the solutions to \eqref{eq:PME_1}--\eqref{eq:PME_2} with the same initial condition $\mu_0^R$, we have $\mathcal{W}_2(\mu^{R}_t, \nu^{R}_t) \leq C_{R}\,\sqrt{t}\, |m-n|$, where $C_R$ is defined by \eqref{eq:constant_continuity_wrt_exponent_W2_PME_main_thm} with $\mu_0$ replaced by $\mu_0^R$. The geodesic convexity implies again that $\mathcal{W}_2(\mu_t^R, \mu_t), \mathcal{W}_2(\nu_t^R, \nu_t) \to 0$ as $R\to \infty$ so $\mathcal{W}_2(\mu^{R}_t, \nu^{R}_t) \to \mathcal{W}_2(\mu_t, \nu_t)$. On the other hand, the constant $C_{R}$ converges by the $L^p(\Rd)$ convergence of $\mu_0^R$. Note that we do not need to worry about the $\|\log \mu_0^R\|_{L^{\infty}(\Rd)}$ term since it has to blow up anyway.\\

\underline{Step 4. Necessary adaptations for the case $\Omega=R\,\Td$.} We discuss here how to prove \ref{eq:derivative_via_1st_var_sect:alternative_proof} for $\Omega=R\,\Td$. We refer to Appendix \ref{rem:optimal_transport_periodic_domain} for the necessary background on the optimal transport on $R\,\Td$. Given $\mu_t, \nu_t \in \mathcal{P}_2(\Omega)\cap L^1(\Omega)$, the geodesic $\{\gamma_s\}_{s\in[0,1]}$ connecting $\mu_t$ and $\nu_t$ is defined as $\gamma_s = T_{s,P}^{\#}\,\mu_t$ with $T_{s,P}(x) = (x-s\,\nabla\varphi(x)) \mbox{ mod } 2R$, where the modulo operation is applied on each coordinate, see \eqref{eq:geodesic_torus}. The map $\nabla \varphi$ is the Kantorovich potential for the optimal transportation problem on $R\,\Td$ and it satisfies $\nabla \varphi(x) = x - T(x)$, where $T$ is the optimal transport between periodically extended measures $\mu_t$ and $\nu_t$, see \eqref{eq:Kantorovich_potential_periodic_domain}. We now consider $\widetilde{\mu}_t$ as a measure $\mu_t$ on $(-R,R]^d \subset \Rd$ and define the curve $\widetilde{\gamma}_s= T_s^{\#}\widetilde{\mu}_t$ where $T_s(x) = x-s\,\nabla\varphi(x)$. By \eqref{eq:bound_displacement_periodic_domain}, $\widetilde{\gamma}_s$ is supported at most in $(-2R,2R]^d$. By Lemma~\ref{lem:Lp_norm_of_geodesic_from_torus_to_Rd}
$$
\frac{1}{n-1} \int_{\Omega} |\gamma_s|^n \diff x = \frac{1}{n-1} \int_{(-2R,2R]^d} |\widetilde{\gamma}_s|^n \diff x.
$$
Note that $\{\widetilde{\gamma}_s\}_{s\in[0,1]}$ is a geodesic as in the classical theory on $\Rd$. Hence, we can compute as in \eqref{eq:bound_slope_PME_s0_real_case} above
$$
\frac{\diff}{\diff s} \mathcal{G}[\gamma_s]\big|_{s=0} = \frac{\diff}{\diff s} \mathcal{G}[\widetilde{\gamma}_s]\big|_{s=0} \geq \int_{\Omega} \nabla \frac{\delta \mathcal{G}}{\delta \nu}[\mu_t] \,(T(x)-x)\diff \mu_t(x) = 
-\int_{\Omega} \nabla \frac{\delta \mathcal{G}}{\delta \nu}[\mu_t] \,\nabla\varphi(x)\diff \mu_t(x). 
$$
For $\frac{\diff}{\diff s} \mathcal{G}[\gamma_s]\big|_{s=1}$, we can write $\frac{\diff}{\diff s} \mathcal{G}[\gamma_s]\big|_{s=1} = -\frac{\diff}{\diff s} \mathcal{G}[{\gamma}_{1-s}]\big|_{s=0}$. By Remark \ref{rem:inverse_geodesic_torus_from_mu1_to_mu0}, $\{\gamma_{1-s}\}_{s\in[0,1]}$ is the geodesic from $\nu_t$ to $\mu_t$ so we can use the reasoning above to compute
$$
\frac{\diff}{\diff s} \mathcal{G}[\gamma_s]\big|_{s=1} =  -\frac{\diff}{\diff s} \mathcal{G}[{\gamma}_{1-s}]\big|_{s=0} \leq \int_{\Omega} \nabla \frac{\delta \mathcal{G}}{\delta \nu}[\nu_t] \,\nabla\phi(x)\diff \nu_t(x),
$$
where $\phi$ is the Kantorovich potential in the optimal transportation problem from $\nu_t$ onto $\mu_t$. This concludes the proof of \ref{eq:derivative_via_1st_var_sect:alternative_proof}.
\end{proof}

\begin{proof}[Proof of Theorem \ref{thm:PME} (by using Theorem \ref{thm:main})] 
We first assume that $\Omega$ is a periodic or a bounded domain. We apply Theorem \ref{thm:main} with 
$$
\mathcal{A}:= \mathcal{A}_1 = \mathcal{A}_2 = \{\rho \in C^{\infty}(\Omega)\cap \mathcal{P}_2(\Omega) \mbox{ such that } \rho>0 \mbox{ on } \overline{\Omega} \mbox{ and } \rho \mbox{ satisfies \eqref{eq:Neuman_boundary_conditions_PME}}\}
$$ and with velocity fields
\begin{equation}\label{eq:veloctiy_fields_applications_second_theorem_PME_torus}
\bvmu[\mu] = -\frac{m}{m-1} \, \nabla \mu^{m-1}, \qquad \qquad \bvnu[\nu] = -\frac{n}{n-1} \, \nabla \nu^{n-1}.
\end{equation}
If $\mu_0 \in \mathcal{A}$, then $\mu_t, \nu_t \in \mathcal{A}$ so $\mu_t$, $\nu_t$ are smooth. Therefore, Assumption \ref{ass:BIG_regularity_of_sln_and_vf} and \ref{ass:abs_cont_first_map} are immediately satisfied. Regarding Assumption \ref{ass:semigroup_second_map}, it is known that the functional $\mathcal{G}$ (defined in \eqref{eq:functionals_F_G_for_PME}) is 0-geodesically convex by \cite[Theorem 5.15]{MR1964483} (for $\Rd$ and its subsets) and \cite[Theorem 1.3]{MR2118836} (for $R\,\Td$). Therefore, \eqref{eq:Wasserstein_Lipschitz-AC} is satisfied with $K(t) = 1$. Theorem \ref{thm:main} implies that
$$
\mathcal{W}_2(\mu_t, \nu_t) \leq  \int_0^t \left( \int_{\Omega} \left|\bvnu[\mu_s]- \bvmu[\mu_s]  \right|^2 \diff \mu_s  \right)^{1/2} \diff s \leq \sqrt{t}  \left( \int_0^t \int_{\Omega} \left|\bvnu[\mu_s]- \bvmu[\mu_s]  \right|^2 \diff \mu_s \diff s  \right)^{1/2}
$$
whenever $\mu_0 \in \mathcal{A}$. Plugging the velocities from \eqref{eq:veloctiy_fields_applications_second_theorem_PME_torus} yields \eqref{eq:stability_PME_result_step1}, i.e. the conclusion of Step~1 above. Moreover, Step 2 carries over verbatim as the computations are purely algebraic so that we obtain the desired estimate \eqref{eq:continuity_wrt_exponent_W2_PME_main_thm} valid for all initial conditions $\mu_0 \in \mathcal{A}$ and for $\Omega$ either bounded or periodic. It remains to remove this restriction and extend the estimate to the case of $\Rd$.\\

To this end, given $\mu_0 \in L^{\infty}(\Omega)\cap \mathcal{P}_2(\Omega)$, letting $\Omega_{\eps} = \{x\in \Omega: \mbox{dist}(x,\partial {\Omega})>2\eps\}$, we define the sequence
$$
\mu_0^{\eps} = N_{\eps}\,(1-\eps)\, \Big( \big(\big(\mu_0 - \essinf_{y\in \Omega} \mu_0(y)\big)\mathds{1}_{\Omega_{\eps}}\big)\ast\psi_{\eps} + \essinf_{y\in \Omega} \mu_0(y)  \Big) + \frac{\eps}{|\Omega|},
$$
where $\{\psi_{\eps}\}_{\eps\in(0,1)}$ is a standard mollifier supported in the ball $B_{\eps}$ and the normalization constant
$$
N_{\eps} := \frac{1}{\essinf_{y\in \Omega} \mu_0(y) \, (|\Omega| - |\Omega_{\eps}|) + \int_{\Omega_{\eps}}  \mu_0(x)\diff x} \to 1 \mbox{ as } \eps \to 0
$$
ensures that $\mu_0^{\eps}$ has mass 1. Moreover, $\mu_0^{\eps}$ is constant close to the boundary $\partial \Omega$ so it satisfies boundary conditions \eqref{eq:Neuman_boundary_conditions_PME} and so, $\mu_0^{\eps} \in \mathcal{A}$. The definition of $\mu_0^{\eps}$ is rather involved, as it must be smooth, constant near the boundary, bounded from below by both a positive constant and the essential infimum of $\mu_0$, and normalized to be a probability measure.\\

We write $\mu_t^{\eps}$, $\nu_t^{\eps}$ for solutions to \eqref{eq:PME_1}--\eqref{eq:PME_2} with initial condition $\mu_0^{\eps}$. By \eqref{eq:continuity_wrt_exponent_W2_PME_main_thm}, we have $\mathcal{W}_2(\mu_t^{\eps}, \nu_t^{\eps}) \leq C_{\eps}\,\sqrt{t}\,|m-n|$, where $C_{\eps}$ is defined by \eqref{eq:constant_continuity_wrt_exponent_W2_PME_main_thm} with $\mu_0$ replaced by $\mu_0^\eps$. The geodesic convexity of $\mathcal{F}$ and $\mathcal{G}$ in \eqref{eq:functionals_F_G_for_PME} implies that $\mathcal{W}_2(\mu_t^{\eps}, \mu_t), \mathcal{W}_2(\nu_t^{\eps}, \nu_t) \leq \mathcal{W}_2(\mu_0^{\eps}, \mu_0) \to 0$ which implies $\mathcal{W}_2(\mu_t^{\eps}, \nu_t^{\eps}) \to \mathcal{W}_2(\mu_t, \nu_t)$. Moreover, $\mu_0^{\eps} \to \mu_0$ in $L^p(\Omega)$ for all $p \in [1,\infty)$ so to prove convergence of the constant $C_{\eps}$ we only need to discuss the terms $ \|\mu_0^{\eps}\|_{L^{\infty}(\Omega)}$, $\|\log \mu_0^{\eps}\|_{L^{\infty}(\Omega)}$. Note that for all $x\in \Omega$
$$
N_{\eps}\,(1-\eps)\,\essinf_{y\in \Omega} \mu_0(y)  \leq \mu_0^{\eps}(x) \leq N_{\eps}\,(1-\eps)\, \|\mu_0\|_{L^{\infty}(\Omega)} + \frac{\eps}{|\Omega|},
$$ 
which implies $\limsup_{\eps \to 0} \|\log \mu_0^{\eps}\|_{L^{\infty}(\Omega)} \leq   \|\log \mu_0\|_{L^{\infty}(\Omega)}$, $\limsup_{\eps \to 0} \|\mu_0^{\eps}\|_{L^{\infty}(\Omega)} \leq   \|\mu_0\|_{L^{\infty}(\Omega)}$. It follows that $\limsup_{\eps\to0} C_{\eps}\leq C$, where $C$ is defined by \eqref{eq:constant_continuity_wrt_exponent_W2_PME_main_thm}. Hence, \eqref{eq:continuity_wrt_exponent_W2_PME_main_thm} holds for $\Omega$ bounded or periodic with a general initial condition.\\

Finally, we extend the estimate to the case $\Omega=\Rd$. First, we assume that the support of $\mu_0=\nu_0$ is a bounded set in $\Rd$. By the finite speed of propagation of \eqref{eq:PME_1}--\eqref{eq:PME_2}, $\mu_t$ and $\nu_t$ are supported in some ball $B_R$ for all $t\in [0,T]$. It follows that $\mu_t, \nu_t$ solve \eqref{eq:PME_1}--\eqref{eq:PME_2} in the ball $B_{2R}$ with Neumann boundary conditions and we can apply the reasoning above to prove the estimate. The general case $\mu_0 \in \mathcal{P}_2(\Rd) \cap L^{\infty}(\Rd)$ follows by an approximation as in Step 3 in the proof above.
\end{proof}

\begin{rem}\label{rem:extending_below_1_the_method_for_PME}
It is natural to ask to what extent these results can be extended to the case $m,n \leq 1$ since it is well-known that the geodesic convexity of the functional $\mathcal{G}[\nu] = \frac{1}{n-1}\int_{\Omega} |\nu|^n \diff x$ holds for $n \geq 1-\frac{1}{d}$ \cite[Examples~5.19]{MR1964483}. In both proofs above, the assumption $m,n > 1$ is mainly used to ensure that arbitrary solutions can be approximated by compactly supported ones. More precisely, in the first proof this property is used to establish \ref{eq:derivative_via_1st_var_sect:alternative_proof} (specifically, to prove \eqref{eq:convergence_integrals_to_prove_A3_pure_PME}) while in the second proof this is used to verify that Assumption \ref{ass:BIG_regularity_of_sln_and_vf} holds. It seems unlikely that the latter can be established for globally supported solutions, since the required regularity estimates usually blow up for large arguments. It is therefore natural to focus on proving \ref{eq:derivative_via_1st_var_sect:alternative_proof} in this setting. Another additional step would be to extend the estimate \eqref{eq:energy_estimate_general_k} to negative values of $k$ (this is required to control for instance the term $X_1$ in \eqref{eq:definition_term_X1_in_the_estimate_PME}, since its estimate relies on \eqref{eq:energy_estimate_general_k} with $k=m-1-\alpha$ for some $\alpha>0$). Such an extension is possible at the cost of estimating the quantity $\int_{\Omega} \mu_t^{k+1} \diff x$ which, for $k>-1$ not too negative, can be controlled in terms of the mass and the second moment (see e.g. \cite[Theorem~2]{lutwak2005crame}).
\end{rem}

\begin{proof}[Proof of Theorem \ref{thm:PME_incompress_limit}]
Since we work with the same PDEs \eqref{eq:PME_1}--\eqref{eq:PME_2}, the estimate \eqref{eq:stability_PME_result_step1} from Step 1 of the proof of Theorem \ref{thm:PME} is satisfied for solutions $\mu_t$, $\nu_t$ solving \eqref{eq:PME_1}--\eqref{eq:PME_2} with $\mu_0$ compactly supported. Hence,
$$
\mathcal{W}_2(\mu_t, \nu_t) \leq 
\sqrt{t}\, m  \left(\int_0^t \int_{\Omega}  |\nabla \mu_s|^2 \, \mu_s^{2m-3} \diff x \diff s \right)^{\frac{1}{2}} + \sqrt{t}\, n  \left( \int_0^t \int_{\Omega}  |\nabla \mu_s|^2 \, \mu_s^{2n-3} \diff x \diff s \right)^{\frac{1}{2}}.
$$
Applying \eqref{eq:energy_estimate_general_k} with $k = m-1$ (which requires $m \geq 1$) and $k = 2n-m-1$ (which requires $n \geq 1$, $n - m \geq 0$) we get
\begin{equation}\label{eq:estimate_large_m_n}
\begin{split}
\mathcal{W}_2(\mu_t, \nu_t) \leq  \, &\frac{\sqrt{t}\,m}{\sqrt{m\,(m-1)\,m}} \left(\int_{\Omega} \mu_0^{m} \diff x \right)^{1/2} \\
&+ \frac{\sqrt{t}\,n}{\sqrt{m\,(2n-m-1)\,(2n-m)}} \left(\int_{\Omega} \mu_0^{2n-m} \diff x\right)^{1/2}.
\end{split}
\end{equation}
We can extend this estimate for all initial conditions as in Step 3 in the first proof of Theorem~\ref{thm:PME}. We now write $\mu_{m,t}$ for the solution to \eqref{eq:PME_1} with exponent $m$ and initial condition $\mu_{m,0}$ (note that so far the initial condition was independent of the exponent). We also let $\widetilde{\mu}_{m,n,t}$ for the solution to \eqref{eq:PME_1} with exponent $m$ but with initial condition $\mu_{n,0}$. We apply \eqref{eq:estimate_large_m_n} with $\mu_t = \widetilde{\mu}_{m,n,t}$ and $\nu_t = \mu_{n,t}$ to obtain for $m<n$
\begin{multline*}
\mathcal{W}_2(\mu_{n,t} , \widetilde{\mu}_{m,n,t} ) \leq \\ \leq \frac{\sqrt{t}}{\sqrt{m-1}} \left(\int_{\Omega} \mu_{n,0}^{m} \diff x \right)^{1/2} + \frac{\sqrt{t}\,n}{\sqrt{m\,(2n -m-1)\,(2n-m)}} \left(\int_{\Omega} {\mu_{n,0}^{2n-m}} \diff x\right)^{1/2}.
\end{multline*}
By the geodesic convexity, $\mathcal{W}_2(\widetilde{\mu}_{m,n,t}, \mu_{m,t}) \leq \mathcal{W}_2(\mu_{n,0}, \mu_{m,0})$ so that
\begin{equation}\label{eq:estimate_to_say_it_is_cauchy_seq_incompressible}
\begin{split}
\mathcal{W}_2(\mu_{n,t} , {\mu}_{m,t} ) \leq  \,&\mathcal{W}_2(\mu_{n,0}, \mu_{m,0})
 + \frac{\sqrt{t}}{\sqrt{m-1}} \left(\int_{\Omega} \mu_{n,0}^{m} \diff x \right)^{1/2} \\ 
 &+ \frac{\sqrt{t}\, n}{\sqrt{m\,(2n-m-1)\,(2n-m)}} \left(\int_{\Omega} {\mu_{n,0}^{2n-m}} \diff x\right)^{1/2}.
\end{split}
\end{equation}
By interpolation we easily establish
\begin{equation}\label{eq:integrals_initial_data_large_exponents}
\int_{\Omega} \mu_{n,0}^{m} \diff x \leq \left(\int_{\Omega} \mu_{n,0}^{2n} \diff x \right)^{\frac{m-1}{2n-1}}, \quad \int_{\Omega} \mu_{n,0}^{2n-m} \diff x \leq   \left(\int_{\Omega} \mu_{n,0}^{2n} \diff x \right)^{\frac{2n-m-1}{2n-1}},
\end{equation}
so that assumption \eqref{ass:bound_integral_incompressible_limit} imply that both integrals stay bounded in the limit $m, n \to \infty$ (with $m < n$). Assuming additionaly $2<m<n$, we get that the term appearing in the denominator of the last quantity in \eqref{eq:estimate_to_say_it_is_cauchy_seq_incompressible} satisfies
$$
(2n-m-1)\,(2n-m) = n^2\, \Big(2-\frac{m+1}{n}\Big)\,\Big(2-\frac{m}{n}\Big)> \frac{n^2}{2}
$$ so that \eqref{eq:estimate_to_say_it_is_cauchy_seq_incompressible} and the assumption on initial conditions \eqref{ass:bound_integral_incompressible_limit} imply that $\{\mu_{m,t}\}_m$ is a Cauchy sequence in $(\mathcal{P}_2(\Omega), \mathcal{W}_2)$ and so, it has a limit $\mu_{\infty,t}$ \cite[Prop.~7.1.5]{MR2129498}. Plugging the estimates \eqref{eq:integrals_initial_data_large_exponents} into \eqref{eq:estimate_to_say_it_is_cauchy_seq_incompressible} and using \eqref{ass:bound_integral_incompressible_limit}, we can pass to the limit $n \to \infty$ in \eqref{eq:estimate_to_say_it_is_cauchy_seq_incompressible} to obtain the desired estimate
$$
\mathcal{W}_2(\mu_{m,t}, \mu_{\infty,t}) \leq \mathcal{W}_2(\mu_{m,0}, \mu_{\infty,0}) + \frac{\sqrt{t}}{\sqrt{m-1}} + \frac{\sqrt{2 \, t}}{\sqrt{m}} \, \mathcal{C}^{1/2}.
$$ 
We remark that passing to the limit $n \to \infty$ in the integral terms \eqref{eq:integrals_initial_data_large_exponents} was justified by \eqref{ass:bound_integral_incompressible_limit}: for each $\eps>0$, there is $N$ such that for all $n \geq N$ we have $\int_{\Omega} \mu_{n,0}^{2n} \diff x \leq \mathcal{C} +\eps$. For such $n$
$$
\limsup_{n \to \infty} \left(\int_{\Omega} \mu_{n,0}^{2n} \diff x \right)^{\frac{2n-m-1}{2n-m}} \leq 
\limsup_{n \to \infty} \, (\mathcal{C} +\eps)^{\frac{2n-m-1}{2n-m}} = \mathcal{C} +\eps
$$
and we conclude by the arbitrariness of $\eps$. Similar argument works for the other intergral.
\end{proof}

\section{Proof of Theorem \ref{thm:rate_of_conv_nonlocal_to_local}}\label{sect:applications_to_nonlocal2local}

\subsection{Proof of Theorem \ref{thm:rate_of_conv_nonlocal_to_local} (the case of $\R$)}

The proof is divided into two steps.\\

\underline{Step 1: Application of Theorem \ref{thm:simplified_for_2-Wasserstein}.} We will prove that
\begin{equation}\label{eq:main_proof_rate_result_step_2}
\mathcal{W}_2(\mu^{\eps}_{t}, \mu_t) \leq \sqrt{t}\, \left(\int_0^t \int_{\R} |\partial_x \mu_s \ast \omega_{\eps} - \partial_x \mu_s |^2 \mu_s \diff x \diff s \right)^{\frac{1}{2}}.
\end{equation}
To this end, we define two functionals
$$
\mathcal{F}[\mu] = \frac{1}{2} \int_{\R} \mu^2 \diff x, \qquad \mathcal{G}[\nu] = \frac{1}{2} \int_{\R} \nu \, \nu\ast\omega_{\eps} \diff x.
$$
Their first variations read
$$
\frac{\delta\mathcal{F}}{\delta \mu}[\mu] = \mu, \qquad \frac{\delta\mathcal{G}}{\delta \nu}[\nu] = \nu\ast\omega_{\eps}.
$$
We want to apply Theorem \ref{thm:simplified_for_2-Wasserstein} with the velocity field $\bvmu[\mu] = -\nabla\frac{\delta \mathcal{F}}{\delta \mu}[\mu]$ and with the functional $\mathcal{G}$. We proceed to verifying Assumption \ref{eq:assumption_simplified_formula}. Regarding condition \ref{ass:regularity_solutions:add_sect_alternative_proof}, functions $\partial_x\frac{\delta \mathcal{F}}{\delta \mu}[\mu_t] \,\sqrt{\mu_t}$, $\partial_x\frac{\delta \mathcal{G}}{\delta \nu}[\nu_t] \,\sqrt{\nu_t}$ belong to $L^2((0,T)\times\R)$ because of standard energy identities
\begin{equation}\label{eq:energy_dissipation_inequality_nonlocal_to_local_onedim}
\partial_t \mathcal{F}[\mu_t] + \int_{\R} \left| \partial_x\frac{\delta \mathcal{F}}{\delta \mu}[\mu_t] \,\sqrt{\mu_t} \right|^2 \diff x\leq 0, \quad \partial_t \mathcal{G}[\nu_t] + \int_{\R} \left| \partial_x\frac{\delta \mathcal{G}}{\delta \nu}[\nu_t] \,\sqrt{\nu_t} \right|^2 \diff x\leq 0.
\end{equation}
Moreover, function $\partial_x\frac{\delta \mathcal{G}}{\delta \nu}[\mu_t] \,\sqrt{\mu_t}$ is also in $L^2((0,T)\times\R)$ because $\partial_x \omega_{\eps} \in L^{\infty}(\R)$ so we can estimate
$$
\left\| \partial_x \mu_t\ast\omega_{\eps} \,\sqrt{\mu_t} \right\|_{L^2((0,T)\times\R)} \leq \sqrt{T} \, \|\partial_x \omega_{\eps}\|_{L^{\infty}(\R)}
$$
by conservation of mass for \eqref{eq:PDE_local_1D_eps}. Next, $\mu_t, \nu_t \in \mathcal{D}(\mathcal{G})$ because $ L^1(\Omega)\subset \mathcal{D}(\mathcal{G})$ since $\omega_{\eps}~\in~L^{\infty}(\Omega)$ and both solutions $\mu_t, \nu_t$ are clearly in $L^1(\Omega)$. Finally, $\mu_t, \nu_t \in \mathcal{P}_2(\R)$ thanks to the identities obtained by multiplying the respective PDEs by $|x|^2$
$$
\partial_t \int_{\R} \mu_t |x|^2 \diff x \leq 2\, \left( \int_{\R} \left| \partial_x\frac{\delta \mathcal{F}}{\delta \mu}[\mu_t] \,\sqrt{\mu_t} \right|^2 \diff x \right)^{\frac{1}{2}}\, \left(\int_{\R} \mu_t |x|^2 \diff x \right)^{\frac{1}{2}},
$$
$$
\partial_t \int_{\R} \nu_t |x|^2 \diff x \leq 2\, \left( \int_{\R} \left| \partial_x\frac{\delta \mathcal{G}}{\delta \nu}[\nu_t] \,\sqrt{\nu_t} \right|^2 \diff x \right)^{\frac{1}{2}}\, \left(\int_{\R} \nu_t |x|^2 \diff x \right)^{\frac{1}{2}},
$$
and \eqref{eq:energy_dissipation_inequality_nonlocal_to_local_onedim}. Regarding \ref{eq:geodesic_conv_G_sect:alternative_proof}, it is well-known that the convexity condition \ref{ass:kernel_convexity} on the kernel implies 0-geodesic convexity of $\mathcal{G}$, see \cite[Proposition~2.7]{MR2935390}. Finally, to verify \ref{eq:derivative_via_1st_var_sect:alternative_proof}, we take the geodesic $\{\gamma_s\}_{s\in[0,1]}$ given by $\gamma_s= T_s^{\#} \mu_t$ where $T_s(x) = (1-s)\,x+s\,T(x)$, $\gamma_1 = \nu_t$ and we compute
\begin{equation}\label{eq:computing_the_slope_n2l_case_R}
\frac{\mathcal{G}[\gamma_s] -\mathcal{G}[\gamma_0]}{s} = \frac{1}{2} \int_{\R}\int_{\R} \frac{\omega_\eps(T_s(x)-T_s(y)) - \omega_{\eps}(x-y)}{s} \gamma_0(x) \gamma_0(y) \diff x \diff y.
\end{equation}
Since $\partial_x \omega_{\eps} \in L^{\infty}(\R)$, $\omega_{\eps}$ is Lipschitz continuous and it is a.e. differentiable. Therefore, the integrand converges for a.e. $x, y$. Since we can estimate
$$
\left|\frac{\omega_\eps(T_s(x)-T_s(y)) - \omega_{\eps}(x-y)}{s}\right| \leq \|\partial_x\omega_{\eps}\|_{L^{\infty}(\R)} \,  (|T(x)-x| + |T(y) -y|),
$$
we can use dominated convergence theorem (this only uses $\gamma_0\in L^1(\R)$ and $\mathcal{W}_2(\gamma_0,\gamma_1)<\infty$) to pass to the limit $s \to 0$ and deduce
$$
\frac{\diff}{\diff s}\mathcal{G}[\gamma_s] \Big|_{s=0} = \frac{1}{2}\int_{\R}\int_{\R} \partial_x\omega_{\eps}(x-y)\, \big((T(x)-x)-(T(y)-y)\big)\, \gamma_0(x) \gamma_0(y) \diff x \diff y.
$$
Using $\omega(x)=\omega(-x)$ we get $\partial_x\omega(x) = -\partial_x\omega(-x)$ and we symmetrize the expression above to get
$$
\frac{\diff}{\diff s}\mathcal{G}[\gamma_s] \Big|_{s=0} = \int_{\R}\int_{\R} \partial_x\omega_{\eps}(x-y)\, (T(x)-x)\, \gamma_0(x) \gamma_0(y) \diff x \diff y = \int_{\R} \partial_x  \frac{\delta\mathcal{G}}{\delta \nu}[\mu_t] \, (T(x)-x) \diff \mu_t(x).
$$
We conclude by noting $T(x)-x=-\nabla\varphi(x)$. The case $s=1$ is studied analogously, using $\gamma_1 = \nu_t \in L^1(\R)$ and $\mathcal{W}_2(\gamma_0,\gamma_1)<\infty$. \\

Having checked Assumption \ref{eq:assumption_simplified_formula}, Theorem \ref{thm:simplified_for_2-Wasserstein} with $\lambda = 0$, $\mu_t:=\mu_t$ and $\nu_t:=\mu_t^{\eps}$ implies \eqref{eq:main_proof_rate_result_step_2}.\\

\underline{Step 2: The result for initial conditions bounded from below.} We assume additionally that $\mu_0$ is smooth and $\mu_0 \geq \sigma \, e^{-|x|^2}$ for some $\sigma>0$. We will deduce from \eqref{eq:main_proof_rate_result_step_2} that
\begin{equation}\label{eq:estimate_W2_eps_regularized_initial_condition}
\mathcal{W}^2_2(\mu_t^{\eps}, \mu_t) \leq 2\,t \, \eps^2\,  \Big(5 t\,\| \partial_x \mu_0\|^2_{L^{\infty}(\R)}\,  + 
\frac{1}{2}\, \|\mu_0\|_{L^{\infty}(\R)} \Big) \, \| |\partial_{x}^2 \mu_0|^{-} \|_{\mathcal{M}(\R)}\, \| \omega\,|y|^2 \|_{L^1(\R)}.
\end{equation}
Since the initial condition is bounded from below, $\mu_t$ is smooth. Starting from \eqref{eq:main_proof_rate_result_step_2}, we first apply Jensen's inequality to deduce
$$
\mathcal{W}^2_2(\mu^{\eps}_t, \mu_t) \leq  t  \int_0^t  \int_{\R} \int_{\R} | \partial_x \mu_s(x-y) - \partial_x \mu_s(x) |^2 \, \mu_s(x) \, \omega_{\eps}(y) \diff x \diff y \diff s := \mathcal{I}.
$$
To estimate $\mathcal{I}$, we split the integration set into two subsets
$$
A := \{(s,x,y):  \mu_s(x) \leq 2\,\|\partial_x \mu_0\|_{L^{\infty}} \, |y| \},  \quad
B := \{(s,x,y): \mu_s(x) > 2\, \|\partial_x \mu_0\|_{L^{\infty}} \, |y| \},   
$$
resulting in two integrals $\mathcal{I}_{A}$ and $\mathcal{I}_{B}$. We start with a general observation
\begin{equation}\label{eq:differentiation_in_terms_of_2nd_derivative}
 \partial_x \mu_s(x-y) - \partial_x \mu_s(x) = \int_0^1 \partial_x^2 \mu_s(x - \theta \, y) \, y \diff \theta,
\end{equation}
so that we can estimate
\begin{multline*}
 | \partial_x \mu_s(x-y) - \partial_x \mu_s(x) |^2 = 2\, \|\partial_x \mu_s\|_{L^{\infty}(\R)}\,  | \partial_x \mu_s(x-y) - \partial_x \mu_s(x) | \leq \\ \leq 2\, \|\partial_x \mu_0\|_{L^{\infty}(\R)} \int_0^1 |\partial_x^2 \mu_s(x - \theta \, y)| \, |y| \diff \theta,  
\end{multline*}
where we used $\| \partial_x \mu_s \|_{L^{\infty}(\R)} \leq \| \partial_x \mu_0 \|_{L^{\infty}(\R)}$ (Lemma \ref{lem:estimates_PME_1D}). Hence, using also the bound defining set $A$
\begin{align*}
\mathcal{I}_{A} &\leq  4\, \|\partial_x \mu_0\|^2_{L^{\infty}(\R)}\, t  \int_0^t  \int_{\R} \int_{\R} \int_0^1 |\partial_x^2 \mu_s(x - \theta \, y)| \, |y|^2  \, \omega_{\eps}(y)  \diff \theta \diff x \diff y \diff s \\
&= 4\, \|\partial_x \mu_0\|^2_{L^{\infty}(\R)}\, t  \int_0^t  \int_{\R}  |\partial_x^2 \mu_s(x)|  \diff x \diff s \int_{\R} \omega_{\eps}(y)\, |y|^2 \diff y \\
&\leq 8\, \|\partial_x \mu_0\|^2_{L^{\infty}(\R)}\, t^2 \, \eps^2 \, \| |\partial^2_x\mu_0|^- \|_{\mathcal{M}(\R)} \, \| \omega\, |y|^2 \|_{L^1(\R)},
\end{align*}
where we used \ref{est:laplacian_pressure} in Lemma \ref{lem:estimates_PME_1D} to estimate $\int_{\R}  |\partial_x^2 \mu_s(x)|  \diff x$ and the change of variables $ \int_{\R} \omega_{\eps}(y)\, |y|^2 \diff y = \eps^2 \int_{\R} \omega(y) \, |y|^2 \diff y$. Concerning $\mathcal{I}_{B}$, we use \eqref{eq:differentiation_in_terms_of_2nd_derivative} together with Jensen's inequality to obtain
$$
| \partial_x \mu_s(x-y) - \partial_x \mu_s(x) |^2 \leq \int_0^1 \left| \partial_x^2 \mu_s(x - \theta \, y) \right|^2 \, |y|^2 \diff \theta.
$$
Plugging this into the definition of $\mathcal{I}_B$ we obtain
\begin{equation}\label{eq:term_I_B_estimate_rate_nonlocal}
\mathcal{I}_B \leq t  \int  \int \int_{B} \int_0^1 \left| \partial_x^2 \mu_s(x - \theta \, y) \right|^2  \, \mu_s(x) \, \omega_{\eps}(y) \,  |y|^2  \diff \theta \diff x \diff y \diff s.
\end{equation}
We want to use the $L^2$ estimate on $\sqrt{\mu_s} \,\partial^2_x \mu_s$, cf. \ref{est:weighted_pressure} in Lemma \ref{lem:estimates_PME_1D}. We note that thanks to the definition of $B$ and $\| \partial_x \mu_s \|_{L^{\infty}(\R)} \leq \| \partial_x \mu_0 \|_{L^{\infty}(\R)}$ we have
$$
\mu_s(x - \theta \, y) \geq \mu_s(x) - \|\partial_x \mu_s \|_{L^{\infty}(\R)} \, |y| \, \theta >
2\, \|\partial_x \mu_0\|_{L^{\infty}(\R)} \, |y|- \|\partial_x \mu_0\|_{L^{\infty}(\R)} \, |y| = \|\partial_x \mu_0\|_{L^{\infty}(\R)} \, |y|.
$$
Hence, we can estimate
\begin{multline*}
\mu_s(x) =\mu_s(x - \theta \, y)\, \frac{\mu_s(x)}{\mu_s(x - \theta \, y)} \leq \mu_s(x - \theta \, y)\left(1 + \frac{|\mu_s(x) -\mu_s(x - \theta \, y)| }{\mu_s(x - \theta \, y)}\right)\\
\leq \mu_s(x - \theta \, y)\, \left(1 + \frac{\|\partial_x \mu_0\|_{L^{\infty}(\R)} \, \theta\, |y|}{\|\partial_x \mu_0\|_{L^{\infty}(\R)} \, |y|}\right) \leq 2\, \mu_s(x - \theta \, y)
\end{multline*}
and plugging this into \eqref{eq:term_I_B_estimate_rate_nonlocal} we arrive at
\begin{align*}
\mathcal{I}_B &\leq 2\, t  \int_0^t  \int_{\R} \int_{\R} \int_0^1 \left| \partial_x^2 \mu_s(x - \theta \, y) \right|^2  \, \mu_s(x-\theta\,y) \, \omega_{\eps}(y) \,  |y|^2  \diff \theta \diff x \diff y \diff s \\
&= 2\,t \, \| \sqrt{\mu_s} \,\partial^2_x \mu_s \|^2_{L^2((0,t)\times \R)} \, \| \omega_{\eps}\,|y|^2 \|_{L^1(\R)}\\
&\leq 2\,t \, \eps^2\,  \Big( t\,\| \partial_x \mu_0\|^2_{L^{\infty}(\R)}\,  + 
\frac{1}{2}\, \|\mu_0\|_{L^{\infty}(\R)} \Big) \, \| |\partial_{x}^2 \mu_0|^{-} \|_{\mathcal{M}(\R)}\, \| \omega\,|y|^2 \|_{L^1(\R)},
\end{align*}
where in the last line we used \ref{est:weighted_pressure} in Lemma \ref{lem:estimates_PME_1D}. Collecting estimates on $\mathcal{I}_A$ and $\mathcal{I}_B$ we arrive at \eqref{eq:estimate_W2_eps_regularized_initial_condition}.\\ 

\underline{Step 3: Conclusion of the argument.} We now extend \eqref{eq:estimate_W2_eps_regularized_initial_condition} for all $\mu_0 \in \mathcal{P}_2(\R)\cap L^{\infty}(\R)$ with $\partial_x \mu_0 \in L^{\infty}(\R)$, $|\partial^2_x \mu_0|^- \in \mathcal{M}(\R)$. We consider a sequence of initial conditions defined by $\mu_0^n = (1-\frac{1}{n}) \mu_0\ast g_n + \frac{1}{n}\, \frac{1}{\sqrt{2\pi}} e^{-\frac{|x|^2}{2}}$ where $g_n$ is the density of the Gaussian distribution with variance $\frac{1}{n}$. Clearly, $\mu_0^n$ is smooth, $\mu_0^n \rightharpoonup \mu_0$ narrowly (in duality with $C_b(\R)$) and $\int_{\R} |x|^2 \diff \mu_0^n(x) \to \int_{\R} |x|^2 \diff \mu_0(x)$ (this follows by expanding $|x|^2 = |x-y|^2 + |y|^2 + 2(x-y)y$ in the convolution $\mu_0\ast g_n$) so that by \cite[Theorem 5.11]{MR3409718}
\begin{equation}\label{eq:proof_nonlocal_2_local_approx_init_cond}
\mathcal{W}_2(\mu_0^n, \mu_0) \to 0 \mbox{ as } n\to \infty.
\end{equation}
We let $\mu^{\eps,n}_t$ and $\mu^n_t$ to be the solutions to \eqref{eq:PDE_nonlocal_1D_eps} and \eqref{eq:PDE_local_1D_eps}, respectively, with initial condition $\mu_0^n$. From \eqref{eq:estimate_W2_eps_regularized_initial_condition} we deduce
\begin{equation}\label{eq:estimate_W2_eps_regularized_initial_condition_fixed_n}
\mathcal{W}^2_2(\mu_t^{\eps,n}, \mu^n_t) \leq 2\,t \, \eps^2\,  \Big(5 t\,\| \partial_x \mu^n_0\|^2_{L^{\infty}(\R)}\,  + 
\frac{1}{2}\, \|\mu^n_0\|_{L^{\infty}(\R)} \Big) \, \| |\partial_{x}^2 \mu^n_0|^{-} \|_{\mathcal{M}(\R)}\, \| \omega\,|y|^2 \|_{L^1(\R)}.
\end{equation}
The well-known geodesic convexity of both functionals $\mathcal{F}$ and $\mathcal{G}$ implies that
$$
\mathcal{W}_2(\mu_t^{n},\mu_t)\leq \mathcal{W}_2(\mu_0^n, \mu_0), \qquad \qquad \mathcal{W}_2(\mu_t^{\eps,n},\mu_t^{\eps})\leq \mathcal{W}_2(\mu_0^n, \mu_0),
$$
so by \eqref{eq:proof_nonlocal_2_local_approx_init_cond} and the triangle inequality we obtain $\lim_{n\to \infty} \mathcal{W}^2_2(\mu_t^{\eps,n}, \mu^n_t) = \mathcal{W}_2^2(\mu_t^{\eps}, \mu_t)$. Finally, direct computation shows that 
$$
\limsup_{n\to \infty} \|  \mu^n_0\|_{L^{\infty}(\R)} \leq \| \mu_0\|_{L^{\infty}(\R)}, \qquad \qquad \limsup_{n\to \infty} \| \partial_x \mu^n_0\|_{L^{\infty}(\R)} \leq \| \partial_x \mu_0\|_{L^{\infty}(\R)},
$$
$$
 \limsup_{n\to \infty}  \| |\partial^2_{x} \mu^n_0|^- \|_{\mathcal{M}(\R)} \leq \| |\partial^2_{x} \mu_0|^- \|_{\mathcal{M}(\R)}.
$$
Indeed, the first two inequalities follow by Young's convolution inequality while the third uses additionally convexity of the function $x \mapsto |x|^{-}$ and Jensen's inequality $| \partial^2_x \mu_0\ast g_n |^- \leq |\partial^2_x \mu_0|^{-} \ast g_n$. It follows that we can take the limit $n \to \infty$ in \eqref{eq:estimate_W2_eps_regularized_initial_condition_fixed_n} and conclude the proof.

\subsection{Necessary modifications for the case of periodic domain}\label{subsect:nonlocal2local_periodic_domain}

Theorem \ref{thm:rate_of_conv_nonlocal_to_local} can be also established on a periodic domain $\mathbb{T}$ identified with $(-1,1]$ after appropriate periodization of the kernel $\omega$. We briefly discuss here the necessary modifications. Given a kernel $\omega$ on $\R$ satisfying Assumption \ref{ass:kernel} and additionally
\begin{equation}\label{eq:assumptions_growth_for_periodization}
|\omega(x)|, |\partial_x\omega(x)| \leq \frac{C_{\omega}}{1+|x|^{\alpha}}, \qquad \alpha>1,
\end{equation}
we define its periodic extension
\begin{equation}\label{eq:omega_kernel_periodic_domain}
\omega^{\mathbb{T}}_{\eps}(x) = \sum_{m\in 2\mathbb{Z}} \omega_{\eps}(x+m)= \frac{1}{\eps}\sum_{m\in 2\mathbb{Z}} \omega\Big(\frac{x+m}{\eps}\Big).
\end{equation}
This allows to avoid ambiguity when considering $\omega_{\eps}$ since $\omega$ does not need to be compactly supported. It also ensures that $\int_{\mathbb{T}} \omega_{\eps}^{\mathbb{T}}(x) \diff x = 1$ if $\omega$ satisfies \ref{ass:kernel_mass}.\\

Let us verify that Assumption \ref{ass:kernel} is satisfied by $\omega^{\mathbb{T}}_{\eps}$ on $\mathbb{T}$, or at least that the assumptions can be suitably modified. Concerning \ref{ass:kernel_regularity}, by \eqref{eq:assumptions_growth_for_periodization}, we have for $x\in[-1,1]$
$$
|\omega^{\mathbb{T}}_{\eps}(x)|\leq \frac{1}{\eps} \sum_{m \in 2\mathbb{Z}} \frac{C_{\omega}}{1+\left|\frac{x+m}{\eps}\right|^{\alpha}} \leq 
\frac{1}{\eps} \sum_{m \in 2\mathbb{Z}\setminus\{0\}} \frac{C_{\omega}}{1+\frac{||m|-1|^{\alpha}}{\eps^{\alpha}}} + \frac{C_{\omega}}{\eps}< \infty.
$$
The same argument works for $\partial_x \omega^{\mathbb{T}}_{\eps}$ again by \eqref{eq:assumptions_growth_for_periodization}. To prove the bound on $\partial^2_x \omega^{\mathbb{T}}_{\eps}$ in $\mathcal{M}(\mathbb{T})$, we consider a periodic $\psi \in C_b^{\infty}(\mathbb{T})$ and we compute
\begin{multline*}
\int_{-1}^1 \psi(x)\, \partial^2_x \omega^{\mathbb{T}}_{\eps}(x) \diff x = \frac{1}{\eps} \sum_{m\in 2\mathbb{Z}} \int_{-1}^1 \omega\left( \frac{x+m}{\eps} \right) \,\partial^2_x \psi(x) \diff x = \\ =  \sum_{m\in 2\mathbb{Z}} \int_{\frac{-1+m}{\eps}}^{\frac{1+m}{\eps}} \omega(y) \, (\partial^2_x \psi)(\eps\,y -m) \diff y= \int_{\R} \omega(y)\, (\partial^2_x \psi)(\eps\,y) \diff y = \int_{\R} \omega_{\eps}(y) \, \partial^2_x \psi(y).
\end{multline*}
Taking supremum over all $\|\psi\|_{\infty}\leq 1$, we obtain $\|\partial^2_x \omega^{\mathbb{T}}_{\eps} \|_{\mathcal{M}(\mathbb{T})} = \|\partial^2_x\omega_{\eps} \|_{\mathcal{M}(\mathbb{R})}$.\\

Concerning condition \ref{ass:kernel_convexity} on $\R$, it is necessary to ensure the geodesic convexity of the functional $\mathcal{G}[\mu] = \frac{1}{2} \int_{\R} \mu\, \mu \ast \omega_{\eps} \diff x$. Let us see that it is also true on $\mathbb{T}$ with the kernel defined by \eqref{eq:omega_kernel_periodic_domain}. Reader may consult Appendix \ref{rem:optimal_transport_periodic_domain} for a general theory of optimal transport on periodic domains. First, given two absolutely continuous measures $\mu, \nu \in \mathcal{P}_2(\mathbb{T})$, there is always a monotone map $T:[-1,1]\to[-2,2]$ such that $T(x)\mbox{ mod }2$ is the optimal transport map from $\mu$ to $\nu$ \cite[Section 2]{MR2558330}. The map $T$ is obtained as the optimal transport between the periodic extensions of measures $\mu$ and $\nu$ to $\R$ \cite{MR1711060}. In particular, $T(x)-x$ is a periodic function. The geodesic connecting $\mu$ and $\nu$ in $\mathcal{P}_2(\mathbb{T})$ is defined via 
\begin{equation}\label{eq:geodesics_one_dim_circle_wasserstein}
\gamma_s=(T_{s,P})^{\#}\mu, \qquad T_{s,P}(x) = (x + s (T(x) - x)) \mbox{ mod }2.
\end{equation}
We need to prove that $[0,1]\ni s\mapsto \mathcal{G}^{\mathbb{T}}[\gamma_s] = \frac{1}{2} \int_{\mathbb{T}} \gamma_s\, \gamma_s \ast \omega^{\mathbb{T}}_{\eps} \diff x$ is convex. Using that $\omega^{\mathbb{T}}_{\eps}$ is periodic, we can ignore modulo in \eqref{eq:geodesics_one_dim_circle_wasserstein} and write
\begin{multline*}
\mathcal{G}^{\mathbb{T}}[\gamma_s]  = 
 \frac{1}{2} \int_{-1}^1 \int_{-1}^1 \omega^{\mathbb{T}}_{\eps}(T_{s,P}(x) - T_{s,P}(y)) \diff \gamma_0(x) \diff \gamma_0(y) =  \\ =
  \frac{1}{2} \int_{-1}^1 \int_{-1}^1 \omega^{\mathbb{T}}_{\eps}\big(s(T(x) -T(y)) + (1-s)(x-y)\big) \diff \gamma_0(x) \diff \gamma_0(y).
\end{multline*}
Now, we write $\omega^{\mathbb{T}}_{\eps}$ as a sum \eqref{eq:omega_kernel_periodic_domain} and we consider each summand separately. For $m = 0$, we split the integral into two sets $x \leq y$ and $x > y$. By monotonicity of $T$, $x\leq y$ implies $T(x) \leq T(y)$ so $s(T(x) -T(y)) + (1-s)(x-y)\leq 0$ and we can use convexity of $\omega$ to deduce
\begin{multline*}
  \frac{1}{2} \int \int_{x\leq y} \omega\left(\frac{s(T(x) -T(y)) + (1-s)(x-y)}{\eps}\right) \diff \gamma_0(x) \diff \gamma_0(y) \leq \\
  \frac{s}{2} \int \int_{x\leq y} \omega\left(\frac{T(x) -T(y)}{\eps}\right) \diff \gamma_0(x) \diff \gamma_0(y) +  \frac{1-s}{2} \int \int_{x\leq y} \omega\left(\frac{x-y}{\eps}\right) \diff \gamma_0(x) \diff \gamma_0(y).
\end{multline*}
Similar argument works for $x>y$. For $m\neq 0$, we observe that $T(-1) + 2 = T(1)$ by periodicity of $T(x)-x$ so by monotonicity $|T(x)-T(y)|\leq 2$. As $|x-y|\leq 2$, we deduce $|s(T(x) -T(y)) + (1-s)(x-y)|\leq 2$. Hence, for $m \neq 0$, $s(T(x) -T(y)) + (1-s)(x-y) - m$ does not change the sign and we can use convexity of $\omega(x)$ for either $x\leq 0$ or $x\geq 0$ to prove
\begin{multline*}
  \frac{1}{2} \int_{-1}^1\int_{-1}^1 \omega\left(\frac{s(T(x) -T(y))\! + \!(1-s)(x-y)\!-\!m}{\eps}\right) \diff \gamma_0(x) \diff \gamma_0(y) \leq \\
  \frac{s}{2} \int_{-1}^1 \int_{-1}^1 \omega\left(\frac{T(x)\! -\!T(y)\!-\!m}{\eps}\right) \diff \gamma_0(x) \diff \gamma_0(y) +  \frac{1-s}{2} \int_{-1}^1 \int_{-1}^1 \omega\left(\frac{x\!-y\!-\!m}{\eps}\right) \diff \gamma_0(x) \diff \gamma_0(y).
\end{multline*}
Now, summing up all the terms we obtain
$$
\mathcal{G}^{\mathbb{T}}[\gamma_s]  \leq \frac{s}{2} \int_{-1}^1 \int_{-1}^1 \omega^{\mathbb{T}}_{\eps}(T(x) -T(y)) \diff \gamma_0(x) \diff \gamma_0(y) + \frac{1-s}{2} \int_{-1}^1 \int_{-1}^1 \omega^{\mathbb{T}}_{\eps}(x-y) \diff \gamma_0(x) \diff \gamma_0(y).
$$
Using periodicity of $\omega^{\mathbb{T}}_{\eps}$, we identify the first term as $s\,\mathcal{G}^{\mathbb{T}}[\gamma_1]$ while the second is equal to $(1-s)\, \mathcal{G}^{\mathbb{T}}[\gamma_0]$, concluding the proof.\\

Finally, we comment that since $\omega_{\eps}^{\mathbb{T}}$ is periodic, there is no issue with computing the slopes $\frac{\diff}{\diff s}\mathcal{G}^{\mathbb{T}}[\gamma_s] \Big|_{s=0} $, $\frac{\diff}{\diff s}\mathcal{G}^{\mathbb{T}}[\gamma_s] \Big|_{s=1}$ since the modulo in the definition of $T_{s,P}$ in \eqref{eq:geodesics_one_dim_circle_wasserstein} can be ignored by the periodicity of $\omega_{\eps}^{\mathbb{T}}$ and we can argue as in \eqref{eq:computing_the_slope_n2l_case_R} using $\partial_x \omega^{\mathbb{T}}_{\eps} \in L^{\infty}(\mathbb{T})$. Since on $\mathbb{T}$, the Kantorovich potential satisfy $-\nabla \varphi(x) = T(x)-x$ (see \eqref{eq:Kantorovich_potential_periodic_domain}), we directly arrive at \ref{eq:derivative_via_1st_var_sect:alternative_proof} for $s=0$ by following the strategy on $\mathbb{R}$. The case $s=1$ follows by applying Remark \ref{rem:inverse_geodesic_torus_from_mu1_to_mu0} and the formula $\frac{\diff}{\diff s}\mathcal{G}^{\mathbb{T}}[\gamma_s] \Big|_{s=1} = -\frac{\diff}{\diff s}\mathcal{G}^{\mathbb{T}}[\gamma_{1-s}] \Big|_{s=0}$. \\

Therefore, with kernel $\omega^{\mathbb{T}}_\eps$ defined via \eqref{eq:omega_kernel_periodic_domain} where $\omega$ satisfies Assumption \ref{ass:kernel} and additionally \eqref{eq:assumptions_growth_for_periodization}, Theorem \ref{thm:rate_of_conv_nonlocal_to_local} can be established in the same way on $\mathbb{T}$. We remark that the main estimates to establish \eqref{eq:main_thm_rate_nonlocal_to_local}, namely Lemma \ref{lem:estimates_PME_1D}, is also valid on $\mathbb{T}$.

\section{Proof of Theorem \ref{thm:stability_estimate_AD_eq}}\label{sect:aggregation-diffusion}

\begin{proof}[Proof of Theorem \ref{thm:stability_estimate_AD_eq}]

\underline{Step 1: Useful a priori estimates.}  First, since $\mu_0 \in L^{1}(\Omega)\cap L^{\infty}(\Omega)$, $\mu_t$ is in $L^p(\Omega)$ for all $p\in[1,\infty]$ and all $t\geq0$. This is implied by the following inequality valid for all $k \geq 0$
\begin{equation}\label{eq:Lp_norm_sol_A-D_not_integrated_in_time_STEP1}
\partial_t \left(L_{\mu}^{-k}(t)\int_{\Omega} \mu_t^{k+1} \diff x \right) +  L_{\mu}^{-k}(t)\,m\,k\,(k+1) \int_{\Omega} \mu_t^{m+k-2}\, |\nabla \mu_t|^2 \diff x \leq 0,
\end{equation}
where $L_\mu(t) := \exp(t\, (\|\Delta V_{\mu}\|_{L^{\infty}(\Omega)} +\|\Delta W_{\mu}\|_{L^{\infty}(\Omega)}))$. To prove \eqref{eq:Lp_norm_sol_A-D_not_integrated_in_time_STEP1} we use the PDE \eqref{eq:aggregation_diffusion_1} to compute
\begin{align*}
\partial_t \int_{\Omega} &\mu_t^{k+1} \diff x = -(k+1) \int_{\Omega} \nabla \mu_t^m \, \nabla \mu_t^k \diff x - k \int_{\Omega} \nabla (V_{\mu}+W_{\mu}\ast \mu_t)\, \nabla \mu_t^{k+1} \diff x\\
&\leq -m\,k\,(k+1) \int_{\Omega} \mu_t^{m+k-2}\, |\nabla \mu_t|^2 \diff x + k\, (\|\Delta V_{\mu}\|_{L^{\infty}(\Omega)} +\|\Delta W_{\mu}\|_{L^{\infty}(\Omega)}) \int_{\Omega} \mu_t^{k+1} \diff x,
\end{align*}
where for the second term we integrated by parts and used $\|\Delta W_{\mu} \ast \mu_s\|_{L^{\infty}(\Omega)} \leq \|\Delta W_{\mu}\|_{L^{\infty}(\Omega)}$. Multiplying by $L_{\mu}^{-k}(t)$, we arrive at \eqref{eq:Lp_norm_sol_A-D_not_integrated_in_time_STEP1}. In particular, by sending $k \to \infty$, we obtain 
\begin{equation}\label{eq:auxiliary_L_inf_estimate}
\|\mu_t\|_{L^{\infty}(\Omega)} \leq L_{\mu}(t)\, \|\mu_0\|_{L^{\infty}(\Omega)}.    
\end{equation}
We will also need a lower bound which is a well-known consequence of the comparison principle (see \cite[Lemma~5.2]{MR4842678})
\begin{equation}\label{eq:auxiliary_lower_bound_on_solution}
\mu_t(x) \geq \essinf_{x \in \Omega} \mu_0(x) \, L_{\mu}(t)^{-1}.
\end{equation}

Similarly, by symmetry to \eqref{eq:Lp_norm_sol_A-D_not_integrated_in_time_STEP1}, we have
\begin{equation}\label{eq:Lp_norm_sol_A-D_not_integrated_in_time_STEP1_for_2nd_sol}
\partial_t \left(L_{\nu}^{-k}(t)\int_{\Omega} \nu_t^{k+1} \diff x \right) +  L_{\nu}^{-k}(t)\,n\,k\,(k+1) \int_{\Omega} \nu_t^{n+k-2}\, |\nabla \nu_t|^2 \diff x \leq 0,
\end{equation}
where $L_\nu(t) := \exp(t\, (\|\Delta V_{\nu}\|_{L^{\infty}(\Omega)} +\|\Delta W_{\nu}\|_{L^{\infty}(\Omega)}))$. We also recall the tail estimate which is proved in Appendix \ref{app:general_aggr_diff_PDE} (see \eqref{eq:moment_estimate_second_lemma_finite_speed}):
\begin{equation}\label{eq:tail_estimate_main_section_aggregation_diffusion}
\left( \int_{\Omega} |x|^2 \, \mu_t(x)\diff x \right)^{1/2} \leq \left( \int_{\Omega} |x|^2 \, \mu_0(x)\diff x \right)^{1/2} + \sqrt{t}\, \left(\mathcal{F}[\mu_0]+ C^0_{V_{\mu}} + C^0_{W_{\mu}} \right)^{1/2},
\end{equation}
\begin{equation}
\label{eq:tail_estimate_main_section_aggregation_diffusion_2_for_nu_t}
\left( \int_{\Omega} |x|^2 \, \nu_t(x)\diff x \right)^{1/2} \leq \left( \int_{\Omega} |x|^2 \, \mu_0(x)\diff x \right)^{1/2} + \sqrt{t}\, \left(\mathcal{G}[\mu_0]+ C^0_{V_{\nu}} + C^0_{W_{\nu}} \right)^{1/2},
\end{equation}
where the functionals $\mathcal{F}$ and $\mathcal{G}$ are defined as follows:
\begin{equation}\label{eq:functional_F_proof_stability_agg_diff_step1}
\mathcal{F}[\mu] = \frac{1}{m-1} \int_{\Omega} \mu^m \diff x + \int_{\Omega}   V_{\mu}\, \mu \diff x + \frac{1}{2}\int_{\Omega}  W_{\mu}\ast\mu \, \mu \diff x,
\end{equation}
\begin{equation}\label{eq:functional_G_proof_stability_agg_diff_step1}
\mathcal{G}[\nu] = \frac{1}{n-1} \int_{\Omega} \nu^n \diff x + \int_{\Omega}   V_{\nu}\, \nu \diff x + \frac{1}{2}\int_{\Omega}  W_{\nu}\ast\nu \, \nu \diff x. 
\end{equation}

\underline{Step 2: Application of Theorem \ref{thm:simplified_for_2-Wasserstein}.} We additionally assume that $\mu_0$ is compactly supported (this is empty assumption on the torus), so $\mu_t$ and $\nu_t$ stay compactly supported for all times by Theorem \ref{thm:aggr-diff-compact-supp}. We also assume that $\Omega=\Rd$ and we discuss necessary adaptations for the case $\Omega=R\,\Td$ in Step 5. We will prove 
\begin{equation}\label{eq:stability_A-D_result_step1}
\begin{split}
\mathcal{W}_2(\mu_t, \nu_t) \leq &\left(\frac{1-e^{-2\Lambda t}}{2\Lambda}\right)^{1/2} \left( \int_0^t  \int_{\Omega} \left| m  - n \, \mu_s^{n-m} \right|^2 |\nabla \mu_s|^2 \, \mu_s^{2m-3} \diff x  \diff s\right)^{\frac{1}{2}}\\
& + \left(\frac{1-e^{-2\Lambda t}}{2\Lambda}\right)^{1/2}  \left( \int_0^t  \int_{\Omega} \left| \nabla V_{\mu}-\nabla V_{\nu} \right|^2 \, \mu_s \diff x  \diff s\right)^{\frac{1}{2}}\\
&+  \left(\frac{1-e^{-2\Lambda t}}{2\Lambda}\right)^{1/2}  \left(  \int_0^t \int_{\Omega} \left| \nabla (W_{\mu}- W_{\nu})\ast\mu_s \right|^2 \, \mu_s \diff x  \diff s\right)^{\frac{1}{2}},
\end{split} 
\end{equation}
where $\Lambda$ is defined in \eqref{eq:ass_convexity_stability_aggr_diffusion_exponent_gamma}. To this end, for the functionals $\mathcal{F}$ and $\mathcal{G}$ defined in \eqref{eq:functional_F_proof_stability_agg_diff_step1}--\eqref{eq:functional_G_proof_stability_agg_diff_step1}, we compute their first variations 
$$
\frac{\delta \mathcal{F}}{\delta \mu}[\mu]=\frac{m}{m-1}\mu^{m-1} + V_\mu+ W_{\mu}\ast\mu, \qquad \frac{\delta \mathcal{G}}{\delta \nu}[\nu]= \frac{n}{n-1}\nu^{n-1} + V_\nu+ W_{\nu}\ast\nu,
$$
and we will verify Assumption \ref{eq:assumption_simplified_formula} with $\mathcal{G}$ as above and $\bvmu[\mu]= -\nabla\frac{\delta \mathcal{F}}{\delta \mu}[\mu]$. Regarding condition \ref{ass:regularity_solutions:add_sect_alternative_proof}, both functions $\nabla\frac{\delta \mathcal{F}}{\delta \mu}[\mu_t] \,\sqrt{\mu_t}$, $\nabla \frac{\delta \mathcal{G}}{\delta \nu}[\nu_t] \,\sqrt{\nu_t}$ belong to $L^2((0,T)\times\Omega)$ because of standard energy identities
\begin{equation}\label{eq:energy_dissipation_for_aggr_diffusion_main_proof}
\partial_t \mathcal{F}[\mu_t] + \int_{\Omega} \left| \nabla \frac{\delta \mathcal{F}}{\delta \mu}[\mu_t] \,\sqrt{\mu_t} \right|^2 \diff x\leq 0, \quad \partial_t \mathcal{G}[\nu_t] + \int_{\Omega} \left| \nabla \frac{\delta \mathcal{G}}{\delta \nu}[\nu_t] \,\sqrt{\nu_t} \right|^2 \diff x\leq 0.
\end{equation}
Moreover, the function $\nabla \frac{\delta \mathcal{G}}{\delta \nu}[\mu_t] \,\sqrt{\mu_t} = n\,\mu_t^{n-\frac{3}{2}}\,\nabla\mu_t + \sqrt{\mu_t}\,\nabla V_{\nu} + \sqrt{\mu_t}\,\nabla W_{\nu}\ast\mu_t$ is also in $L^2((0,T)\times\Omega)$. Indeed, the first term is controlled thanks to \eqref{eq:Lp_norm_sol_A-D_not_integrated_in_time_STEP1} with $k=2n-m-1$. Next, $\sqrt{\mu_t}\,\nabla V_{\nu}$ can be estimated in $L^2((0,T)\times\Omega)$ by the tail estimate \eqref{eq:tail_estimate_main_section_aggregation_diffusion} and the growth condition \ref{ass_item:growth_conditions_V_W_main_thm}. Finally, for $\sqrt{\mu_t}\,\nabla W_{\nu}\ast\mu_t$ we use \ref{ass_item:growth_conditions_V_W_main_thm} to estimate
\begin{equation}\label{eq:estimate_nabla_W_mollifies_mu_Step_2}
\begin{split}
\sqrt{\mu_t}\,|\nabla W_{\nu}\ast\mu_t| &\leq \sqrt{\mu_t} \, \int_{\Omega} (C^3_{W_{\nu}} + C^4_{W_{\nu}}\,|x-y|)\diff \mu_t(y) \\
&\leq \sqrt{\mu_t} \,\left(C^3_{W_{\nu}} + C^4_{W_{\nu}} \,|x| + C^4_{W_{\nu}} \,\left(\int_{\Omega} |y|^2 \diff \mu_t(y)\right)^{\frac{1}{2}}\right),
\end{split}
\end{equation}
so we can conclude again by \eqref{eq:tail_estimate_main_section_aggregation_diffusion} and we deduce $\nabla \frac{\delta \mathcal{G}}{\delta \nu}[\mu_t] \,\sqrt{\mu_t} \in L^2((0,T)\times\Omega)$ as desired. Finally, $\mu_t, \nu_t \in \mathcal{D}(\mathcal{G})$ because by 
\eqref{eq:Lp_norm_sol_A-D_not_integrated_in_time_STEP1} and \eqref{eq:Lp_norm_sol_A-D_not_integrated_in_time_STEP1_for_2nd_sol}, $\mu_t, \nu_t \in L^p(\Omega)$ for all $p \in [1,\infty]$, they are compactly supported while $V_\nu$, $W_\nu$ are locally bounded by \ref{ass_item:growth_conditions_V_W_main_thm} so all the integrals in $\mathcal{G}[\mu_t]$, $\mathcal{G}[\nu_t]$ are finite. Note that this argument shows that each integrand in $\mathcal{G}[\mu_t]$ and $\mathcal{G}[\nu_t]$ belongs to $L^1(\Omega)$.\\

Regarding \ref{eq:geodesic_conv_G_sect:alternative_proof}, we will prove that $\mathcal{G}$ is $\Lambda$-geodesically convex, where $\Lambda$ is defined in \eqref{eq:ass_convexity_stability_aggr_diffusion_exponent_gamma}. Indeed, we split
\begin{equation}\label{eq:functional_G_aggr-diff_splitting_for_3_functioonals}
\mathcal{G}_1[\nu] =   \frac{1}{n-1} \int_{\Omega} \nu^n \diff x, \quad \mathcal{G}_2[\nu] = \int_{\Omega}   V_{\nu}\, \nu \diff x, \quad \mathcal{G}_3[\nu] = \frac{1}{2}\int_{\Omega}  W_{\nu}\ast\nu \, \nu \diff x.
\end{equation}
From \cite[Theorem 5.15]{MR1964483} and \cite{CMV03}, $\mathcal{G}_1$ is 0-geodesically convex while $\mathcal{G}_2$ is $c_{V_{\nu}}$-geodesically convex. Regarding $\mathcal{G}_3$, when $c_{W_{\nu}}\leq 0$ it is $c_{W_{\nu}}$-geodesically convex while when $c_{W_{\nu}}>0$ it is 0-geodesically convex. Moreover, when $c_{W_{\nu}}>0$ and $\mathcal{G}_3$ is restricted to the space of measures $\nu$ such that $\int_{\Omega}x \diff \nu(x)$ is constant, it is also $c_{W_{\nu}}$-geodesically convex. The latter can be applied when $V=0$ since then a simple computation, exploiting symmetry of $W_{\mu}$ and $W_{\nu}$, shows that the maps $t\mapsto \int_{\Omega}x \diff \mu_t(x)$ and $t\mapsto \int_{\Omega}x \diff \nu_t(x)$ are constant. In particular, the geodesic $\{\gamma_s\}_{s\in[0,1]}$ connecting $\mu_t$ and $\nu_t$ satisfies
$$
\int_{\Omega} x \diff \gamma_s(x) = \int_{\Omega} (s\,T(x) + (1-s)x)\diff \mu_t(x) = s \int_{\Omega} x\diff \nu_t(x) + (1-s)\int_{\Omega} x \diff \mu_t(x) = \int_{\Omega} x \diff \mu_0(x),
$$
where $T^{\#}\mu_t = \nu_t$. Therefore, when $V=0$, we can restrict the reasoning to measures $\nu$ such that $\int_{\Omega} x\diff \nu(x) = \int_{\Omega} x \diff \mu_0(x)$ and exploit better convexity of $\mathcal{G}_3$. Summarizing the three scenarios described above, we arrive at the value of $\Lambda$ in \eqref{eq:ass_convexity_stability_aggr_diffusion_exponent_gamma} and \ref{eq:geodesic_conv_G_sect:alternative_proof} is proved.\\

Finally, we prove that \ref{eq:derivative_via_1st_var_sect:alternative_proof} is satisfied. First, arguing as in Step 1 in the proof of Theorem~\ref{thm:PME} we obtain
$$
\frac{\diff}{\diff s} \mathcal{G}_1[\gamma_s] \Big|_{s=0} \geq -\int_{\Omega} \nabla \frac{\delta \mathcal{G}_1}{\delta \nu}[\mu_t]\, \nabla\varphi(x) \diff \mu_t(x),
$$
where $\varphi(x)$ is a Kantorovich potential corresponding to the optimal transport of $\mu_t$ onto $\nu_t$ (this reasoning used only the fact that $\nabla \mu_t^{n-\frac{1}{2}} \in L^2(\Omega)$ for a.e. $t$ which we know by \eqref{eq:Lp_norm_sol_A-D_not_integrated_in_time_STEP1} applied with $k=2n-m-1$ and that $\mu_t$, $\nu_t$ were compactly supported which we know by Theorem~\ref{thm:aggr-diff-compact-supp}). Next, for $\mathcal{G}_2$ and $\mathcal{G}_3$ we can use directly \cite[Theorem 5.30]{MR1964483} or \cite{CMV03} to obtain
$$
\frac{\diff}{\diff s} \mathcal{G}_2[\gamma_s] \Big|_{s=0} \geq -\int_{\Omega} \nabla V(x)\,\nabla\varphi(x) \diff\mu_t(x) = -\int_{\Omega} \nabla \frac{\delta \mathcal{G}_2}{\delta \nu}[\mu_t]\,\nabla \varphi(x) \diff\mu_t(x),
$$
$$
\frac{\diff}{\diff s} \mathcal{G}_3[\gamma_s] \Big|_{s=0} \geq -\int_{\Omega} \int_{\Omega} \nabla W(x-y)\,\nabla\varphi(x) \diff\mu_t(y) \diff\mu_t(x) = -\int_{\Omega} \nabla \frac{\delta \mathcal{G}_3}{\delta \nu}[\mu_t] \nabla \varphi(x) \diff\mu_t(x).
$$
The assumptions of \cite[Theorem 5.30]{MR1964483}, namely that $\mu_t V_{\nu}$, $\nu_t V_{\nu}$, $\mu_t W_{\nu}\ast \mu_t$, $\nu_t W_{\nu}\ast \nu_t$ are in $L^1(\Omega)$, are satisfied as this was shown when proving that $\mu_t, \nu_t \in \mathcal{D}(\mathcal{G})$ in the discussion of \ref{ass:regularity_solutions:add_sect_alternative_proof} above.\\

Summing up the inequalities above we arrive at \eqref{eq:slope_G_geodesic_general_assumption_1}. To see \eqref{eq:slope_G_geodesic_general_assumption_2}, we argue as in the proof of Theorem \ref{thm:PME} (see \eqref{eq:proof_the_slope_at_s1_by_inversion_pure_PME}) by considering the geodesic $\{\gamma_{1-s}\}_{s\in[0,1]}$ which connects $\nu_t$ and $\mu_t$. Since all the properties of $\mu_t$ that we used are also satisfied by $\nu_t$ (the compact support is true by Theorem \ref{thm:aggr-diff-compact-supp} while the bound $\nabla \nu_t^{n-\frac{1}{2}}$ in $L^2(\Omega)$ for a.e. $t$ is true by \eqref{eq:Lp_norm_sol_A-D_not_integrated_in_time_STEP1_for_2nd_sol} with $k=n-1$), we can apply the reasoning above to the geodesic $\{\gamma_{1-s}\}_{s\in[0,1]}$ and deduce \eqref{eq:slope_G_geodesic_general_assumption_2}.\\

Hence, we can apply Theorem \ref{thm:simplified_for_2-Wasserstein} with $\lambda = \Lambda$. Therefore,
$$
\partial_t\left(e^{\Lambda \,t} \, \mathcal{W}_2(\mu_t, \nu_t)\right) e^{-\Lambda \,t} \leq \left( \int_{\Omega} \left| \nabla\frac{\delta \mathcal{G}}{\delta \nu}[\mu_t] - \nabla\frac{\delta \mathcal{F}}{\delta \mu}[\mu_t] \right|^2 \diff \mu_t(x) \right)^{\frac{1}{2}}.
$$
By triangle inequality in $L^2(\Omega)$ equipped with measure $\mu_t$ we can estimate
\begin{multline*}
\partial_t\left(e^{\Lambda \,t} \, \mathcal{W}_2(\mu_t, \nu_t)\right) e^{-\Lambda \,t} \leq  \left( \int_{\Omega} \left|m-n\,\mu_t^{n-m}  \right|^2\, |\nabla \mu_t|^2\, \mu_t^{2m-3} \diff x  \right)^{\frac{1}{2}} \\ + \left( \int_{\Omega} \left| \nabla V_{\mu} - \nabla V_{\nu}  \right|^2 \diff \mu_t(x) \right)^{\frac{1}{2}} + \left( \int_{\Omega} \left| \nabla (W_{\mu} - W_{\nu})\ast \mu_t  \right|^2 \diff \mu_t(x) \right)^{\frac{1}{2}}.
\end{multline*}
To arrive at \eqref{eq:stability_A-D_result_step1}, we multiply the equation by $e^{\Lambda \,t}$, integrate in time, apply H{\"o}lder's inequality and use the integral $\int_0^t e^{2\Lambda s} \diff s = \frac{1}{2\Lambda}(e^{2 \Lambda t}-1)$. \\

\underline{Step 3: Proof of \eqref{eq:continuity_wrt_exponent_W2_aggr-diff_main_thm} for compactly supported $\mu_0$.} Comparing \eqref{eq:continuity_wrt_exponent_W2_aggr-diff_main_thm} and \eqref{eq:stability_A-D_result_step1}, we see it remains to estimate $\int_0^t  \int_{\Omega} \left| m  - n \, \mu_s^{n-m} \right|^2 |\nabla \mu_s|^2 \, \mu_s^{2m-3} \diff x  \diff s$. To this end, we first deduce from \eqref{eq:Lp_norm_sol_A-D_not_integrated_in_time_STEP1} that for all $k>0$
\begin{equation}\label{eq:a_priori_estimate_power_density_A-D}
m\, k \, (k+1) \int_0^t \int_{\Omega} \mu_s^{m+k-2}\, |\nabla \mu_s|^2 \diff x \diff s \leq L_{\mu}^{k}(t) \, \int_{\Omega} \mu_0^{k+1},
\end{equation}
This is obtained by integrating in time \eqref{eq:Lp_norm_sol_A-D_not_integrated_in_time_STEP1} and using $L_{\mu}^{-k}(t) \leq L_{\mu}^{-k}(s)$ for $s\in[0,t]$.\\

Now, we observe that the term $\int_0^t  \int_{\Omega} \left| m  - n \, \mu_s^{n-m} \right|^2 |\nabla \mu_s|^2 \, \mu_s^{2m-3} \diff x  \diff s$ has been already estimated in Step 2 of the proof of Theorem \ref{thm:PME} for $\mu_t$ being the solution of $\partial_t \mu_t = \Delta \mu_t^m$. Now, the reasoning be easily generalized to the general aggregation-diffusion equation \eqref{eq:aggregation_diffusion_1} by replacing \eqref{eq:energy_estimate_general_k} with \eqref{eq:a_priori_estimate_power_density_A-D}. Writing $m  - n \, \mu_s^{n-m} = (m-n)\,\mu_s^{n-m} + m-m\,\mu_s^{n-m}$, we estimate
\begin{align*}
&\left(\int_0^t  \int_{\Omega} \left| m  - n \, \mu_s^{n-m} \right|^2 |\nabla \mu_s|^2 \, \mu_s^{2m-3} \diff x  \diff s\right)^{\frac{1}{2}} \leq \\
&\leq \left( \int_0^t  \int_{\Omega} \left| m-n \right|^2 |\nabla \mu_s|^2 \, \mu_s^{2n-3} \diff x  \diff s \right)^{\frac{1}{2}} + \left(\int_0^t  \int_{\Omega} \left| m  - m  \, \mu_s^{n-m} \right|^2 |\nabla \mu_s|^2 \, \mu_s^{2m-3} \diff x  \diff s\right)^{\frac{1}{2}} \\
& \leq |m-n| \left( \int_0^t  \int_{\Omega} |\nabla \mu_s|^2 \, \mu_s^{2n-3} \diff x  \diff s \right)^{\frac{1}{2}} \!+ m\,|m-n| \left( \int_0^t  \int_{\Omega} |\nabla \mu_s|^2 \, \mu_s^{2m-3} \, |\log \mu_s|^2\diff x  \diff s\right)^{\frac{1}{2}}   \\
& \phantom{\leq \, } + m\,|m-n|  \left( \int_0^t  \int_{\Omega} |\nabla \mu_s|^2 \, \mu_s^{2n-3} \, |\log \mu_s|^2 \diff x  \diff s\right)^{\frac{1}{2}} =: Y+X_1 + X_2,
\end{align*}
where terms $X_1$, $X_2$ arise from estimating $|1- \mu_s^{n-m}| \leq |n-m| \, (1+\mu_s^{n-m})\, |\log \mu_s|$. Now, we apply \eqref{eq:a_priori_estimate_power_density_A-D} with $k= 2n-m-1$ to estimate
$$
Y \leq  \frac{|m-n|\, L_{\mu}(t)^{\frac{2n-m-1}{2}}}{\sqrt{m\,(2n-m-1)\,(2n-m)}} \left( \int_{\Omega} \mu_0^{2n-m} \diff x \right)^{\frac{1}{2}}.
$$
Next, for $X_1$, $X_2$ we choose parameters $\alpha\in [0, m-1)$, $\beta \in [0,2n-m-1)$, we write
$$
\mu_s^{2m-3} \, |\log\mu_s|^2=\mu_s^{\alpha}\, |\log\mu_s|^2\mu_s^{2m-3-\alpha}, \qquad \mu_s^{2n-3} \, |\log\mu_s|^2= \mu_s^{\beta}\, |\log\mu_s|^2 \mu_s^{2n-3-\beta},
$$
and we use \eqref{eq:a_priori_estimate_power_density_A-D} with $k=m-1-\alpha$, $k = 2n-m-1-\beta$ to obtain
$$
X_1 \leq \frac{m\,|m-n|\, L_{\mu}(t)^{\frac{m-1-\alpha}{2}}}{\sqrt{m\,(m-1-\alpha)\,(m-\alpha)}} \, \| \mu_s^{\frac{\alpha}{2}}\,|\log \mu_s| \|_{L^{\infty}(\Omega)}\, \left(\int_{\Omega} \mu_0^{m-\alpha} \diff x \right)^{\frac{1}{2}},
$$
$$
X_2 \leq \frac{m\,|m-n|\, L_{\mu}(t)^{\frac{2n-m-1-\beta}{2}}}{\sqrt{m\,(2n-m-1-\beta)\,(2n-m-\beta)}} \,  \| \mu_s^{\frac{\beta}{2}}\,|\log \mu_s| \|_{L^{\infty}(\Omega)}\, \left(\int_{\Omega} \mu_0^{2n-m-\beta} \diff x \right)^{\frac{1}{2}}.
$$
The terms $\| \mu_s^{\frac{\alpha}{2}}\,|\log \mu_s| \|_{L^{\infty}(\Omega)},\| \mu_s^{\frac{\beta}{2}}\,|\log \mu_s| \|_{L^{\infty}(\Omega)}$ can be estimated as in \eqref{eq:how_to_handle_log_PME} but we have to take into account additional terms in the maximum principle \eqref{eq:auxiliary_L_inf_estimate} and \eqref{eq:auxiliary_lower_bound_on_solution}. More precisely, as in \eqref{eq:how_to_handle_log_PME}, we can estimate using \eqref{eq:auxiliary_L_inf_estimate}
$$
\|\mu_s^{\kappa} \, \log \mu_s \|_{L^{\infty}(\Omega)} \leq \frac{1}{e\,\kappa} + \| \mu_s^{\kappa+1} \|_{L^{\infty}(\Omega)} \leq \frac{1}{e\,\kappa} + L_{\mu}^{1+\kappa} \, \| \mu_0 \|^{\kappa+1}_{L^{\infty}(\Omega)},    
$$
where $\kappa = \frac{\alpha}{2},\frac{\beta}{2}$. Alternatively, we can use additionally the lower bound \eqref{eq:auxiliary_lower_bound_on_solution} to estimate
$$
\|\mu_s^{\kappa} \, \log \mu_s \|_{L^{\infty}(\Omega)} \leq \| \mu_s^{\kappa} \|_{L^{\infty}(\Omega)} \, \| \log \mu_{s}\|_{L^{\infty}(\Omega)} \leq L_{\mu}^{\kappa} \, \| \mu_0 \|^{\kappa}_{L^{\infty}(\Omega)}\,(\| \log\mu_0\|_{L^{\infty}(\Omega)} + \log L_{\mu}(t)).    
$$
In particular, since $L_{\mu}\geq 1$,
$$
L_{\mu}^{-\kappa}\, \|\mu_s^{\kappa} \, \log \mu_s \|_{L^{\infty}(\Omega)} \leq \min\left(\frac{1}{e\,\kappa} + L_{\mu} \, \| \mu_0 \|^{\kappa+1}_{L^{\infty}(\Omega)} ,  \| \mu_0 \|^{\kappa}_{L^{\infty}(\Omega)}\,(\| \log\mu_0\|_{L^{\infty}(\Omega)} + \log L_{\mu}(t))  \right).
$$
To unify the remaining powers of $L_{\mu}$ in the estimates of $X_1$ and $X_2$ we note that, using $L_{\mu}(t)\geq 1$ and $n\geq m$, we have
$$
L_{\mu}(t)^{\frac{m-1}{2}} \leq  L_{\mu}(t)^{\frac{2n-m-1}{2}},
$$
so collecting the estimates above, we arrive at \eqref{eq:continuity_wrt_exponent_W2_aggr-diff_main_thm}.\\

\underline{Step 4: Proof of \eqref{eq:continuity_wrt_exponent_W2_aggr-diff_main_thm} for general $\mu_0$.} This step is only necessary for the case $\Omega=\Rd$. Let $\mu_0^R = \frac{1}{\int_{B_R} \mu_0(x)\diff x}\, \mu_0 \, \mathds{1}_{B_R}$ and let $\mu_t^R$, $\nu_t^R$ be the corresponding, compactly supported solutions. By $\lambda$-geodesic convexity of functionals $\mathcal{F}$ and $\mathcal{G}$ for some $\lambda$ (implied by \ref{ass:hessian_of_potential}), we have $\mathcal{W}_2(\mu_t^R, \mu_t) \to 0$, $\mathcal{W}_2(\nu_t^R, \nu_t) \to 0$ as $R\to \infty$. Moreover, 
\begin{equation}\label{eq:two_properties_approx_sequence_mu_N_agg-diff}
\sup_{R>1} \sup_{s\in[0,T]} \int_{\Rd} \int_{\Rd} |x|^2 \diff\mu^R_s < \infty, \qquad \sup_{R>1} \sup_{s\in[0,T]} \int_{\Rd} \left|\mu^R_s\right|^m \diff x < \infty.
\end{equation}
The first property follows by \eqref{eq:tail_estimate_main_section_aggregation_diffusion} while the second by \eqref{eq:Lp_norm_sol_A-D_not_integrated_in_time_STEP1} with $k=m-1$.\\

By Step 3, \eqref{eq:continuity_wrt_exponent_W2_aggr-diff_main_thm} is satisfied for each $R$ and we only need to explain how to pass to the limit in the terms $\mathcal{W}_2(\mu_t^R, \nu_t^R)$, $C_R$, $\int_0^t  \int_{\Omega} \left| \nabla V_{\mu}-\nabla V_{\nu} \right|^2 \mu^R_s \diff x  \diff s$, $\int_0^t \int_{\Omega} \left| \nabla (W_{\mu}- W_{\nu})\ast\mu^R_s \right|^2 \mu^R_s \diff x  \diff s$ appearing in \eqref{eq:continuity_wrt_exponent_W2_aggr-diff_main_thm}. Here, $C_R$ is a constant defined by \eqref{eq:constant_continuity_wrt_exponent_W2_PME_main_thm} with $\mu_0$ replaced by $\mu_0^R$ and with $C_{\kappa}$ given now by \eqref{eq:constant_c_kappa_for_aggr_diff_PDE}. First, $\mathcal{W}_2(\mu_t^R, \nu_t^R)\to \mathcal{W}_2(\mu_t, \nu_t)$ by triangle inequality and convergences recalled above. Second, convergence of constant $C_R$ has been analyzed in Step~3 in the proof of Theorem~\ref{thm:PME} (minor differences in the definition of $C_{\kappa}$ do not change the argument).\\

We now discuss the term $\int_0^t  \int_{\Omega} \left| \nabla V_{\mu}-\nabla V_{\nu} \right|^2 \, \mu^R_s \diff x  \diff s$. It is sufficient to prove
\begin{equation}\label{eq:Step5_convergence_bdd_domain_the_whole_space_potential_terms_fixed_time}
  \int_{\Omega} \left| \nabla V_{\mu}-\nabla V_{\nu} \right|^2 \, \mu^R_s \diff x  \to  \int_{\Omega} \left| \nabla V_{\mu}-\nabla V_{\nu} \right|^2 \, \mu_s \diff x
\end{equation}
for all $s\in(0,t)$ and then argue by the dominated convergence for the integral in time (condition \ref{ass_item:growth_conditions_V_W_main_thm} implies that $\left| \nabla V_{\mu}-\nabla V_{\nu} \right|^2 \leq \widetilde{C}(1 + |x|^2)$ for some constant $\widetilde{C}$ and the resulting majorant is integrable in time by \eqref{eq:two_properties_approx_sequence_mu_N_agg-diff}). To see \eqref{eq:Step5_convergence_bdd_domain_the_whole_space_potential_terms_fixed_time}, it is sufficient to apply Lemma \ref{lem:compactness_integral_fn_mu_n_app}. \\

Finally, for the term $\int_0^t \int_{\Omega} \left| \nabla (W_{\mu}- W_{\nu})\ast\mu^R_s \right|^2 \mu^R_s \diff x  \diff s$, we first observe that for a.e. $x\in\Rd$ and $s\in(0,t)$
$$
\left| \nabla (W_{\mu}- W_{\nu})\ast\mu^R_s \right|^2 \to \left| \nabla (W_{\mu}- W_{\nu})\ast\mu_s \right|^2
$$
by writing $\nabla (W_{\mu}- W_{\nu})\ast\mu^R_s = \int_{\Rd} \nabla (W_{\mu}- W_{\nu})(x-y) \diff \mu^R_s(y)$ and using growth conditions \ref{ass_item:growth_conditions_V_W_main_thm} and Lemma \ref{lem:compactness_integral_fn_mu_n_app}. Then, we observe that
\begin{equation}\label{eq:quadratic_tail_nonlocal_term_approx_bounded_set_to_the_whole_space}
\left| \nabla (W_{\mu}- W_{\nu})\ast\mu^R_s \right|^2 \leq \overline{C}\,(1+|x|^2),
\end{equation}
where $\overline{C}$ is a constant independent of $x$ and $s\in(0,t)$. This is proved as in \eqref{eq:estimate_nabla_W_mollifies_mu_Step_2}. Therefore, we can apply Lemma \ref{lem:compactness_integral_fn_mu_n_app} once again to deduce $\int_{\Omega} \left| \nabla (W_{\mu}- W_{\nu})\ast\mu^R_s \right|^2 \mu^R_s \diff x \to \int_{\Omega} \left| \nabla (W_{\mu}- W_{\nu})\ast\mu_s \right|^2 \mu_s \diff x$ for all $s\in(0,t)$. We conclude by the dominated convergence (exploiting \eqref{eq:two_properties_approx_sequence_mu_N_agg-diff} and \eqref{eq:quadratic_tail_nonlocal_term_approx_bounded_set_to_the_whole_space}).

\underline{Step 5: Necessary adaptations for the case $\Omega=R\,\Td$.} We only need to explain how to establish \ref{eq:derivative_via_1st_var_sect:alternative_proof} for $\Omega=R\,\Td$ in Step 2 of the proof above. We need to estimate three slopes $\frac{\diff}{\diff s} \mathcal{G}_1[\gamma_s]$, $\frac{\diff}{\diff s} \mathcal{G}_2[\gamma_s]$ and $\frac{\diff}{\diff s} \mathcal{G}_3[\gamma_s]$ at $s=0$ and $s=1$, where $\mathcal{G}_1$, $\mathcal{G}_2$, $\mathcal{G}_3$ are defined in \eqref{eq:functional_G_aggr-diff_splitting_for_3_functioonals} and $\{\gamma_s\}_{s\in[0,1]}$ is a geodesic connecting $\mu_t$ and $\nu_t$. For $\frac{\diff}{\diff s} \mathcal{G}_1[\gamma_s]$, we argue as in Step 4 of the proof of Theorem \ref{thm:PME} (via Theorem \ref{thm:simplified_for_2-Wasserstein}). For $\frac{\diff}{\diff s} \mathcal{G}_2[\gamma_s]$ observe that we can write
$$
\mathcal{G}_2[\gamma_s]\!= \!\int_{\Omega}   V_{\nu}(x) \diff \gamma_s(x) \!=\! \int_{\Omega}   V_{\nu}((x- s\,\nabla\varphi(x))\mbox{ mod } 2 {R} ) \diff \gamma_0(x)\!=\!\int_{\Omega}   V_{\nu}(x- s\,\nabla\varphi(x) ) \diff \gamma_0(x),
$$
where in the second equality we exploited the definition of geodesics \eqref{eq:geodesic_map_torus} while in the third equality we used periodicity of $V_{\nu}$. By a direct differentiation
$$
\frac{\diff}{\diff s} \mathcal{G}_2[\gamma_s] \Big|_{s=0} = -\int_{\Omega} \nabla V_{\nu}(x)\cdot \nabla \varphi(x) \diff \mu_t(x) = -\int_{\Omega} \nabla \frac{\delta \mathcal{G}_2}{\delta \nu}[\mu_t] \cdot \nabla \varphi(x) \diff \mu_t(x).
$$
The case $s=1$ is analyzed similarly by noting that $\{\gamma_{1-s}\}_{s\in[0,1]}$ is a geodesic connecting $\nu_t$ to $\mu_t$ (see Remark \ref{rem:inverse_geodesic_torus_from_mu1_to_mu0}) and the formula $\frac{\diff}{\diff s} \mathcal{G}_2[\gamma_s] \Big|_{s=1} = -\frac{\diff}{\diff s} \mathcal{G}_2[\gamma_{1-s}] \Big|_{s=0}$. For $\frac{\diff}{\diff s} \mathcal{G}_3[\gamma_s]$ we argue similarly. Using the definition of geodesics \eqref{eq:geodesic_map_torus} and periodicity of $W_{\nu}$ we obtain
$$
\mathcal{G}_3[\gamma_s] = \frac{1}{2} \int_{\Omega} \int_{\Omega}   W_{\nu}\big(x- s\,\nabla\varphi(x) - (y- s\,\nabla\varphi(y)) \big) \diff \gamma_0(x) \diff \gamma_0(y).
$$
We differentiate directly to obtain
\begin{align*}
\frac{\diff}{\diff s} \mathcal{G}_3[\gamma_s] \Big|_{s=0} &= - \frac{1}{2} \int_{\Omega} \int_{\Omega}   \nabla W_{\nu}\big(x- y \big) (\nabla \varphi(x)-\nabla\varphi(y)) \diff \mu_t(x) \diff \mu_t(y)\\
&= -\int_{\Omega} \nabla W_{\nu}\ast \mu_t(x) \cdot \nabla \varphi(x) \diff \mu_t(x) = - \int_{\Omega} \nabla \frac{\delta \mathcal{G}_3}{\delta \nu}[\mu_t] \cdot \nabla \varphi(x) \diff \mu_t(x),
\end{align*}
where in the last step we used $\nabla W_{\nu}(x-y)=-\nabla W_{\nu}(y-x)$. The case $s=1$ follows analogously by the argument above.
\end{proof}

\begin{rem}
We note that in Step 2 when applying Theorem \ref{thm:simplified_for_2-Wasserstein}, it was necessary to decide wheather we use geodesic convexity of $\mathcal{F}$ or $\mathcal{G}$. We selected $\mathcal{G}$ and it is important to emphasize that choosing $\mathcal{F}$ instead would not be possible. Indeed, if we had chosen $\mathcal{F}$, we would have needed to estimate in Step 3 integrals of the form $\int_0^t \int_{\Omega} \nu_s^{2m-3} |\nabla \nu_s|^2 \diff x \diff s$, where $\nu_s$ solves \eqref{eq:aggregation_diffusion_2}. Such an estimate cannot be obtained in full generality. In fact, for solutions to \eqref{eq:aggregation_diffusion_2}, by using \eqref{eq:Lp_norm_sol_A-D_not_integrated_in_time_STEP1_for_2nd_sol} one controls $\int_0^t \int_{\Omega} \nu_s^{\kappa}\, |\nabla \nu_s|^2 \diff x \diff s$ for $\kappa\geq n-2$. The critical value $\kappa = n-2$ corresponds to multiplying \eqref{eq:aggregation_diffusion_2} by $\log\nu_s$. Therefore, one gets the restriction $2m-3\geq n-2$, i.e. $2m-1 \geq n$. This range can be slightly improved by multiplying by $\nu_s^{k}$ for $k \in (-1,0)$; however, this requires estimating integrals of the form $\int_{\Omega} \nu_t^{k+1} \diff x$, which can be controlled for $k$ close to 0 in terms of the mass and suitable moment estimates, see e.g. \cite{lutwak2005crame}.
\end{rem}

\subsection*{Acknowledgements}
JAC and JS were supported by the Advanced Grant Nonlocal-CPD (Nonlocal PDEs for Complex Particle Dynamics: Phase Transitions, Patterns and Synchronization) of the European Research Council Executive Agency (ERC) under the European Union’s Horizon 2020 research and innovation programme (grant agreement No. 883363). PG was supported by the National Science Center (Poland), project UMO-2023/51/B/ST1/01546. We also wish to thank Jakub Woźnicki and André Schlichting for bringing to our attention several references on a similar subject as well as to Fabian Rupp for the observation that assumptions \ref{eq:geodesic_conv_G_sect:alternative_proof} and \ref{eq:derivative_via_1st_var_sect:alternative_proof} can be replaced by the inequality~\eqref{eq:condition_replacing_A2+A_3_by_Fabian}.

\appendix
\section{Semigroup theory}\label{app:semigroup_theory}
We briefly review the theory of semigroups acting on a general metric space $(S,d)$. In the applications in this paper, $(S,d)$ will be a subspace of the space of probability measures $\mathcal{P}_p(\Omega)$ equipped with the Wasserstein metric $\mathcal{W}_p$. The material comes from \cite[Appendix~I]{MR4309603}. 

\begin{Def}\label{def:semigroup_autonomous}
We say that the one-parameter map $\{\mathcal{S}_{t}\}_{t \in [0,T]}$ defines an autonomous semigroup on $(S,d)$ if 
\begin{itemize}
\item $\mathcal{S}_{0} = \mathcal{I}$ where $\mathcal{I}$ is the identity map,
\item $\mathcal{S}_{t} \circ \mathcal{S}_{s} = \mathcal{S}_{t+s}$ for all $t,s, s+t \in [0,T]$.
\end{itemize}
\end{Def}

\begin{Def}\label{def:abs_cont_map_general_metric_space}
We say that the map $\mu_t: [0,T] \ni t \mapsto S$ is absolutely continuous if there exists a function $h: [0,T] \to \R$, $h \in L^1(0,T)$ with $h \geq 0$ such that for all $s,t\in [0,T]$ with $s < t$
$$
d(\mu_t, \mu_s) \leq \int_s^t h(u) \diff u.
$$
\end{Def}

\begin{Def}\label{def:semigroup_Lipschitz-AC-autonomous}
We say that the semigroup $\{\mathcal{S}_{t}\}_{t \in [0,T]}$ is Lipschitz if there exists a constant $K(t)>0$ such that
\begin{equation}\label{eq:lipschitz-AC-estimate-app-aut}
d(\mathcal{S}_{t} x, \mathcal{S}_{t} y) \leq K(t) \, d(x,y).
\end{equation}
We say it is absolutely continuous in time if for all $x \in S$, the map $t \mapsto \mathcal{S}_t x$ is absolutely continuous.
\end{Def}

The main property of Lipschitz semigroups which are absolutely continuous in time is the following estimate proved in \cite[Theorem 2.9]{MR1816648} for Lipschitz maps $\mu_t$ and Lipschitz semigroups $\mathcal{S}_t$ which are also Lipschitz in time (i.e. the map $t\mapsto \mathcal{S}_t x$ is Lipschitz). The estimate was extended in \cite[Prop.~I.9]{MR4309603} to cover the case of absolutely continuous maps $\mu_t$ and nonautonomous semigroups. 

\begin{lem}\label{lem:bressan_estimate_between_map_semigroup_aut}
Let $\{\mathcal{S}_{t}\}_{t \in [0,T]}$ be a Lipschitz, absolutely continuous in time semigroup on $(S,d)$ with the constant $K(t)$ being continuous and let $\mu_t: [0,T]\ni t \mapsto \mu_t \in S$ be an absolutely continuous map. Then, 
$$
d(\mu_t, \mathcal{S}_{t} \mu_0 ) \leq \int_0^t K(t-s)\,\liminf_{h \to 0^+} \frac{d(\mu_{s+h}, \mathcal{S}_{h} \mu_s)}{h} \diff s, 
$$
where $K(t)$ is the constant from~\eqref{eq:lipschitz-AC-estimate-app-aut}.
\end{lem}
\begin{proof}
We provide a proof since the aforementioned references considered only the case of $K$ being constant. Consider the map $\Phi(s) = d(\mathcal{S}_{t-s} \mu_s, \mathcal{S}_t \mu_0)$. By absolute continuity of $\mu_s$ and the semigroup, the map $\Phi$ is absolutely continuous and differentiable a.e. We compute for $h>0$ by triangle inequality and \eqref{eq:lipschitz-AC-estimate-app-aut}
\begin{align*}
\Phi(s+h) - \Phi(s) &= d(\mathcal{S}_{t-s-h} \mu_{s+h}, \mathcal{S}_t \mu_0)-d(\mathcal{S}_{t-s} \mu_s, \mathcal{S}_t \mu_0) \leq d(\mathcal{S}_{t-s-h} \mu_{s+h}, \mathcal{S}_{t-s} \mu_s)\\ 
&= d(\mathcal{S}_{t-s-h} \mu_{s+h}, \mathcal{S}_{t-s-h} \mathcal{S}_{h}  \mu_s) \leq K(t-s-h)\, d(\mu_{s+h}, \mathcal{S}_{h}  \mu_s). 
\end{align*}
It follows that for a.e. $s$
$$
\Phi'(s)\leq K(t-s)\, \liminf_{h\to 0^+} \frac{d(\mu_{s+h}, \mathcal{S}_{h}  \mu_s)}{h}.
$$
The conclusion follows by integrating the expression above from $s=0$ to $s=t$.
\end{proof}

\section{Continuity equation and optimal transport theory}\label{app:continuity_eq_opt_transport}

\subsection{Continuity equation}
In what follows, $\BL(\Omega)$ is the space of bounded Lipschitz functions on $\Omega$ with the norm defined by \eqref{eq:norm_BL}.

\begin{lem}\label{lem:representation_solution_continuity_equation}
Let ${\bf v}:(0,T)\times \Omega \to \R^d$ be a vector field such that ${\bf v} \in L^1(0,T; \BL(\Omega))$ (in the case $\Omega$ is a bounded domain, assume additionally ${\bf v}(t,\cdot) \cdot \bn \leq 0$ on $\partial \Omega$ for a.e. $t \in (0,T)$). Then, there exists a unique narrowly continuous curve $[0,T] \ni t \mapsto \mu_t \in \mathcal{P}(\Omega)$ such that $\int_0^T \int_{\Omega} |{\bf v}(t,x)| \diff \mu_t(x) \diff t<\infty$ and $\mu_t$ solves 
$$
\partial_t \mu_t + \DIV(\mu_t \, {\bf v}) = 0
$$
in $[0,T]\times\Omega$. In fact, if $X(s,t,x)$ is the flow of ${\bf v}$ as in \eqref{eq:flow_of_the_vf}, $\mu_t$ is explicitly given by
\begin{equation}\label{eq:push-forward-representation}
\mu_t = X(0,t,\cdot)^{\#} \mu_0,
\end{equation}
i.e. for all bounded functions $\psi: \Omega \to \R$ we have
$
\int_{\Omega} \psi(x) \diff \mu_t(x) = \int_{\Omega} \psi(X(0,t,x)) \diff \mu_0(x).
$ 
\end{lem}
\begin{proof}
In the case $\Omega = \Rd$ or $\Omega = \Td$ the proof is classical, see \cite[Prop. 8.1.7, 8.1.8]{MR2129498}. Let us sketch the argument for a bounded domain. In fact, \cite[Prop. 8.1.7, 8.1.8]{MR2129498} implies the formula \eqref{eq:push-forward-representation} and the only question is if the weak condition ${\bf v}(t,\cdot) \cdot \bn \leq 0$ on $\partial \Omega$ for a.e. $t \in (0,T)$ implies that $X(0,t,x) \in \Omega$. This would be clear if the condition ${\bf v}(t,\cdot) \cdot \bn \leq 0$ on $\partial \Omega$ was satisfied for all times. We extend $v$ by zero for $t<0$ and we define ${\bf v_{\eps}} = \frac{1}{\eps} \int_{t-\eps}^t {\bf v}(s,x) \diff s = {\bf v}(\cdot,x) \ast (\frac{1}{\eps} \mathds{1}_{[0,\eps]})$ so that ${\bf v_{\eps}}$ is continuous in time and space and the corresponding flow map $X_{\eps}$ stays in $\Omega$. It is easy to check that
\begin{equation}\label{eq:uniform_BL_mollifier_time}
\| {\bf{v_{\eps}}} \|_{L^1(0,t; \BL(\Omega))} \leq \| {\bf v} \|_{L^1(0,t; \BL(\Omega))}.
\end{equation}
We estimate
\begin{align*}
&|X_{\eps}(0,t,x) - X(0,t,x)| \leq \int_0^t |{\bf v_{\eps}}(s, X_{\eps}(0,s,x)) - {\bf v}(s,X(0,s,x))| \diff s\\
&\leq \int_0^t |{\bf v_{\eps}}(s, X_{\eps}(0,s,x)) - {\bf v_{\eps}}(s,X(0,s,x))| \diff s \\ 
&\phantom{ \leq \, }+ \int_0^t \left|{\bf v_{\eps}}(s, X(0,s,x)) - {\bf v}(\cdot,X(0,\cdot,x)) \ast \left(\frac{1}{\eps} \mathds{1}_{[0,\eps]}\right) (s) \right| \diff s \\
&\phantom{ \leq \, } + \int_0^t \left|{\bf v}(\cdot,X(0,\cdot,x)) \ast \left(\frac{1}{\eps} \mathds{1}_{[0,\eps]}\right) (s) - {\bf v}(s,X(0,s,x)) \right| \diff s =: I_1^{\eps} + I_2^{\eps} + I_3^{\eps}.
\end{align*}
The term $I_1^{\eps}$ can be estimated by using Lipschitz continuity of ${\bf v_{\eps}}$:
$$
|I_1^{\eps}|\leq \int_0^t \Lip({\bf v_{\eps}}(s,\cdot))\, |X_{\eps}(0,s,x)-X(0,s,x)| \diff s.
$$
By Grönwall's inequality and \eqref{eq:uniform_BL_mollifier_time} we obtain
$$
|X_{\eps}(0,t,x) - X(0,t,x)| \leq (I_2^{\eps}+I_3^{\eps})\, e^{\| {\bf v}_{\eps} \|_{L^1(0,t; \BL(\Omega))}} \leq (I_2^{\eps}+I_3^{\eps})\, e^{\| {\bf v} \|_{L^1(0,t; \BL(\Omega))}}.
$$
We now prove that when $\eps \to 0$, $I_2^{\eps}, I_3^{\eps} \to 0$ which implies $|X_{\eps}(0,t,x) - X(0,t,x)| \to 0$ so that $X(0,t,x) \in \overline{\Omega}$ and the proof will be concluded.\\

For $I_3^{\eps}$, we have $I_3^{\eps}\to 0$ because the function $t \mapsto {\bf v}(t,X(0,t,x))$ is integrable on $(0,T)$ (indeed, we can estimate $\int_0^T |{\bf v}(t,X(0,t,x))| \diff t \leq \int_0^T \| {\bf v}(t,\cdot)\|_{L^{\infty}(\Omega)} \diff t$) so its convolution with a mollification kernel converges strongly in $L^1(0,T)$. For $I_2^{\eps}$ we write
\begin{align*}
I_2^{\eps} &\leq \frac{1}{\eps}
\int_0^t \int_{s-\eps}^s \left|{\bf v}(u, X(0,s,x)) - {\bf v}(u,X(0,u,x)) \right| \diff s\\
& \leq \frac{1}{\eps}
\int_0^t \int_{s-\eps}^s \Lip({\bf v}(u,\cdot)) \, |X(0,s,x)-X(0,u,x)| \diff u \diff s
\\
& \leq \frac{1}{\eps}
\int_0^t \left(\int_{s-\eps}^s \Lip({\bf v}(u,\cdot)) \diff u\right) \, \left(\int_{s-\eps}^s \| {\bf v}(\tau,\cdot) \|_{L^{\infty}(\Omega)} \diff \tau \right)  \diff s.
\end{align*}
By convergence of mollifiers and $\Lip({\bf v}(u,\cdot)) \in L^1(0,T)$, $\frac{1}{\eps}\, \int_{s-\eps}^s \Lip({\bf v}(u,\cdot)) \diff u \to \Lip({\bf v}(s,\cdot))$ in $L^1(0,T)$ while $\int_{s-\eps}^s \| {\bf v}(\tau,\cdot) \|_{L^{\infty}(\Omega)} \diff \tau \to 0$ in $L^{\infty}(0,T)$ by absolute continuity of the Lebesgue integral. 
\end{proof}

\begin{lem}\label{lem:p_moment_stays_bounded}
Under the setting of Lemma \ref{lem:representation_solution_continuity_equation}, if $\mu_0 \in \mathcal{P}_p(\Omega)$ then $\mu_t \in \mathcal{P}_p(\Omega)$ for all $t>0$ as long as $\int_0^t \| {\bf v}(\tau,\cdot) \|_{L^{\infty}(\Omega)} \diff \tau < \infty$. 
\end{lem}
\begin{proof}
We estimate
$$
|X(0,t,x)| \leq |x| + \int_0^t \|{\bf v}(\tau, \cdot)\|_{L^{\infty}(\Omega)} \diff \tau.
$$
Therefore, using that $\mu_0 \in \mathcal{P}_p(\Omega)$
$$
\int_{\Omega}|x|^p \diff \mu_t(x) = \int_{\Omega} |X(0,t,x)|^p \diff \mu_0(x) \leq 2^p \int_{\Omega} |x|^p \diff \mu_0(x) + 
2^p\, \| {\bf v}\|^p_{L^1(0,t; L^{\infty}(\Omega))}.
$$
\end{proof}

\subsection{Wasserstein distance} 

\begin{lem}\label{lem:estimate_Wasserstein_by_pushforwards}
Let $\mu^1, \mu^2 \in \mathcal{P}_p(\Omega)$ be given by $\mu^i = (X^i)^{\#} \mu$ for some $X^i: \Omega\to\Omega$ and $\mu\in \mathcal{P}_p(\Omega)$. Then,
$$
\mathcal{W}_p(\mu^1, \mu^2) \leq \|X^1 - X^2 \|_{L^{\infty}(\Omega)} \, \|\mu\|^{\frac{1}{p}}_{TV}.
$$
\end{lem}
\begin{proof}
This is a simple consequence of the fact that the measure $(X^1, X^2)^{\#}\mu$ on $\Omega\times\Omega$ is an admissible coupling in the definition of $\mathcal{W}_p(\mu^1, \mu^2)$. 
\end{proof}
\begin{lem}[rescaled Benamou-Brenier formula]\label{lem:B-B}
Let $\mu, \nu \in \mathcal{P}_p(\Omega)$ and let $h>0$. Then
\begin{equation}
\mathcal{W}_p^p(\mu,\nu) \leq h^{p-1} \int_0^h \int_{\Omega} |{\bf v}|^p \diff \mu_\tau \diff \tau,
\end{equation}
where $\mu_t$ and ${\bf v}$ satisfy the continuity equation $\partial_\tau \mu_{\tau} + \DIV(\mu_{\tau}\,{\bf v})=0$ on $[0,h]\times\Omega$ with $\mu_{\tau} \in \mathcal{P}(\Omega)$ for all $\tau \in [0,h]$, $\mu_0=\mu$, $\mu_h =\nu$, ${\bf v} \in L^1(0,h; \BL(\Omega))$ and ${\bf v}(t,\cdot) \cdot \bn \leq 0$ on $\partial \Omega$ for a.e. $t \in [0,h]$. 
\end{lem}
\begin{proof}
We follow \cite[Theorem 4.1.3]{MR4331435}. By Lemma \ref{lem:representation_solution_continuity_equation} we have $\mu_{\tau} = X(0,\tau,\cdot)^{\#} \mu_0$. Hence,
\begin{align*}
\int_0^h \int_{\Omega} |{\bf v}|^p &\diff \mu_\tau \diff \tau = 
\int_0^h \int_{\Omega} |{\bf v}(\tau, X(0,\tau,x))|^p \diff \mu_0(x) \diff \tau \\
&= \int_0^h \int_{\Omega} |\partial_{\tau} X(0,\tau,x)|^p \diff \mu_0(x) \diff \tau 
\geq h\,  \int_{\Omega} \left |\frac{1}{h} \int_0^h \partial_{\tau} X(0,\tau,x) \diff \tau \right|^p \diff \mu_0(x)\\
&= h^{1-p} \, \int_{\Omega} |X(0,t,x) - x|^p \diff \mu_0(x) \geq h^{1-p} \,\mathcal{W}^p_p(\mu, \nu).
\end{align*}
\end{proof}

\begin{lem}\label{lem:abs_cont_continuity_equat}
Let $\mu_t \in \mathcal{P}_p(\Omega)$ be a solution of $\partial_t \mu_t + \DIV(\mu_t \, {\bf v}) = 0$ with ${\bf v}$ satisfying assumptions of Lemma \ref{lem:representation_solution_continuity_equation}. Additionally assume that $\int_0^T \int_{\Omega} |{\bf v}|^p \diff \mu_t(x) \diff t < \infty$. Then, $t\mapsto \mu_t$ is absolutely continuous in $(\mathcal{P}_p(\Omega), \mathcal{W}_p)$.
\end{lem}
\begin{proof}
Let $p>1$. By Lemma \ref{lem:B-B}, we have for $t>s$
$$
\mathcal{W}_p^p(\mu_s, \mu_t)\leq |t-s|^{p-1}\, \int_s^t\int_{\Omega} |{\bf v}|^p \diff \mu_u(x) \diff u \leq |t-s|^p + \left(\int_s^t\int_{\Omega} |{\bf v}|^p \diff \mu_u(x) \diff u\right)^p.
$$
Taking the $p$-th root, we see that $\mathcal{W}_p(\mu_s, \mu_t)\leq \int_s^t f(u)\diff u$ for $f \in L^1(0,T)$, concluding the proof. For $p=1$, we have directly from Lemma \ref{lem:B-B} that $\mathcal{W}_1(\mu_s, \mu_t)\leq \int_s^t\int_{\Omega} |{\bf v}| \diff \mu_u(x) \diff u$.
\end{proof}

\begin{lem}\label{lem:compactness_integral_fn_mu_n_app}
Let $\{\mu_n\}_{n\in\N} \subset \mathcal{P}_2(\Rd)\cap L^1(\Rd)$ such that $\sup_{n \in \N} \int_{\Rd} |\mu_n|^p \diff x < \infty$ for some $p>1$ and $\mathcal{W}_2(\mu_n,\mu)\to 0$ when $n\to\infty$ for some $\mu\in \mathcal{P}_2(\Rd)$. Let $f_n:\Rd \to \R$ be such that $|f_n(x)|\leq C(1+|x|^2)$ for a uniform constant $C$ and $f_n(x)\to f(x)$ for a.e. $x\in\Rd$ as $n\to \infty$. Then, $\int_{\Rd} f_n(x) \diff \mu_n(x) \to \int_{\Rd} f(x)\diff \mu(x)$.  
\end{lem}
\begin{proof}
First, we observe that
\begin{equation}\label{eq:weak_compactness_L1_approx_lemma}
\mu_n \rightharpoonup \mu \mbox{ weakly in } L^1(\Rd), \qquad \int_{\Rd} |x|^2\diff\mu_n \to \int_{\Rd} |x|^2\diff\mu.
\end{equation}
Indeed, $\mathcal{W}_2(\mu_n,\mu)\to 0$ implies narrow convergence of $\mu_n \to \mu$, the second property in \eqref{eq:weak_compactness_L1_approx_lemma} and tightness of the sequence $\{\mu_n\}_{n \in \N}$. The condition $\sup_{n} \int_{\Rd} |\mu_n|^p \diff x < \infty$ together with the Dunford-Pettis theorem imply the weak convergence in $L^1(\Rd)$ and $\int_{\Rd} |\mu|^p \diff x < \infty$.\\

Next, we note that for each $R>0$
\begin{equation}\label{eq:convergence_fn_mun_local_integrals}
\int_{B_R}f_n(x) \diff \mu_n(x) \to \int_{B_R} f(x) \diff \mu(x), \qquad \int_{\Rd \setminus B_R} |x|^2 \diff \mu_n(x)\to \int_{\Rd \setminus B_R} |x|^2 \diff \mu(x).
\end{equation}
Indeed, by Egorov's theorem, for each $\eps >0$ there is a set $A_{\eps}\subset B_R$ such that $|A_{\eps}|\leq \eps$ and $f_n\to f$ uniformly on $B_R\setminus A_{\eps}$. Combined with the weak $L^1$ convergence of $\mu_n$, we obtain the convergence of the integral over $B_R\setminus A_{\eps}$. On $A_{\eps}$, we can estimate
\begin{multline*}
\left|\int_{A_{\eps}}f_n(x) \diff \mu_n(x)\right| + \left|\int_{A_{\eps}}f(x) \diff \mu(x)\right| \leq C\,(1+R^2)\,\left(\int_{A_{\eps}} \diff \mu_n(x) + \int_{A_{\eps}} \diff \mu(x)\right) \\
\leq C\,(1+R^2)\, \sup_{n\in\N} \left(\int_{\Rd} (|\mu_n|^p + |\mu|^p) \diff x\right)^{\frac{1}{p}} \, |A_{\eps}|^{1-\frac{1}{p}} \to 0 \mbox{ as } \eps \to 0,
\end{multline*}
implying the first convergence in \eqref{eq:convergence_fn_mun_local_integrals}. The second convergence in \eqref{eq:convergence_fn_mun_local_integrals} is a direct consequence of \eqref{eq:weak_compactness_L1_approx_lemma} since $|x|^2\,\mathds{1}_{B_R} \in L^{\infty}(\Rd)$. \\

Using \eqref{eq:convergence_fn_mun_local_integrals} we obtain
\begin{align*}
&\limsup_{n\to \infty} \left|\int_{\Rd}f_n(x) \diff \mu_n(x)-\int_{\Rd}f(x) \diff \mu(x)\right| \leq \\ 
& \qquad \qquad \leq \limsup_{n\to \infty} C\,\int_{\Rd\setminus B_R}(1+|x|^2) \diff\mu_n(x) + C\,\int_{\Rd\setminus B_R}(1+|x|^2) \diff\mu(x)\\ 
&\qquad \qquad =
2\,C\,\int_{\Rd\setminus B_R}(1+|x|^2) \diff\mu(x).
\end{align*}
The (RHS) can be made arbitrary small by letting $R\to\infty$ and this concludes the proof.
\end{proof}

\subsection{Optimal transport on the torus}
\label{rem:optimal_transport_periodic_domain}
We discuss briefly the optimal transportation on a periodic domain $R\,\Td$ and the related geometry of the space $(\mathcal{P}_2(R\,\Td), \mathcal{W}_2)$ in this subsection. We recall that we identify $\Td = (-1,1]^d$. On ${R}\, \Td$, the geodesic distance is defined as
\begin{equation}\label{eq:geodesic_distance_periodic_domain}
|[z_1 - z_2]| := \min\{ |z_1-z_2+ 2 {R}\,k|: k\in \mathbb{Z}^d\},  \qquad z_1, z_2 \in {R}\, \Td.
\end{equation}
Let $\mu$, $\nu$ be two absolutely continuous periodic measures on $R\, \Td$. By \cite{MR1711060}, there exists the optimal transport map $T: \Rd \to \Rd$ which maps $\mu$ into $\nu$ understood as measures extended periodically to $\Rd$ (the optimality is understood here in the sense that the support of the optimal coupling is a cyclically monotone set). Moreover, the projection of the map $T(x)$, defined as the unique $T_P(x)\in (-R,R]^d$ such that
\begin{equation}\label{eq:euclidean_periodic_distances}
|[T_P(x)-x]| = |T(x) - x|,
\end{equation}
is the optimal map pushing $\mu$ onto $\nu$ as measures on $R\,\Td$. This is well-defined as $T(x)-x$ is a periodic map \cite[Theorem 1]{MR1711060} so that $T(x+2Rk) = T(x) + 2R\,k$ for all $k\in \mathbb{Z}^d$. The optimality can be seen from the proof in \cite{MR1711060} by \cite[Theorem 1.49]{MR3409718} which gives equivalence between optimality and cyclical monotonicity of the support. In fact, \cite{MR1711060} shows the cyclical monotonicity of the supports of the couplings $(\mathcal{I}, T_P)^{\#}\mu$ and $(\mathcal{I}, T)^{\#}\mu$ on $R\,\Td$ and $\Rd$, respectively (see \cite[Prop.~2]{MR1711060}).\\

From \cite{MR1711060} we know that $T(x)=\partial \psi(x)$ for a convex, lower semicontinuous and proper function $\psi$, where $\partial \psi$ is its subdifferential. Here, $T(x)$ is only understood a.e. More precisely, it is defined at the points where $\nabla \psi$ exists in which case $\partial\psi(x) =\{\nabla\psi(x)\}$. This set is Borel measurable as an intersection of Borel measurable sets \cite[Theorem 3.2]{EvansGariepy} and hence, $\nabla \psi(x)$ is Borel measurable when restricted to $\Rd\setminus \{x: \nabla \psi(x) \mbox{ does not exist}\}$ (or extended to $\Rd$ by a constant). \\

Next, we discuss the geodesics. The Kantorovich potential appearing in the dual formulation of the optimal transport satisfies 
\begin{equation}\label{eq:Kantorovich_potential_periodic_domain}
-\nabla \varphi(x) = T(x)-x 
\end{equation}
(this may be seen from the proof in \cite[Theorem 1.25]{MR3409718}). We remind that $T(x)-x$ is a periodic map \cite[Theorem 1]{MR1711060}. A consequence of \eqref{eq:euclidean_periodic_distances} and \eqref{eq:Kantorovich_potential_periodic_domain} is that
\begin{equation}\label{eq:Wass_dis_periodic_integral_of_Kantorovich}
\mathcal{W}_2^2(\mu,\nu) = \int_{R\,\Td} |\nabla \varphi(x)|^2 \diff \mu(x). 
\end{equation}
It is also a simple calculus exercise to check that the optimal value of $k$ in \eqref{eq:geodesic_distance_periodic_domain} belongs to $\{-1,0,1\}^d$ so that from \eqref{eq:euclidean_periodic_distances} and \eqref{eq:Kantorovich_potential_periodic_domain} we deduce
\begin{equation}\label{eq:bound_displacement_periodic_domain}
|\partial_{x_i} \varphi(x)| = |(T(x) - x)_i| \leq R,
\end{equation}
where $(T(x) - x)_i$ is the $i$-th coordinate of $T(x)-x$. The geodesic between $\mu$ and $\nu$ is defined via
\begin{equation}\label{eq:geodesic_torus}
\gamma_s = T_{s,P}^{\#}\, \mu, \qquad T_{s,P}(x) = \exp_{x}(-s \nabla \varphi(x)), \qquad s \in [0,1],
\end{equation}
where $\exp_x(v)$ is the exponential map in the sense of differential geometry \cite[remark after Corollary 10.10]{MR4886890}. On the flat torus $R\,\Td$, the geodesics are straight lines so that
\begin{equation}\label{eq:geodesic_map_torus}
T_{s,P}(x) = (x - s \nabla \varphi(x))\, \mbox{ mod } 2 {R},
\end{equation}
where the modulo operation is applied on each coordinate. We will also need the following auxiliary map ${T}_s:\Rd \to \Rd$
\begin{equation}\label{eq:geodesic_map_torus_auxilliary}
{T}_s(x) =x - s \nabla \varphi(x).
\end{equation}
\begin{rem}\label{rem:inverse_geodesic_torus_from_mu1_to_mu0}
It is clear from the discussion above that the geodesic connecting $\nu$ and $\mu$ is given by $\gamma_{1-s}$ for $s\in[0,1]$. Indeed, if $S$ is the map transporting $\nu$ into $\mu$ (as measures extended periodically on $\Rd$) and $\phi$ is the Kantorovich potential, we have $-\nabla \phi(x)= S(x)-x$ and $-\nabla \phi(T(x)) = x-T(x)=\nabla \varphi(x)$. Therefore, using that $\nabla \phi$ is periodic,
\begin{align*}
S_{s,P}(T_P(x)) &= ((T(x) \mbox{ mod } 2 {R})-s\,\nabla \phi(T(x))) \mbox{ mod } 2 {R} = (T(x) +s\,\nabla \varphi(x)) \mbox{ mod } 2 {R}\\
&= (x - (1-s)\, \nabla \varphi(x)) \mbox{ mod } 2 {R} = T_{1-s, P}(x)
\end{align*}
and we conclude that $S_{s,P}^{\#}\nu = S_{s,P}^{\#}T_P^{\#} \mu = T_{1-s,P}^{\#}\,\mu$.
\end{rem}
We conclude with two technical lemmas used in the paper, whose main purpose is to interpret the Wasserstein gradient flow of the internal energy on $R\,\Td$ as that on $\Rd$.

\begin{lem}\label{lem:bijection_of_interpolating_curves_periodic}
Let $\mu$ be extended periodically to $\Rd$. Then, $T_s^{\#}\mu$ is absolutely continuous with respect to the Lebesgue measure. For $T_s^{\#}\mu$-a.e. point $y$, there exists unique $x$ such that $T_s(x)= y$. Similarly, when $\mu|_{(-R,R]^d}$ is the restriction of $\mu$ to $(-R,R]^d$, $T_s^{\#}\big(\mu|_{(-R,R]^d}\big)$ is also absolutely continuous and for $T_s^{\#}\big(\mu|_{(-R,R]^d}\big)$-a.e. $y$, there is unique $x\in(-R,R]^d$ such that $T_s(x)=y$.
\end{lem}
\begin{proof}
We write $T_s(x) = s\, \partial\psi(x) +(1-s)\,x = \partial \psi_s(x)$ for $\psi_s(x):= s\,\psi(x) + (1-s)\,\frac{|x|^2}{2}$ and we write $\psi^*_s$ for the Legendre transform of $\psi_s$. We first observe that $T_s^{\#}\mu$ is absolutely continuous with respect to the Lebesgue measure. For $s=1$, this follows from the fact that $\nu=T_1^{\#}\mu$ is absolutely continuous. For $s\in(0,1)$, by the reasoning in \cite[Lemma 4.6(i)]{MR1964483}, $T_s^{\#}\mu(B) = \mu(\partial\psi^*_s(B))$ for all sets $B$. We will show that $\partial\psi^*_s$ is a single-valued, Lipschitz continuous map which implies that $\partial\psi^*_s(B)$ is a set of measure zero whenever $B$ is. To this end, for any $y_1 \in \partial \psi_s(x_1)$, $y_2 \in \partial \psi_s(x_2)$ we have
$$
y_1 - y_2 \in s (\partial \psi(x_1)-\partial \psi(x_2)) + (1-s)\,(x_1-x_2).
$$
Multiplying by $x_1-x_2$ and noting that $(\partial \psi(x_1)-\partial \psi(x_2))\,(x_1-x_2) \geq 0$ we obtain
$$
|x_1 - x_2| \leq \frac{1}{1-s}\, |y_1-y_2|.
$$
Since $y_1 \in \partial \psi_s(x_1) \iff x_1 \in \partial \psi_s^*(y_1)$, this implies that $\partial\psi_s^*$ is single-valued. Moreover, it is Lipschitz continuous with constant $\frac{1}{1-s}$.\\

Now, let $A_x=\{x: \nabla \psi_s(x) \mbox{ exists} \}$, $A_y=\{y: \nabla \psi_s^*(y) \mbox{ exists} \}$ which are Borel sets of full (Lebesgue) measure. Note that $(A_x \times A_y)$ is a set of full $(\mathcal{I}, T_s)^{\#}\mu$ measure because
$$
\mathbb{R}^{2d} \setminus (A_x \times A_y) = ((\Rd \setminus A_x) \times \Rd) \cup (\Rd \times (\Rd \setminus A_y))
$$
and we know that both $\mu$ and $T_s^{\#}\mu$ are absolutely continuous. Hence, for $(\mathcal{I}, T_s)^{\#}\mu$-a.e. $(x,y) \in (A_x \times A_y)$ we have $y=\nabla \psi_s(x)$ (just look at the measure of $\{ y\neq\nabla \psi_s(x)\}$ which is Borel measurable) so that $y \in \partial\psi_s(x)$ and $x\in \partial\psi_s^*(y)$. Hence, $x=\nabla \psi^*_s(y)$ and finally $y=\nabla \psi_s(x) = \nabla \psi_s(\nabla \psi^*_s(y))$ for $(\mathcal{I}, T_s)^{\#}\mu$-a.e. $(x,y) \in A_x \times A_y$. Integrating out the $x$ variable (this uses that $A_x$ is of full measure), $y=\nabla \psi_s(\nabla \psi^*_s(y))$ for $T_s^{\#}\mu$-a.e. $y$ (this uses that $A_y$ is of full measure).\\

Concerning the statement for $\mu|_{(-R,R]^d}$, one only needs to notice that we still have $x=\nabla \psi_s^*(y)$ for $(\mathcal{I}, T_s)^{\#}\big(\mu|_{(-R,R]^d}\big)$-a.e. $(x,y)$ and since $(\Rd\setminus(-R,R]^d) \times\Rd$ is the null set, we can integrate out $x$ in the condition $\nabla \psi_s^*(y)\in (-R,R]^d$.
\end{proof}
\begin{lem}\label{lem:Lp_norm_of_geodesic_from_torus_to_Rd}
Let $\mu \in L^p(R\, \Td)$. Let $\gamma_s$ be defined as in \eqref{eq:geodesic_torus} (it is a measure on $R\, \Td$) and let $\widetilde{\gamma}_s = {T}^{\#}_s \big(\mu|_{(-R,R]^d}\big)$ be a measure on $\R^d$ (here, $\mu$ is a measure extended periodically to $\Rd$ and $\mu|_{(-R,R]^d}$ is its restrction to $(-R,R]^d$). Then,
$$
\int_{R\, \Td} |\gamma_s|^p \diff x = \int_{\Rd} |\widetilde{\gamma}_s|^p \diff x.  
$$
\end{lem}
\begin{proof}
We write $Q_0 = (-R, R]^d$ and we let $\{Q_i\}_{i=1,...,N}$ to be the set of disjoint cubes of side length $2R$ around $\Omega$, parallel to $Q_0$. Let $k_i$ be the unique vector in $\mathbb{Z}^d$ such that for all $y\in Q_i$, $y + 2 R k_i \in Q_0$ (we set $k_0=0$). We consider sets 
$$
A_i =\{y \in Q_0: {T}_s(x_y) = y \mbox{ for some } x_y \in Q_i\}.
$$
The idea is that in order to compute the integral $\int_{R\, \Td} |\gamma_s|^p \diff x$, we need to trace where the support of $\gamma_s$ comes from. The sets $A_i$ are well-defined and disjoint by Lemma \ref{lem:bijection_of_interpolating_curves_periodic} in the sense that for ${T}^{\#}_s\mu$-a.e. $y$ there is only one $x_y$ such that $T_s(x_y)=y$. By comparing formulas \eqref{eq:geodesic_map_torus} and \eqref{eq:geodesic_map_torus_auxilliary}, we deduce that $\gamma_s$ is a push-forward of $\widetilde{\gamma}_s$ under the modulo $2\,R$ map. We will prove
\begin{equation}\label{eq:splitting_support_periodic_measures_Aj}
\gamma_s(y)= \sum_{j=0}^N \widetilde{\gamma}_s(y + 2R k_j) = \widetilde{\gamma}_s(y +  2R k_i), \qquad \mbox{ for } y \in A_i.
\end{equation}
To see the first equality, we observe that by \eqref{eq:bound_displacement_periodic_domain}, we only need to consider cubes which are neighbours to $Q_0$. To see the second equality, we will prove $T_s^{-1}(y +  2R k_j)\notin Q_0$ for $j\neq i$. Indeed, if there was $x\in Q_0$ such that $T_s(x)= y+2R k_j$ for $y\in Q_0$ and some $j\neq i$, we would have $T_s(x-2Rk_j) = y$ and since $x-2Rk_j \in Q_j$, this would imply $y \in A_j$. The latter is impossible since $\{A_j\}_{j=0,...,N}$ are disjoint.\\

Next, we will cover the support of $\widetilde{\gamma}_s$ by proving that for $\widetilde{\gamma}_s$-a.e. $y$, there is $j$ such that $y\in \{A_j + 2R k_j\}$ which are disjoint sets. Indeed, let $E:= \{ y :\exists {x_y \in Q_0} \mbox{ such that } {T}_s(x_y) =y\}$ which is a set of full $\widetilde{\gamma}_s$ measure by Lemma \ref{lem:bijection_of_interpolating_curves_periodic}. Let $Q^a_j$ be an antipodal cube to $Q_j$ with respect to the origin. Note that for $y\in Q^a_j$, $y-2Rk_j\in Q_0$. Therefore, by \eqref{eq:bound_displacement_periodic_domain},
\begin{equation}\label{eq:decomposition_support_gamma_s_tilde}
E =\bigcup_{j=0}^N \{ y\in Q^a_j :\exists {x_y \in Q_0}  \mbox{ such that } {T}_s(x_y) =y\} =: \bigcup_{j=0}^N B_j,
\end{equation}
The decomposition \eqref{eq:decomposition_support_gamma_s_tilde} also implies disjointness of $\{B_j\}_{j=0,...,N}$. Since $y=y-2Rk_j +2Rk_j$, we can write
\begin{align*}
B_j &= \{ y\in Q_0 :\exists {x_y \in Q_0}  \mbox{ such that } {T}_s(x_y) =y +2Rk_j\} +2Rk_j \\
&= \{ y\in Q_0 :\exists {x_y \in Q_0}  \mbox{ such that } {T}_s(x_y-2Rk_j) =y\} +2Rk_j\\
&= \{ y\in Q_0 :\exists {x_y \in Q_j}  \mbox{ such that } {T}_s(x_y) =y\} +2Rk_j =  A_j + 2Rk_j.
\end{align*}
We can now conclude the proof of the lemma by applying \eqref{eq:splitting_support_periodic_measures_Aj} and the fact that the disjoint sets $\{A_j +2Rk_j\}$ cover the set $E$ given by \eqref{eq:decomposition_support_gamma_s_tilde}
$$
\int_{Q_0} |\gamma_s(y)|^p \diff y= \sum_{j=0}^N \int_{A_j} |\widetilde{\gamma}_s(y + 2R k_j) |^p \diff y = \sum_{j=0}^N\int_{A_j +2Rk_j} |\widetilde{\gamma}_s(y) |^p \diff y =\int_{E} |\widetilde{\gamma}_s(y)|^p \diff y.
$$
Since $E$ is a set of full $\widetilde{\gamma}_s$ measure by Lemma \ref{lem:bijection_of_interpolating_curves_periodic}, $\int_{E} |\widetilde{\gamma}_s(y)|^p \diff y = \int_{\Rd} |\widetilde{\gamma}_s(y)|^p \diff y$ and the proof is concluded.
\end{proof}
\section{Rate of convergence of the porous medium equation towards the heat equation}\label{app:rate_PME_nto1}

\begin{thm}
Let $\Omega$ be $\Rd$, a torus or a bounded smooth domain. Let $n >1$ and let $\mu$ and $\nu$ be the solutions to
\begin{equation}\label{eq:app:intro_PDE_rate_nto1}
\partial_t \mu = \Delta \mu, \qquad \qquad \partial_t \nu=\Delta \nu^n,
\end{equation}
on $\Omega$ with the same initial condition $\mu_0=\nu_0 \in \mathcal{P}_2(\Omega)\cap L^{\alpha}(\Omega)$ for some $\alpha>1$. Moreover, if $\Omega$ is a bounded domain, we equip \eqref{eq:app:intro_PDE_rate_nto1} with the Neumann boundary conditions. Then, for all $n \in [1, \alpha)$
$$
\mathcal{W}_2^2(\mu_t, \nu_t) \leq 
C(t,\alpha,n,\mu_0)\, {|n-1|},
$$
where the constant $C(t, \alpha,n, \mu_0)$ is defined by
$$
C(t, \alpha,n, \mu_0)\! =\!  4\,t\left(\frac{4}{(\alpha-n)^2} \|\mu_0\|_{L^{\alpha}(\Omega)}^{\alpha}\! +\!  \int_{\Omega} \mu_0\, |x|^2 \diff x \!+\!d\, t \left(2\!+\! \|\mu_0\|_{L^{\alpha}(\Omega)}^{\alpha}\right) \!+ \! \frac{16}{e^2} \int_{\Omega} e^{-\frac{|x|}{2}} \diff x 
\right)
$$
for $\Omega=\Rd$ or $\Omega$ being a periodic domain and
$$
C(t, \alpha,n, \mu_0) =  4\,t\left(\frac{4}{(\alpha-n)^2} \|\mu_0\|_{L^{\alpha}(\Omega)}^{\alpha}\! + R^2 + \! \frac{16}{e^2} \int_{\Omega} e^{-\frac{|x|}{2}} \diff x 
\right)
$$
if $\Omega$ is a bounded domain with $\sup_{x\in\Omega} |x| \leq R$.
\end{thm}
\begin{proof}
The PDEs \eqref{eq:app:intro_PDE_rate_nto1} are 2-Wasserstein gradient flows of the functionals $\mathcal{F}[\mu]= \int_{\Omega} \mu \log \mu \diff x$ and $\mathcal{G}[\nu] = \int_{\Omega} \frac{1}{n-1} \nu\,(\nu^{n-1}-1) \diff x$, respectively. Both functionals are geodesically convex so $\mu$ and $\nu$ satisfy the EVI: for all $\rho^1,\rho^2 \in \mathcal{P}_2(\Omega)$ and for a.e. $t$ 
$$
\mathcal{F}[\rho^1] \geq \mathcal{F}[\mu_t] + \frac{1}{2} \frac{\diff}{\diff t} \mathcal{W}_2^2(\mu_t, \rho^1), \qquad \mathcal{G}[\rho^2] \geq \mathcal{G}[\nu_t] + \frac{1}{2} \frac{\diff}{\diff t} \mathcal{W}_2^2(\nu_t, \rho^2).
$$
Hence, by the product rule,
\begin{equation}\label{eq:rate_nto1_timederivativeofW2}
\frac{1}{2} \frac{\diff}{\diff t} \mathcal{W}_2^2(\mu_t, \nu_t) \leq \mathcal{F}[\nu_t] - \mathcal{F}[\mu_t] + \mathcal{G}[\mu_t] - \mathcal{G}[\nu_t].
\end{equation}
We estimate $\mathcal{G}[\mu_t]-\mathcal{F}[\mu_t]$. We have
\begin{equation}\label{eq:difference_of_functionals_limit_nto1}
\mathcal{G}[\mu_t]-\mathcal{F}[\mu_t]  =
\int_{\Omega} \mu_t \left(\frac{\mu_t^{n-1}-1}{n-1} - \log\mu_t\right) \diff x. 
\end{equation}
Now, we compute the Taylor's expansion of the function $f(y) = a^y= e^{y\log a}$ around $y=0$. We have $f'(y) = (\log a) \, f(y)$, $f''(y)= (\log a)^2\, f(y)$ so that 
$$
f(y) = 1 + y\, \log a  + R, \qquad |R|\leq \frac{|y|^2}{2}\,(\log a)^2\, f(\xi), \qquad \xi \in [0,y].
$$
We can further estimate $f(\xi)\leq f(0) +f(y) = 1+f(y)$ since, depending on $\log a$, $f$ is either increasing or decreasing. Therefore, applying the Taylor's expansion to \eqref{eq:difference_of_functionals_limit_nto1} with $a=\mu_t$ and $y=n-1$, we obtain
\begin{equation}\label{eq:estimate_G-F_interms_mulog2_rate_nto1}
|\mathcal{G}[\mu_t]-\mathcal{F}[\mu_t]|\leq \frac{|n-1|}{2} \int_{\Omega} (\mu_t + \mu_t^n)|\log \mu_t|^2 \diff x.
\end{equation}
Note that the case $\mu_t=0$ is negligible since the integrand of \eqref{eq:difference_of_functionals_limit_nto1} vanish for $\mu_t=0$. Now, to estimate the integral in \eqref{eq:estimate_G-F_interms_mulog2_rate_nto1} we split $\Omega$ for three sets $\Omega_1 =\{\mu_t \geq 1\}$, $\Omega_2 = \{\mu_t \in [e^{-|x|},1]\}$, $\Omega_3 = \{\mu_t \leq e^{-|x|}\}$. On $\Omega_1$, we use the inequality $|\log y|\leq \frac{y^{(\alpha-n)/2}}{(\alpha-n)/2}$ for $y\geq 1$ so that
\begin{equation}\label{eq:estimate_mu_log2_rate_nto1_omega1}
\int_{\Omega_1} (\mu_t + \mu_t^n) \, |\log \mu_t|^2 \diff x \leq 2 \int_{\Omega_1} \mu_t^n \, |\log \mu_t|^2 \diff x  \leq \frac{8\, \|\mu_t\|_{L^{\alpha}(\Omega)}^{\alpha}}{(\alpha-n)^2}  \leq \frac{8\, \|\mu_0\|_{L^{\alpha}(\Omega)}^{\alpha}}{(\alpha-n)^2}  ,
\end{equation}
where we used that $L^p(\Omega)$ norms of $\mu_t$ do not increase with $t$. On the set $\Omega_2$, we use that $\log \mu_t \in [-|x|, 0]$ so that 
\begin{equation}\label{eq:rate_nto1_estimate_Omega_2}
\int_{\Omega_2} (\mu_t+\mu_t^n) \, |\log \mu_t|^2 \diff x \leq 2  \int_{\Omega_2} \mu_t \, |\log \mu_t|^2 \leq 2\int_{\Omega} \mu_t \, |x|^2 \diff x.
\end{equation}
To estimate the (RHS), if $\Omega = \Rd$ or $\Omega$ is a torus, we multiply the PDE for $\mu_t$ by $|x|^2$ and integrate by parts twice to obtain
$$
\int_{\Omega} \mu_t\, |x|^2 \diff x = 2\,d \int_0^t \int_{\Omega} \mu_s  \diff x \diff s + \int_{\Omega} \mu_0\, |x|^2 \diff x  = 2 \, d \, t + \int_{\Omega} \mu_0\, |x|^2 \diff x.
$$
If $\Omega$ is a bounded domain, there is $R$ such that $|x|\leq R$ on $\Omega$ so that $\int_{\Omega} \mu_t\, |x|^2 \diff x\leq R^2$. Plugging this into \eqref{eq:rate_nto1_estimate_Omega_2} we obtain the estimate on $\Omega_2$. Finally, on $\Omega_3$ we use that $\sqrt{y}|\log(y)|^2 \leq \frac{16}{e^2}$ for $y\in[0,1]$ so that
\begin{equation}\label{eq:estimate_mu_log2_rate_nto1_omega3}
\int_{\Omega_3} (\mu_t+\mu_t^n) \, |\log \mu_t|^2 \diff x \leq 2 \int_{\Omega_3} \mu_t \, |\log \mu_t|^2 \diff x \leq 
\frac{32}{e^2} \int_{\Omega} e^{-\frac{|x|}{2}} \diff x. 
\end{equation}  
Collecting estimates \eqref{eq:estimate_mu_log2_rate_nto1_omega1}, \eqref{eq:rate_nto1_estimate_Omega_2} and \eqref{eq:estimate_mu_log2_rate_nto1_omega3} and plugging them into \eqref{eq:estimate_G-F_interms_mulog2_rate_nto1} we obtain
$$
|\mathcal{G}[\mu_t]-\mathcal{F}[\mu_t]|\leq \frac{|n-1|}{2}\left(\frac{8\,\|\mu_0\|_{L^{\alpha}(\Omega)}^{\alpha}}{(\alpha-n)^2}  + 4 \, d \, t + 2 \int_{\Omega} \mu_0\, |x|^2 \diff x+  \frac{32}{e^2} \int_{\Omega} e^{-\frac{|x|}{2}} \diff x 
\right)
$$
if $\Omega =\Rd$ or $\Omega$ is a periodic domain and
$$
|\mathcal{G}[\mu_t]-\mathcal{F}[\mu_t]|\leq \frac{|n-1|}{2}\left(\frac{8\, \|\mu_0\|_{L^{\alpha}(\Omega)}^{\alpha}}{(\alpha-n)^2}  + 2\, R^2+  \frac{32}{e^2} \int_{\Omega} e^{-\frac{|x|}{2}} \diff x 
\right)
$$
if $\Omega$ is a bounded domain with $|x|\leq R$ on $\Omega$. Now, to estimate $|\mathcal{G}[\nu_t] - \mathcal{F}[\nu_t]|$, we observe that the only difference is the estimate on the second moment. On a bounded domain we argue in the same way, resulting in the same estimate for $|\mathcal{G}[\nu_t] - \mathcal{F}[\nu_t]|$ while on $\Rd$ or a~periodic domain we have
$$
\int_{\Omega} \nu_t \, |x|^2 \diff x = 2\,d \int_0^t \int_{\Omega} |\nu_s|^n \diff x \diff s + \int_{\Omega} \mu_0 \, |x|^2 \diff x \leq 2\,d\,t \int_{\Omega} |\mu_0|^n \diff x + \int_{\Omega} \mu_0 \, |x|^2 \diff x,
$$
where we used that the $L^n(\Omega)$ norm of $\mu_s$ does not increase in time. To remove the dependence on $n$, we estimate using H{\"o}lder inequality with respect to the measure $\mu_0(x)\diff x$
$$
\int_{\Omega} |\mu_0|^n \diff x \leq \left(\int_{\Omega} |\mu_0|^{\alpha}\diff x\right)^{\frac{n-1}{\alpha-1}} \leq 1 + \int_{\Omega} |\mu_0|^{\alpha}\diff x,
$$
where we also used that $\frac{n-1}{\alpha-1} \leq 1$. Therefore,
\begin{multline*}
|\mathcal{G}[\nu_t]-\mathcal{F}[\nu_t]|\leq \\ \leq \frac{|n-1|}{2}\left(\frac{8\,\|\mu_0\|_{L^{\alpha}(\Omega)}^{\alpha}}{(\alpha-n)^2}  + 4\,d\,t\, (1+\|\mu_0\|_{L^{\alpha}(\Omega)}^{\alpha}) + 2 \int_{\Omega} \mu_0 \, |x|^2 \diff x +  \frac{32}{e^2} \int_{\Omega} e^{-\frac{|x|}{2}} \diff x 
\right).
\end{multline*}
Plugging these estimates into \eqref{eq:rate_nto1_timederivativeofW2} and integrating in time we arrive at the claim. 
\end{proof}

\section{The quadratic porous medium equation}\label{app:PME_m2_and_nonlocal2local_facts}
We recall here estimates which are crucial for the proof of Theorem \ref{thm:rate_of_conv_nonlocal_to_local}.

\begin{lem}\label{lem:estimates_PME_1D}
Let $\Omega= \R$ or $\Omega =\mathbb{T}$. Let $\rho = \rho_t(x)$ with $t\in [0,\infty), x \in \Omega$ be a solution to 
$$
\partial_t \rho = \frac{1}{2} \partial^2_{x} \rho^2 = \partial_x(\rho \, \partial_x \rho)
$$
with the initial condition $\rho_0 \in L^1(\Omega)\cap L^{\infty}(\Omega)$, $\rho_0 \geq 0$.
\begin{enumerate}[label=(P\arabic*)]
\item\label{est:Lipschitz_pressure} If $\partial_x \rho_0 \in L^{\infty}(\Omega)$, then $\| \partial_x \rho_t \|_{L^{\infty}(\Omega)} \leq \| \partial_x \rho_0\|_{L^{\infty}(\Omega)}$. 
\item\label{est:laplacian_pressure} If the negative part of $\partial_{x}^2 \rho_0$ satisfies $|\partial_{x}^2 \rho_0|^{-} \in \mathcal{M}(\Omega)$, then $\partial^2_{x} \rho_t \in \mathcal{M}(\Omega)$ and
$$
\| \partial^2_{x} \rho_t \|_{\mathcal{M}(\Omega)}  \leq 2\,\| |\partial_{x}^2 \rho_0|^{-} \|_{\mathcal{M}(\Omega)},
$$
where $\mathcal{M}(\Omega)$ is the space of bounded Radon measures equipped with the total variation norm.
\item\label{est:weighted_pressure} Under assumptions of \ref{est:Lipschitz_pressure} and \ref{est:laplacian_pressure}, for all $T>0$, $\sqrt{\rho_t} \, \partial_{x}^2 \rho_t \in L^{2}((0,T)\times\Omega)$. More precisely,
$$
\int_0^T \int_{\Omega} \rho_t \, |\partial^2_x \rho_t|^2 \diff x \diff t \leq \Big(T\,\|\partial_x \rho_0\|^2_{L^{\infty}(\Omega)} + \frac{1}{2}\|\rho_0\|_{L^{\infty}(\Omega)}\Big) \, \| |\partial^2_x \rho_0|^- \|_{\mathcal{M}(\Omega)}.
$$
\end{enumerate}
\end{lem}

As remarked in \cite{MR2487898}, the estimate \ref{est:laplacian_pressure}, called the Aronson-B\'{e}nilan estimate \cite{MR524760}, is valid only on the whole space or on the periodic domain. We refer to \cite{bevilacqua2022aronson} for more variants of this inequality.

\begin{proof}[Proof of Lemma \ref{lem:estimates_PME_1D}]
All the manipulations in the proof are performed assuming $\rho$ is smooth and strictly positive. The general case can be established by a usual approximation. Estimate \ref{est:Lipschitz_pressure} follows directly from \cite[Proposition 15.4]{MR2286292}. Concerning \ref{est:laplacian_pressure}, we only explain how to modify the argument from \cite[Proposition 9.4]{MR2286292}. Direct computation shows that the function $w = \partial^2_x \rho$ satisfies
\begin{equation}\label{eq:AB_differential_ineq_before_integrating}
\partial_t w - 3\, w^2 -  \rho \, \partial^2_x w - 4\, \partial_x \rho \, \partial_x w = 0.
\end{equation}
We want to estimate $|w|^{-} = -w\, \mathds{1}_{w\leq0} \geq 0$. Multiplying \eqref{eq:AB_differential_ineq_before_integrating} by $-\mathds{1}_{w\leq0}$ yields
$$
\partial_t |w|^{-} +3 \, (|w|^-)^2 +  \rho \, \partial^2_x w\,\mathds{1}_{w\leq0} - 4\, \partial_x \rho \, \partial_x |w|^- = 0.
$$
By Kato's inequality \cite{MR2056467}, we have $ \partial^2_x w\,\mathds{1}_{w\leq0} \geq -  \partial^2_x |w|^-$ so that
$$
\partial_t |w|^{-} +3 \, (|w|^-)^2 - \rho \, \partial^2_x |w|^- - 4\, \partial_x \rho \, \partial_x |w|^- \leq 0.
$$
 Now, we integrate over $\Omega$. Since
 $$
 \int_{\Omega} \rho \, \partial^2_x |w|^- \diff x = - \int_{\Omega} \partial_x \rho \, \partial_x |w|^- \diff x = \int_{\Omega} \partial^2_x \rho \,  |w|^- \diff x = - \int_{\Omega} (|w|^-)^2 \diff x,
 $$
 we obtain $\partial_t \int_{\Omega} |w|^{-} \diff x \leq 0$ which gives
\begin{equation}\label{eq:negative_part_pressure} 
 \int_{\Omega} |\partial^2_x \rho_t|^{-} \diff x \leq \int_{\Omega} |\partial^2_x \rho_0|^{-} \diff x.
\end{equation}
By writing $|\partial_{x}^2 \rho_t | = \partial_{x}^2 \rho_t  + 2\, |\partial_{x}^2 \rho_t |^{-}$ and integrating in space we arrive at the claim.
 \\

\noindent Concerning \ref{est:weighted_pressure}, assuming $\rho$ is a smooth solution, we write
$$
\partial_t \rho = |\partial_x \rho|^2 + \rho \, \partial^2_x \rho.
$$ 
We multiply by $\partial^2_x \rho$ and integrate by parts to get
$$
\int_{\Omega} \rho \, |\partial^2_x \rho|^2 \diff x = - \int_{\Omega} |\partial_x \rho|^2 \, \partial^2_x \rho \diff x - \partial_t\,  \frac{1}{2} \int_{\Omega} |\partial_x \rho_t|^2 \diff x.
$$
Integrating in time and using the maximum principle \ref{est:Lipschitz_pressure} as well as \eqref{eq:negative_part_pressure} we obtain
\begin{align*}
&\int_0^t \int_{\Omega} \rho \, |\partial^2_x \rho|^2 \diff x \diff s \leq - \int_0^t \int_{\Omega} |\partial_x \rho|^2 \, \partial^2_x \rho \diff x \diff s +  \frac{1}{2} \int_{\Omega} |\partial_x \rho_0|^2 \diff x = \\ 
&- \int_0^t \int_{\Omega} |\partial_x \rho|^2  \partial^2_x \rho \diff x \diff s -\frac{1}{2} \int_{\Omega} \rho_0 \, \partial^2_x\rho_0 \diff x \leq \Big(t\,\|\partial_x \rho_0\|^2_{L^{\infty}(\Omega)} + \frac{1}{2}\|\rho_0\|_{L^{\infty}(\Omega)}\Big) \, \| |\partial^2_x \rho_0|^- \|_{\mathcal{M}(\Omega)}. 
\end{align*}
Note that since $|\partial_x\rho|^2$ and $\rho_0$ are nonnegative, we could use the information only on the negative part of $\partial^2_x \rho$ as in \eqref{eq:negative_part_pressure}. The general result follows by approximation.
\end{proof}

\section{Finite speed of propagation for aggregation-diffusion equations}\label{app:general_aggr_diff_PDE}

In this section we prove that solutions to 
\begin{equation}\label{eq:aggregation-diffusion-PDE-appendix}
\partial_t \rho = \Delta \rho^m + \DIV(\rho \nabla(V + W\ast \rho))
\end{equation}
with a bounded and compactly supported initial condition stay compactly supported for all times. The PDE is a 2-Wasserstein gradient flow of the energy
\begin{equation}\label{eq:energy_F_wass_gradient_flow_aggr_diff}
\mathcal{F}[\rho] = \frac{1}{m-1}  \int_{\Omega} \rho^m \diff x+\int_{\Omega}   V\, \rho \diff x + \frac{1}{2}\int_{\Omega}  W\ast\rho \, \rho \diff x.
\end{equation}
Note that the distributional solutions to \eqref{eq:aggregation-diffusion-PDE-appendix} with $\int_0^t \int_{\Omega} \rho_s\, \left|\nabla \frac{\delta \mathcal{F}}{\delta \rho}[\rho_s] \right|^2 \diff x \diff s<\infty$ are unique. Indeed, by Assumption \ref{ass:potential}, the functional $\mathcal{F}$ in \eqref{eq:energy_F_wass_gradient_flow_aggr_diff} is $\lambda$-geodesically convex. Therefore, the unique EVI solution corresponds to the subdifferential one \cite[Theorem 4.35]{MR3050280} which is equivalent to the distributional one \cite[Propositions 4.36-4.38]{MR3050280}.
 
\begin{thm}\label{thm:aggr-diff-compact-supp} Let $\rho$ be a distributional solution on $\Omega = \Rd$ to \eqref{eq:aggregation-diffusion-PDE-appendix} with compactly supported initial condition $\rho_0 \in \mathcal{P}_2(\Rd)\cap L^{\infty}(\Rd)$. We assume that $V, W:\Rd\to\R$ satisfy 
\ref{ass_item:growth_conditions_V_W_main_thm}, \ref{ass:hessian_of_potential} and $W(x)=W(-x)$. Then, there exists a function $R(t):[0,\infty)\to [0,\infty)$ such that $\rho_t$ is supported in the ball $B_{\sqrt{2R(t)}}$ for all $t>0$.
\end{thm}
\begin{proof}[Proof of Theorem \ref{thm:aggr-diff-compact-supp}] For the computations below, we assume that $\rho_t$ solving \eqref{eq:aggregation-diffusion-PDE-appendix} is smooth. The relevant technical details are discussed below in Remark \ref{rem:justification_smoothness_solutions_to_aggr-diff}.\\

\underline{Step 1: A priori estimates.} We first prove that
\begin{equation}
\label{eq:moment_estimate_second_lemma_finite_speed}
\left( \int_{\Omega} |x|^2 \, \rho_t(x)\diff x \right)^{1/2} \leq \left( \int_{\Omega} |x|^2 \, \rho_0(x)\diff x \right)^{1/2} + \sqrt{t}\, \left(\mathcal{F}[\rho_0]+ C^0_{V} + C^0_W \right)^{1/2},
\end{equation} 
\begin{equation}\label{eq:estimate:gradient_advection}
|\nabla V + \nabla W\ast \rho_t|(x) \leq C^3_{V} + C^3_{W} +   (C^4_V+C^4_{W}) \,  |x| + C^4_W\, \left( \int_{\Omega} |y|^2  \rho_t(y)\diff y \right)^{1/2},
\end{equation}
where the constants come from conditions \ref{ass_item:growth_conditions_V_W_main_thm}, \ref{ass:hessian_of_potential}. To this end, we note that \eqref{eq:aggregation-diffusion-PDE-appendix} can be written as $\partial_t \rho_t = \DIV(\rho_t \, \nabla \frac{\delta \mathcal{F}}{\delta \rho}[\rho_t])$ where $\mathcal{F}$ is defined in \eqref{eq:energy_F_wass_gradient_flow_aggr_diff}. Multiplying \eqref{eq:aggregation-diffusion-PDE-appendix} by $\frac{\delta \mathcal{F}}{\delta \rho}[\rho_t]$ and integrating by parts, we obtain for a.e. $t$
$$
\mathcal{F}[\rho_t] + \int_0^t \int_{\Omega} \rho_s\, \left|\nabla \frac{\delta \mathcal{F}}{\delta \rho}[\rho_s] \right|^2 \diff x \diff s \leq \mathcal{F}[\rho_0].
$$
We note that \ref{ass_item:growth_conditions_V_W_main_thm} implies $\mathcal{F}[\rho_t]  \geq -C^0_{V}-C^0_{W}$ by the conservation of mass so that
\begin{equation}\label{eq:bound_on_the_diss_finite_speed_prop_aggr-diff}
\int_0^t \int_{\Omega} \rho_s\, \left|\nabla \frac{\delta \mathcal{F}}{\delta \rho}[\rho_s] \right|^2 \diff x \diff s \leq \mathcal{F}[\rho_0] + C^0_{V} + C^0_{W}.
\end{equation}
To prove \eqref{eq:moment_estimate_second_lemma_finite_speed}, we compute
$$
\frac{\diff}{\diff t} \int_{\Omega} |x|^2 \, \rho_t \diff x = -2 \int_{\Omega} x\, \rho_t \, \nabla \frac{\delta F}{\delta \rho}[\rho_t] \diff x \leq 2\left(  \int_{\Omega} |x|^2 \, \rho_t \diff x  \right)^{1/2} \, \left( \int_{\Omega} \rho_t  \left|  \nabla \frac{\delta F}{\delta \rho}[\rho_t] \right|^2 \diff x \right)^{1/2}.
$$
It follows that 
$$
\frac{\diff}{\diff t} \left(\int_{\Omega} |x|^2 \, \rho_t \diff x\right)^{1/2} \leq 
\left( \int_{\Omega} \rho_t  \left|  \nabla \frac{\delta F}{\delta \rho}[\rho_t] \right|^2 \diff x \right)^{1/2}.
$$
Integrating in time and using \eqref{eq:bound_on_the_diss_finite_speed_prop_aggr-diff} we arrive at \eqref{eq:moment_estimate_second_lemma_finite_speed}. To prove \eqref{eq:estimate:gradient_advection}, we compute using condition \ref{ass_item:growth_conditions_V_W_main_thm} for both $V$ and $W$
\begin{equation*}
\begin{split}
|\nabla V + \nabla W\ast \rho_t|(x) &\leq C^3_{V}  +  C^4_V \, |x| + \left| \int_{\Omega} \nabla W(x-y) \, \rho_t(y) \diff y\right|  \\
&\leq C^3_{V} + C^3_{W} +  (C^4_V+C^4_{W}) \, |x|  + C^4_W \int_{\Omega} |y|\, \rho_t(y) \diff y\\
&\leq C^3_{V} + C^3_{W} +   (C^4_V+C^4_{W}) \, |x| + C^4_W\, \left( \int_{\Omega} |y|^2 \, \rho_t(y)\diff y \right)^{1/2}.
\end{split}
\end{equation*} 
\underline{Step 2: Construction of $R(t)$.} Let
$
v_t(x) := C (\left| R(t) - |x|^2/2 \right|^+)^{\frac{1}{m-1}},
$
where $|x|^+ = \max(x,0)$. We will find $C>0$ and $R(t)$ so that $v$ satisfies
\begin{equation}\label{eq:v_supersolution_aggr_diffusion}
\partial_t v \geq \Delta v^m +\DIV(v\, \nabla(V + W\ast \rho_t)) \mbox{ a.e. on } (0,\infty)\times\Rd. 
\end{equation}
Direct computation of $\Delta v^m$, $\partial_t v$, $\nabla v$ shows that
$$
\nabla v^m = - \frac{C^{m-1} \,m}{m-1}\, v \, x, \qquad
\Delta v^m = -\frac{C^{m-1} \, d\,m}{m-1}\, v +  \frac{C^{2m-2} \,m}{(m-1)^2}\, v^{2-m} \, \mathds{1}_{R(t) \geq |x|^2/2}  \, |x|^2,
$$
$$
\partial_t v = \frac{C^{m-1}}{m-1} \, v^{2-m} \,  \mathds{1}_{R(t) \geq |x|^2/2} \, R'(t),\qquad 
\nabla v = -\frac{C^{m-1}}{m-1} \, v^{2-m} \,  \mathds{1}_{R(t) \geq |x|^2/2} \, x.
$$
We now observe that all the terms appearing in \eqref{eq:v_supersolution_aggr_diffusion} are in fact in $L^{\infty}(0,T; L^1(\Rd))$ for all $T>0$ (and so, they can be understood pointwisely) whenever $R(t), \frac{1}{R(t)} \in L^{\infty}(0,T)$. Indeed, it suffices to establish integrability of $v^{2-m}$. To this end, we let $\vartheta = \frac{2-m}{m-1}$ and we compute
\begin{multline*}
\int_{\Rd} \left(\left| R(t) - |x|^2/2 \right|^+\right)^{\vartheta} \diff x \leq \widetilde{C} \, \int_{\Rd} \left(\big| \sqrt{2\,R(t)} - |x| \big|^+\right)^{\vartheta} \diff x \leq\\ \leq \widetilde{C} \, \int_{0}^{\sqrt{2\,R(t)}} (\sqrt{2\,R(t)} - r)^{\vartheta} \, r^{d-1} \diff r = \widetilde{C} \, R(t)^{\frac{d+\vartheta}{2}} \int_0^1 (1-s)^{\vartheta}\, s^{d-1} \diff s,
\end{multline*}
where $\widetilde{C}$ changes from line to line and depends on $\|R(t)\|_{L^{\infty}(0,T)}$, $\|1/R(t)\|_{L^{\infty}(0,T)}$, $\vartheta$ and $d$. The last integral is finite because $\vartheta + 1 = \frac{1}{m-1} > 0$.\\

Using $|x|^2 \leq 2\,R(t)$ and the explicit formulas above we obtain
\begin{equation}\label{eq:finite_speed_prop_pure_porous_medium_part}
\partial_t v - \Delta v^m \geq \frac{C^{m-1}}{m-1}\, v^{2-m} \, \mathds{1}_{R(t) \geq |x|^2/2}  \left[ R'(t) \!-\! C^{m-1} \, \frac{2\,m}{m-1}  \, R(t) \right]  + C^{m-1} \, \frac{d\,m}{m-1}\, v.
\end{equation}
Next, we estimate the terms corresponding to advection and aggregation:
\begin{equation}\label{eq:eq:eq:finite_speed_prop_advection_diffusion_1}
-v\, \Delta(V + W\ast \rho_t) \geq -v \, \left(\| \Delta V\|_{L^{\infty}(\Omega)} + \| \Delta W\|_{L^{\infty}(\Omega)} \right),
\end{equation}
\begin{equation}\label{eq:eq:eq:finite_speed_prop_advection_diffusion_2}
-\nabla v\, \nabla(V + W\ast \rho_t) \geq -\frac{C^{m-1}}{m-1} \, v^{2-m} \,  \mathds{1}_{R(t) \geq |x|^2/2} \, |x| \, |\nabla V + \nabla W\ast \rho_t|.
\end{equation}
We estimate the second term directly with \eqref{eq:estimate:gradient_advection} and then \eqref{eq:moment_estimate_second_lemma_finite_speed}  
\begin{equation}\label{eq:eq:finite_speed_prop_more_advection_diffusion}
\begin{split}
- &\mathds{1}_{R(t) \geq |x|^2/2} \, |x| \, |\nabla V + \nabla W\ast \rho_t| \geq \\ 
&\, \geq -\mathds{1}_{R(t) \geq |x|^2/2}\, \sqrt{2R(t)}\,\left(\! C^3_{V} \!+ \!C^3_{W} \!+ \!  (C^4_V+C^4_{W})   |x| \!+ \!C^4_W\left( \int_{\Omega} |y|^2  \rho_t(y)\diff y \right)^{1/2}\right)\\
&\, \geq -\mathds{1}_{R(t) \geq |x|^2/2}\, \left(2\,R(t)\, (C^4_V+C^4_{W}) + \sqrt{2R(t)}\, B(t) \right),
\end{split}
\end{equation}
where $B(t)$ is defined via 
\begin{equation}\label{eq:parameter_B(t)_to_define_the_support}
{B}(t):= C^3_{V} + C^3_{W}  + C^4_{W}\left( \int_{\Omega} |x|^2 \, \rho_0(x)\diff x \right)^{1/2}+ C^4_{W}\,\sqrt{t} \left(\mathcal{F}[\rho_0]+ C^0_{V} + C^0_W\right)^{1/2}. 
\end{equation}
Combining \eqref{eq:finite_speed_prop_pure_porous_medium_part}, \eqref{eq:eq:eq:finite_speed_prop_advection_diffusion_1}, \eqref{eq:eq:eq:finite_speed_prop_advection_diffusion_2} and \eqref{eq:eq:finite_speed_prop_more_advection_diffusion}
\begin{equation}\label{eq:full_PDE_for_comparison_before_choosing_R_and_C}
\begin{split}
&\partial_t v \! -\! \Delta v^m \!-\! \DIV(v\,\nabla(V + W\ast \rho_t)) \geq (C^{m-1} \, \frac{d\,m}{m-1} \! -\!  \| \Delta V\|_{L^{\infty}(\Omega)} \! - \! \| \Delta W\|_{L^{\infty}(\Omega)}) \,v \\
&+\!\frac{C^{m-1}}{m\!-\!1}\, v^{2-m} \, \mathds{1}_{R(t) \geq |x|^2/2}  \left[ R'(t) \!-\! \left(\frac{2C^{m-1}m}{m-1} \!+\!2(C^4_V +C^4_W)\!\right)\!R(t)  \!-\! B(t)  \sqrt{2R(t)} \right].
\end{split}
\end{equation}
We will choose $C$ such that 
\begin{equation}\label{eq:constraint_on_C_laplacian_V_W}
C^{m-1} \, \frac{d\,m}{m-1} \geq  \| \Delta V\|_{L^{\infty}(\Omega)} + \| \Delta W\|_{L^{\infty}(\Omega)},
\end{equation}
so that the first term on the (RHS) is nonnegative and we define $R(t)$ via ODE
\begin{equation}\label{eq:ode_R(t)_to_define_the_support}
R'(t) =  \left( \frac{2C^{m-1}m}{m-1}   +2\,(C^4_V +C^4_W) +  B(t)\right)\, R(t)  + B(t),
\end{equation}
with $R(0)$ to be chosen. Thanks to the inequality $-\sqrt{2 R(t)} \geq -R(t) -1$, we see that the second term on the (RHS) of \eqref{eq:full_PDE_for_comparison_before_choosing_R_and_C} vanishes, concluding the proof of \eqref{eq:v_supersolution_aggr_diffusion}.\\

\underline{Step 3: Conclusion.} We will first prove 
that the negative part $|v_t-\rho_t|^{-} = -(v_t-\rho_t)\, \mathds{1}_{v_t-\rho_t<0}$ satisfies
\begin{equation}\label{eq:comparison_principle_L1_with_supersolution_v}
\int_{\Omega} |v_t-\rho_t|^{-} \diff x \leq \int_{\Omega} |v_0-\rho_0|^{-} \diff x.
\end{equation}
This is a standard comparison principle for the porous medium equation. Indeed, using that
$$
\partial_t (v-\rho) \geq \Delta (v^m-\rho^m) +\DIV((v-\rho)\, \nabla(V + W\ast \rho)),
$$
we obtain \eqref{eq:comparison_principle_L1_with_supersolution_v} by multiplying this identity by $-\mathds{1}_{v-\rho<0}$ and integrating by parts as in \cite[Prop. 3.5]{MR2286292}. Now, note that in the definition of $
v_t(x) := C (\left| R(t) - |x|^2/2 \right|^+)^{\frac{1}{m-1}}$, we can still choose $C$ (subject to the constraint \eqref{eq:constraint_on_C_laplacian_V_W}) and $R(0)$. We choose them sufficiently large so that $v_0 \geq \rho_0$ (which is possible since $\rho_0$ is bounded and compactly supported). From \eqref{eq:comparison_principle_L1_with_supersolution_v} we deduce that $\rho_t \leq v_t$ meaning that $\rho_t$ is supported in $B_{\sqrt{2R(t)}}$ as desired.
\end{proof}
\begin{rem}\label{rem:justification_smoothness_solutions_to_aggr-diff}
The computation above requires some regularity of $\rho_t$. To justify the computation, we first embed the problem into a large periodic domain  $R\,\Td$ where $R$ is large. To this end, we modify the potentials $V$, $W$ by defining
$$
V_P= V\,\zeta_P, \qquad \qquad W_{P} = W\,\zeta_P,
$$
where $P$ is a parameter such that $P>1$, $2P<R$ and $\zeta_{P}:\Rd \to [0,1]$ is a cutoff function such that $\zeta_{P} = 1$ on $B_{P}$, $\zeta_P$ is supported in $B_{2P}$, $\zeta_{P}(x)=\zeta_{P}(-x)$, $|\nabla \zeta_P|\leq \frac{C_{\zeta}}{P}$, $|\nabla^2 \zeta_P| \leq \frac{C_{\zeta}}{P^2}$ for some constant $C_{\zeta}$. Since $V_P$ and $W_P$ are supported on $B_R$, they can be periodically extended to $R\Td$. We note two things.
\begin{itemize}[leftmargin=1cm]
\item The cutoff function $\zeta_P$ can be obtained as follows. If $\{\psi_{\delta}\}_{\delta}$ is a usual mollifier, we define $\zeta_{P} = \mathds{1}_{B_{\frac{3P}{2}}} \ast \psi_{\frac{P}{2}}$. Then, $\zeta_{P}=1$ on $B_{P}$, $|\zeta_P|\leq 1$ and $\zeta_P$ is supported on $B_{2P}$. Moreover, by Young's inequality
\begin{align*}
|\nabla \zeta_P| &\leq \|\nabla \psi_{\frac{P}{2}}\|_{L^1(\Rd)}\leq \frac{2}{P}\,\| \nabla \psi\|_{L^1(\Rd)},\\
|\nabla^2 \zeta_P| &\leq \|\nabla^2 \psi_{\frac{P}{2}}\|_{L^1(\Rd)}\leq \frac{4}{P^2}\,\| \nabla^2 \psi\|_{L^1(\Rd)},
\end{align*}
so that $C_{\zeta}:= \max(2\| \nabla \psi\|_{L^1(\Rd)},4\|\nabla^2\psi\|_{L^1(\Rd)})$. 
\item Second, $V_P$ and $W_P$ satisfy \ref{ass_item:growth_conditions_V_W_main_thm}, \ref{ass:hessian_of_potential} and $W_P(x)=W_P(-x)$. Indeed, it is sufficient to estimate $\nabla V_P$, $\nabla^2V_P$ (the same computation works for $\nabla W_P$, $\nabla^2W_P$). We have
\begin{equation}\label{eq:estimates_hessian_VP_app:finite_speed}
\begin{split}
|\nabla^2 V_{P}| &\leq |\nabla^2 V|\, \zeta_P +2 \,|\nabla V \otimes \nabla \zeta_P| + |V\, \nabla^2\zeta_P| \\
&\leq |\nabla^2 V| + 2\, (C^3_{V} + C^4_{V}|x|) \, \frac{C_{\zeta}}{P} \, \mathds{1}_{|x|\leq 2P} + (C^1_{V} + C^2_{V}|x|^2) \, \frac{C_{\zeta}}{P^2} \, \mathds{1}_{|x|\leq 2P}\\
&\leq |\nabla^2 V| + 4\,(C^2_{V}+C^4_{V})\,C_{\zeta} + \frac{C^1_{V} C_{\zeta}}{P^2} + \frac{2\,C^3_{V} C_{\zeta}}{P},
\end{split}
\end{equation}
\begin{equation}\label{eq:estimates_gradient_VP_app:finite_speed}
\begin{split}
&|\nabla V_P| \leq |\nabla V \, \zeta_P| + |V\, \nabla \zeta_P| \leq |\nabla V| + (C^1_{V} +C^2_V |x|^2 ) \, \mathds{1}_{|x|\leq 2P} \,\frac{C_{\zeta}}{P}\\
& \leq C^3_V +C^4_V \, |x| +C^1_V \, C_{\zeta} + 2\,C^2_V\, C_{\zeta}\, |x| = (C^1_V\, C_{\zeta} + C^3_V) + (2\,C^2_V\, C_{\zeta} + C^4_V)\, |x|.
\end{split}
\end{equation}
\end{itemize} 
Next, we consider initial condition $\rho_0^{\eps} = (1-\eps)\,\rho_0 +\frac{\eps}{|R\Td|}$ with $\eps \in (0,1)$ and we let $\rho_t^{\eps}$ to be the solution to \begin{equation}\label{eq:PDE_truncated_potentials_lifted_initial_condition}
\partial_t \rho^{\eps} = \Delta (\rho^{\eps})^m + \DIV(\rho^{\eps} \nabla(V_P + W_P\ast \rho^{\eps}))
\end{equation}
on $R\,\Td$ with initial condition $\rho_0^{\eps}$. The comparison principle as in \cite[Lemma~5.2]{MR4842678} implies
$$
\rho_t^{\eps} \geq \frac{\eps}{|R\Td|}\, \exp(-t\,(\|V_P \|_{L^{\infty}(R\Td)} +\|W_P \|_{L^{\infty}(R\Td)} )) >0,
$$
\begin{equation}\label{eq:Linfinity_bound_appendix_approximation_via_torus_to_justify_the_computations}
\rho_t^{\eps} \leq \left((1-\eps)\,\|\rho_0\|_{L^{\infty}(R\Td)} + \frac{\eps}{|R\Td|} \right)\, \exp(t\,(\|V_P \|_{L^{\infty}(R\Td)} +\|W_P \|_{L^{\infty}(R\Td)} )).
\end{equation}
It follows that $\rho_t^{\eps}$ is sufficiently regular to perform the computations in the proof of Theorem~\ref{thm:aggr-diff-compact-supp} and we obtain as in \eqref{eq:comparison_principle_L1_with_supersolution_v}
\begin{equation}\label{eq:conclusio_compact_support_for_shifted_initial_conditions}
\int_{\Omega} |v^{\eps}_t-\rho^{\eps}_t|^{-} \diff x \leq \int_{\Omega} |v^{\eps}_0-\rho^{\eps}_0|^{-} \diff x,
\end{equation}
where $v^{\eps}_t =  C (\left| R^{\eps}(t) - |x|^2/2 \right|^+)^{\frac{1}{m-1}}$ and $R^{\eps}(t)$ is defined by \eqref{eq:ode_R(t)_to_define_the_support}, with $\rho_0$ replaced by $\rho_0^{\eps}$ in the definition of $B(t)$ in \eqref{eq:parameter_B(t)_to_define_the_support}. Now, we want to pass to the limit $\eps \to 0$ both in the PDE \eqref{eq:PDE_truncated_potentials_lifted_initial_condition} and in \eqref{eq:conclusio_compact_support_for_shifted_initial_conditions}. The only difficulty is posed by the nonlinear term $\Delta (\rho^\eps)^m$. To overcome it, we observe two a priori estimates. 
\begin{itemize}[leftmargin=1cm]
\item First, \eqref{eq:bound_on_the_diss_finite_speed_prop_aggr-diff} and Lemma \ref{lem:B-B} imply that $\mathcal{W}_2^2(\rho^\eps_t, \rho^{\eps}_s) \leq |t-s| \, (\mathcal{F}[\rho^\eps_0]+C^0_V+C^0_W)$. Then, a~variant of Arzela-Ascoli theorem \cite[Prop. 3.3.1]{MR2129498} implies that $\rho^{\eps}_t \to \rho_t$ for all $t$ narrowly. 
\item Second, \eqref{eq:moment_estimate_second_lemma_finite_speed}, \eqref{eq:estimate:gradient_advection}, and \eqref{eq:bound_on_the_diss_finite_speed_prop_aggr-diff} provide a uniform in $\eps$ estimate on $\{\nabla(\rho^{\eps}_t)^{m-\frac{1}{2}}\}$ in $L^2((0,T)\times R\Td)$. We choose a countable sequence $\eps_k \to 0$ such that $\rho^{\eps_k}_t \to \rho_t$ for all~$t$ narrowly and we let ${A} := \{ t \in [0,T]: \nabla(\rho^{\eps_k}_t)^{m-\frac{1}{2}} \in L^2(R\Td) \mbox{ for all } k\}$ which is a set of full measure. Fix $t \in {A}$. We can extract a further subsequence along which $\rho^{{\eps_k}_l}_t$ converges a.e. and, by the dominated convergence theorem together with \eqref{eq:Linfinity_bound_appendix_approximation_via_torus_to_justify_the_computations}, also in $L^1(R\Td)$. By the narrow convergence, the limit has to be $\rho_t$ so standard subsequence argument shows that $\rho_t^{\eps_k} \to \rho_t$ strongly in $L^1(R\Td)$ for a.e. $t$ (all $t\in A$). This is sufficient to pass to the limit in the nonlinear term.
\end{itemize} 

Passing to the limit $\eps \to 0$ in \eqref{eq:conclusio_compact_support_for_shifted_initial_conditions} (the convergence of $v^{\eps}_t$ is clear), we deduce that $\rho_t$ is compactly supported on a large torus $R\,\Td$ by choosing $C$ and $R(0)$ as in Step 3 of the proof above. We note that the bounds \eqref{eq:estimates_hessian_VP_app:finite_speed}--\eqref{eq:estimates_gradient_VP_app:finite_speed} do not depend on $P$ when $P\to\infty$, hence we can take $P$ (and $R$) sufficiently large (comparing to the support of $\rho_t$) so that
$$
\rho\, \nabla V_P = \rho \, \nabla V, \qquad \rho \nabla W_{P} \ast \rho = \rho  \nabla W \ast \rho.
$$
Hence, $\rho_t$ is the distributional solution on $\Rd$ which is known to be unique.
\end{rem}

\bibliographystyle{abbrv}
\bibliography{fastlimit}
\end{document}